\documentclass[12pt]{amsart}

\usepackage{etex}

\usepackage{mystyle}

\usepackage[pdf]{pstricks}
\usepackage{pst-plot}
\usepackage{amsmath,amsthm,amssymb,amsrefs}
\usepackage{lscape}

\newcommand{\ol}{\overline}
\newcommand{\D}{\Delta}

\newrgbcolor{cldotcolor}{0.5 0.5 0.5}
\newrgbcolor{clhzerocolor}{0.5 0.5 0.5}
\newrgbcolor{clhonecolor}{0.5 0.5 0.5}
\newrgbcolor{clhtwocolor}{0.5 0.5 0.5}

\newrgbcolor{Ctauhiddenhzerocolor}{0 0 1}
\newrgbcolor{Ctauhiddenhonecolor}{0 0 1}
\newrgbcolor{Ctauhiddenhtwocolor}{0 0 1}

\newrgbcolor{taubottomcolor}{0.5 0.5 0.5}
\newrgbcolor{tautopcolor}{1 0 0}

\newrgbcolor{hzerotowercolor}{0.5 0.5 0.5}

\newrgbcolor{cldtwocolor}{0 0.7 0.7}
\newrgbcolor{cldthreecolor}{1 0 0}
\newrgbcolor{cldfourcolor}{0.1 0.7 0.1}

\newrgbcolor{dtwocolor}{0 0.7 0.7}
\newrgbcolor{h1dtwocolor}{1 0.5 0}
\newrgbcolor{dthreecolor}{0 0.7 0.7}
\newrgbcolor{dfourcolor}{0.1 0.7 0.1}

\newgray{gridline}{0.85}

\newcommand{\cirrad}{0.07}

\title[The special fiber of the motivic deformation is algebraic]{The special fiber of the motivic deformation of the stable homotopy category is algebraic}

\author{Bogdan Gheorghe}
\address{Max-Planck-Institut f\"ur Mathematik\\
53111 Bonn, Germany}
\email{gheorghebg@mpim-bonn.mpg.de}

\author{Guozhen Wang}
\address{Shanghai Center for Mathematical Sciences, Fudan University, Shanghai, China, 200433}
\email{wangguozhen@fudan.edu.cn}

\author{Zhouli Xu}
\address{Department of Mathematics, Massachusetts Institute of Technology, Cambridge, MA 02139}
\email{xuzhouli@mit.edu}

\begin{document}

\begin{abstract}
For each prime $p$, we define a $t$-structure on the category $\widehat{S^{0,0}}/\tau\text{-}\mathbf{Mod}_{\textup{harm}}^b$ of harmonic $\mathbb{C}$-motivic left module spectra over $\widehat{S^{0,0}}/\tau$, whose MGL-homology has bounded Chow-Novikov degree, such that its heart is equivalent to the abelian category of $p$-completed $\textup{BP}_*\textup{BP}$-comodules that are concentrated in even degrees. We prove that $\widehat{S^{0,0}}/\tau\text{-}\mathbf{Mod}_{\textup{harm}}^b$ is equivalent to $\mathcal{D}^b({\textup{BP}_*\textup{BP}\text{-}\mathbf{Comod}}^{\textup{ev}})$ as stable $\infty$-categories equipped with $t$-structures.

As an application, for each prime $p$, we prove that the motivic Adams spectral sequence for $\widehat{S^{0,0}}/\tau$, which converges to the motivic homotopy groups of $\widehat{S^{0,0}}/\tau$, is isomorphic to the algebraic Novikov spectral sequence, which converges to the classical Adams-Novikov $E_2$-page for the sphere spectrum $\widehat{S^0}$. This isomorphism of spectral sequences allows Isaksen and the second and third authors to compute the stable homotopy groups of spheres at least to the 90-stem, with ongoing computations into even higher dimensions.
\end{abstract}

\maketitle

\setcounter{tocdepth}{2}
\tableofcontents


\section{Introduction}
Motivic homotopy theory, introduced by Voevodsky and Morel \cite{MorelVoevodsky, Voe, Voered, VoeEM, VoeZ2, MorelA1, MorelAdams}, is a successful application of abstract homotopy theory to solve problems in number theory and algebraic geometry (see \cite{VoeMC1, VoeMC2, VoeBK} for example).

Over $\Spec \C$, one may view the $p$-completed stable motivic homotopy category as a deformation of the $p$-completed classical stable homotopy category. The parameter of the deformation is given by an element $\tau$ in $\pi_{0, -1}$ of the $p$-completed motivic sphere spectrum, which can be intuitively viewed as the standard coordinate $t\mapsto e^{2\pi it}$ on $\mathbb{G}_m$. Formally speaking, following Hu-Kriz-Ormsby \cite{HKOrem}, the element $\tau$ is the inverse limit of the Bockstein pre-images of the Morel classes \cite{Morelpi0} of roots of unity. Dugger-Isaksen \cite{DuggerIsaksenMASS} have identified the generic fiber ``$\tau = 1$" with the classical stable homotopy category, and the first main result of this paper identifies the special fiber ``$\tau = 0$" with the derived category of $\textup{BP}_*\textup{BP}$-comodules that are concentrated in even degrees, which is entirely algebraic in nature. Moreover, under this identification, the motivic Adams-Novikov spectral sequence for the motivic sphere spectrum corresponds to the $\tau$-Bockstein spectral sequence. This deformation induces a deformation of motivic Adams spectral sequences. The second main result of this paper identifies the motivic Adams spectral sequence for the motivic sphere spectrum at the special fiber ``$\tau = 0$" with the algebraic Novikov spectral sequence for the classical sphere spectrum, which is again entirely algebraic. This deformation makes it possible for Isaksen, the second and third authors \cite{IWX} to compute classical stable homotopy groups of spheres at least to the 90-stem, with ongoing computations into even higher dimensions.

\subsection{Main results}

In this paper, we prove two results in the stable motivic homotopy theory over $\Spec \C$, with connections to chromatic homotopy theory and applications to classical homotopy theory. 

The first result identifies the special fiber ``$\tau = 0$" of the motivic deformation with the derived category of $\textup{BP}_*\textup{BP}$-comodules. We prove an $\infty$-category version of a conjecture due to the first author and Isaksen in 2016 \cite{IsaP}. The derived category in the following Theorem 1.1 is understood as a stable $\infty$-category in the sense of Lurie in Higher Algebra \cite[Section 1.3.2]{HA}. 

\begin{thm}[Theorem \ref{thm:ctaumod}]\label{intro:mainthmpart1}
There is an equivalence of stable $\infty$-categories equipped with $t$-structures at each prime $p$
$$\mathcal{D}^b({\textup{BP}_*\textup{BP}\text{-}\mathbf{Comod}}^{\textup{ev}}) \simeq \widehat{S^{0,0}}/\tau\text{-}\mathbf{Mod}_{\textup{harm}}^b,$$
between the bounded derived category of $p$-completed $\textup{BP}_*\textup{BP}$-comodules that are concentrated in even degrees,
and the category of harmonic motivic left module spectra over $\widehat{S^{0,0}}/\tau$, whose MGL-homology has bounded Chow-Novikov degree, with morphisms the $\widehat{S^{0,0}}/\tau$-linear maps.
\end{thm}

Here $\widehat{S^{0,0}}/\tau$ is a motivic $E_\infty$-ring spectrum, which is also known as the cofiber of $\tau$. The motivic spectrum MGL is the algebraic cobordism spectrum introduced by Voevodsky \cite{Voe} and studied by Levine-Morel \cite{ML}, Panin-Pimenov-R\"ondigs \cite{PPR2} and many others. A motivic left module spectrum over $\widehat{S^{0,0}}/\tau$ is \emph{harmonic}, if it is $\widehat{S^{0,0}}/\tau$-cellular and the map to its MGL-completion induces an isomorphism on $\pi_{*,*}$. See Definition \ref{def:pisscomplete} for a precise definition. The Chow-Novikov degree is the topological degree minus twice the motivic weight. 

The derived category of $p$-completed $\textup{BP}_*\textup{BP}$-comodules that are concentrated in even degrees is also known as the derived category of quasi-coherent sheaves on the moduli stack of formal groups over $\mathbb{Z}_p$-algebras. This connection is foundational to chromatic homotopy theory, and is due to Quillen \cite{Qui} and Morava \cite{Mor} (see also Goerss-Hopkins \cite{Go, Hop}). Our theorem further connects these categories to motivic homotopy theory.

The equivalence of stable $\infty$-categories in Theorem~\ref{intro:mainthmpart1} is actually symmetric monoidal. See Remark~\ref{symmo} for more details.

By an $\textup{Ind}$-object argument, we have an unbounded version of Theorem \ref{intro:mainthmpart1} that connects to Hovey's \cite{Hovey} derived category $\mathbf{Stable}(\textup{BP}_*\textup{BP})$. Since every $\textup{BP}_*\textup{BP}$-comodule splits as its even graded part and its odd graded part, the underlying stable $\infty$-category of Hovey's unbounded derived category $\mathbf{Stable}(\textup{BP}_*\textup{BP})$ splits accordingly.
$$\mathbf{Stable}(\textup{BP}_*\textup{BP}) \simeq \mathbf{Stable}(\textup{BP}_*\textup{BP}\text{-}\mathbf{Comod}^\textup{ev}) \times \mathbf{Stable}(\textup{BP}_*\textup{BP}\text{-}\mathbf{Comod}^\textup{odd}).$$
 
 
 
\begin{cor}
There is an equivalence of stable $\infty$-categories at each prime $p$
$$\mathbf{Stable}(\textup{BP}_*\textup{BP}\text{-}\mathbf{Comod}^\textup{ev})\simeq \widehat{S^{0,0}}/\tau\text{-}\mathbf{Mod}_\textup{cell}$$
between Hovey's unbounded derived category of $\textup{BP}_*\textup{BP}$-comodules that are concentrated in even degrees and the category of cellular motivic left module spectra over $\widehat{S^{0,0}}/\tau$.
\end{cor}

After the announcement of Theorem 1.1, alternative proofs of certain versions of Corollary 1.2 have appeared in work of Krause \cite{Kra} and Pstr$\text{\c a}$gowski \cite{Pst}.
~\\

The second result identifies the motivic Adams spectral sequence at the special fiber ``$\tau = 0$" with the algebraic Novikov spectral sequence. It can be used to systematically compute a huge number of classical Adams differentials that are hard to obtain by other methods.

It is known to Isaksen \cite[Proposition 6.2.5]{StableStems} and the first author \cite[Corollary 3.14]{GheCt} that there is an isomorphism between the motivic homotopy groups of $\widehat{S^{0,0}}/\tau$ and the classical Adams-Novikov $E_2$-page. Our second result shows that there is an isomorphism of spectral sequences that converge to them.

\begin{thm}[Theorem \ref{thm:iso13}] \label{intro:thmequivss}
For each prime $p$, there is an isomorphism of spectral sequences between the motivic Adams spectral sequence for $\widehat{S^{0,0}}/\tau$ and the algebraic Novikov spectral sequence for the classical sphere spectrum $\widehat{S^0}$. 
\end{thm}

Based on Theorem \ref{intro:thmequivss}, Isaksen, the second and third authors \cite{IWX} have computed classical stable stems at least to the 90-stem, with ongoing computations into even higher dimensions. Computations of many historically difficult differentials in the range up to the 45-stem are included in the appendix.

In contrast to the original motivations of motivic homotopy theory, Isaksen and his collaborators \cite{StableStems, IsaksenCharts, IsaksenCharts2, IX} have recently begun to reverse the information flow and applied stable motivic homotopy theory to obtain computational results in the classical stable homotopy theory. Our Theorems \ref{intro:mainthmpart1} and \ref{intro:thmequivss} have the same spirit and further deepen the connections to chromatic homotopy theory. Using motivic homotopy theory, we build up a new connection between the classical Adams spectral sequence and the Adams-Novikov spectral sequence, that allows us to compute stable stems in a much larger range than was previously possible.

\begin{remark}
	Our Theorems \ref{intro:mainthmpart1} and \ref{intro:thmequivss} actually hold for any algebraically closed field of characteristic 0. In fact, it was clear in Dugger-Isaksen \cite{DuggerIsaksenMASS} that all related computations in algebraic closed field of characteristic 0 work in the same way as over $\mathbb{C}$, which are based on Voevodsky's computation of the motivic Steenrod algebra (see \cite[Theorem~4.47]{VoeEM} and \cite[Section 11]{Voered}).
\end{remark}

\subsection{The stable $\infty$-category of motivic spectra over $\widehat{S^{0,0}}/\tau$}

We work with the stable $\infty$-category of motivic spectra over $\Spec \C$, denoted by
$$\C\text{-}\mathbf{mot}\text{-}\mathbf{Spectra}.$$
This is a symmetric monoidal $\infty$-category in the sense of Lurie \cite[Section 2.1.2]{HA}.
 
There are several approaches for the construction of this category $\C\text{-}\mathbf{mot}\text{-}\mathbf{Spectra}$. For example, we can take the category of $\mathcal{S}$-modules constructed by Hu \cite{Hu03}, which is a symmetric monoidal model category, and we take $\C\text{-}\mathbf{mot}\text{-}\mathbf{Spectra}$ as the underlying $\infty$-category of Hu's model category. There is another construction entirely in $\infty$-categorical terms by Robalo \cite{Robalo}. In fact, any symmetric monoidal stable $\infty$-category satisfying the universal properties of Corollary 2.39 in \cite{Robalo} would serve our purposes.
 
For a fixed prime $p$, Voevodsky (see Section 3 of \cite{VoeEM}) 
constructed the mod $p$ motivic Eilenberg-Mac Lane spectrum that represents the mod $p$ motivic cohomology. We denote it by H$\mathbb{F}_p^{\textup{mot}}$. By arguments in Dundas-R\"ondigs-\O stv\ae r \cite[Example 3.4]{DRO}, H$\mathbb{F}_p^{\textup{mot}}$ is an $E_\infty$-algebra in $\C\text{-}\mathbf{mot}\text{-}\mathbf{Spectra}$. 

Its value at a point is
$${\text{H}\mathbb{F}_p^{\textup{mot}}}_{*,*} = \mathbb{F}_p[\tau],$$
where $\tau$ is in bidegree $(0,-1)$.

We denote by $S^{0,0}$ the motivic sphere spectrum. For the grading, we denote by $S^{1,0}$ the suspension spectrum of the simplicial sphere $S^1$, and by $S^{1,1}$ the suspension spectrum of the multiplicative group $\mathbb{G}_m = \mathbb{A}^1 - 0$. 

Let $\widehat{S^{s,w}}$ be the H$\mathbb{F}_p^{\textup{mot}}$-completed motivic sphere spectrum in bidegree $(s,w)$. It is a theorem of Hu-Kriz-Ormsby \cite{HKOrem, HKOcon} that $\widehat{S^{0,0}}$ and the usual $p$-completion of the motivic sphere spectrum have isomorphic motivic homotopy groups. 
Moreover, $\widehat{S^{0,0}}$ is an $E_\infty$-algebra in the symmetric monoidal $\infty$-category $\C\text{-}\mathbf{mot}\text{-}\mathbf{Spectra}$. See Section~\ref{Hcompletion} for more details regarding this fact and discussion on the H$\mathbb{F}_p^{\textup{mot}}$-completion. 

We denote by 
$$\widehat{S^{0,0}}\text{-}\mathbf{Mod}$$ 
the stable $\infty$-category of motivic module spectra over $\widehat{S^{0,0}}$.

The class $\tau$ can be lifted to a map between H$\mathbb{F}_p^{\textup{mot}}$-completed motivic sphere spectra
$$\tau: \widehat{S^{0,-1}} \longrightarrow \widehat{S^{0,0}}$$
that induces a nonzero map on mod $p$ motivic homology. The reader should be warned that $\tau$ does not further lift to a map between uncompleted motivic sphere spectra. See Dugger-Isaksen \cite{DuggerIsaksenMASS} and Hu-Kriz-Ormsby \cite{HKOrem} for more details. We denote by $\widehat{S^{0,0}}/\tau$ the cofiber of $\tau$.
\begin{equation*}    \label{eq:introdefCt}
\widehat{S^{0,-1}} \stackrel{\tau}{\lto} \widehat{S^{0,0}} \lto \widehat{S^{0,0}}/\tau \lto \widehat{S^{1,-1}}  
\end{equation*}

\begin{conv}
All smash products without subscript $\wedge$ in this paper 
are understood taken over the H$\mathbb{F}_p^{\textup{mot}}$-completed sphere spectrum $\widehat{S^{0,0}}$. We may still write $\wedge_{\widehat{S^{0,0}}}$ in a few places to emphasis that the smash product is taking over $\widehat{S^{0,0}}$.
\end{conv}

We have suspension functors 
$$\Sigma^{s,w}( - ) = \widehat{S^{s,w}} \wedge_{\widehat{S^{0,0}}} - $$ 
in the category $\widehat{S^{0,0}}\text{-}\mathbf{Mod}$ for any $s, w \in \mathbb{Z}$.  In particular, the suspension functor $\Sigma^{1,0}$ gives the translation automorphism (in the sense of Lurie \cite[Section 1.3.2]{HA}) of the stable $\infty$-category $\widehat{S^{0,0}}\text{-}\mathbf{Mod}$.  

Given an $E_\infty$-algebra $R \in \widehat{S^{0,0}}\text{-}\mathbf{Mod}$, denote by 
$$R\text{-}\mathbf{Mod}$$ 
the stable $\infty$-category of left modules over $R$ in $\widehat{S^{0,0}}\text{-}\mathbf{Mod}$.

Following Dugger-Isaksen \cite[Definition 2.10]{DuggerIsaksen}, denote by 
$$R\text{-}\mathbf{Mod}_{\textup{cell}}$$ 
the smallest stable subcategory containing $R$ that is closed under arbitrary small colimits and suspension by $\widehat{S^{s,w}}$ for all $p, q \in \Z$. We say an object in $R\text{-}\mathbf{Mod}$ is $R$-cellular if it is contained in $R\text{-}\mathbf{Mod}_{\textup{cell}}$.

Recall from Lurie \cite[Definition 1.1.3.2]{HA} that a stable subcategory of a stable $\infty$-category is a full subcategory containing a zero object and stable under the formation of fibers and cofibers. The reader should be warned that not all motivic spectra in $\widehat{S^{0,0}}\text{-}\mathbf{Mod}$ are weakly equivalent to a cellular object.


It is a theorem of the first author \cite{GheCt} that $\widehat{S^{0,0}}/\tau$ is an $E_\infty$-algebra in $\widehat{S^{0,0}}$-$\mathbf{Mod}$ at all primes $p$. In fact, the first author carried out all details in \cite{GheCt} for the case for $p=2$, using the vanishing regions of $\pi_{*,*}\widehat{S^{0,0}}/\tau$. It is straightforward to use the same arguments for all primes $p$. See Theorem~1.1 and explanations in Section~1.2 of \cite{GheCt} for more details. 

We therefore have defined stable $\infty$-categories 
$$\widehat{S^{0,0}}/\tau\text{-}\mathbf{Mod}\text{~ and~} \widehat{S^{0,0}}/\tau\text{-}\mathbf{Mod}_{\textup{cell}}.$$
We can view the ring map
$$\widehat{S^{0,0}} \longrightarrow \widehat{S^{0,0}}/\tau$$ 
to exhibit $\widehat{S^{0,0}}/\tau$ as the special fiber of the deformation parametrized by $\tau$. The generic fiber of this deformation is $\tau^{-1}\widehat{S^{0,0}}$. 

Let MGL be the cellular motivic algebraic cobordism spectrum introduced by Voevodsky \cite{Voe} and studied by Levine-Morel \cite{ML}, Panin-Pimenov-R\"ondigs \cite{PPR2} and many others. It is an $E_\infty$-algebra in $\C\text{-}\mathbf{mot}\text{-}\mathbf{Spectra}$ (See \cite[Theorem 14.2]{Hu03} for example). We define 
$$\textup{MU}^\textup{mot} := \textup{MGL} \wedge_{S^{0,0}} \widehat{S^{0,0}}.$$
It is therefore an $E_\infty$-algebra in $\widehat{S^{0,0}}$-$\mathbf{Mod}_{\textup{cell}}$. There is a natural map
$$\textup{MU}^\textup{mot} = \textup{MGL} \wedge_{S^{0,0}} \widehat{S^{0,0}} \rightarrow \textup{MGL}^\wedge_{\textup{H}\mathbb{F}_p^\textup{mot}}$$
to the H$\mathbb{F}_p^\textup{mot}$-completion of MGL. As we will explain in Proposition~\ref{MUmot}, this map induces an isomorphism on $\pi_{*,*}$. Their motivic homotopy groups are computed by Hu-Kriz-Ormsby \cite{HKOrem} and Dugger-Isaksen \cite[Section~8.3]{DuggerIsaksenMASS}.
$$\pi_{*,*}\textup{MU}^{\textup{mot}} = \pi_{*,*}\textup{MGL}^\wedge_{\textup{H}\mathbb{F}_p^\textup{mot}} = \mathbb{Z}_p[\tau][x_1, x_2, \cdots].$$
Here $\mathbb{Z}_p$ is the $p$-adic integers and $x_i$ is in bidegree $(2i, i)$. Since $\pi_{*,*}\text{MGL}$ is much more complicated, we will mostly work with $\textup{MU}^{\textup{mot}}$ instead of MGL in our paper.

For any $X \in \widehat{S^{0,0}}$-$\mathbf{Mod}_{\textup{cell}}$, we define the $\textup{MU}^{\textup{mot}}$-homology of $X$ as
$$\textup{MU}^{\textup{mot}}_{*,*}X = \pi_{*,*} (\textup{MU}^{\textup{mot}} \wedge_{\widehat{S^{0,0}}} X).$$
By adjunction, it is clear that the $\textup{MU}^{\textup{mot}}$-homology of $X$ equals to $\text{MGL}_{*,*}X$ when taking $X$ as its underlying motivic spectrum in $\C$-$\mathbf{mot}$-$\mathbf{Spectra}$.

The spectrum $\textup{MU}^{\textup{mot}}/\tau := \widehat{S^{0,0}}/\tau \wedge_{\widehat{S^{0,0}}} \textup{MU}^{\textup{mot}}$ is an $E_\infty$-algebra in $\widehat{S^{0,0}}/\tau$-$\mathbf{Mod}_{\textup{cell}}$. Its motivic homotopy groups are
$$\pi_{*,*}(\textup{MU}^{\textup{mot}}/\tau) = \mathbb{Z}_p[x_1, x_2, \cdots] = \textup{MU}^{\textup{mot}}_{*,*}/\tau.$$

Forgetting the motivic weight, the bigraded ring $\textup{MU}^{\textup{mot}}_{*,*}/\tau$ can be identified as the singly graded ring $\textup{MU}_*$ completed at the prime $p$.

\begin{defn} \label{def:pisscomplete}
	Let $X$ be a motivic spectrum in $\widehat{S^{0,0}}/\tau$-$\mathbf{Mod}$. We say that $X$ is \emph{harmonic},
	if $X$ is $\widehat{S^{0,0}}/\tau$-cellular and the map
	$$X \longrightarrow X^{\wedge}_{\text{MGL}}$$
	induces an isomorphism on $\pi_{*,*}$. We denote by 
	$$\widehat{S^{0,0}}/\tau\text{-}\mathbf{Mod}_{\textup{harm}}$$ 
	the full stable $\infty$-subcategory of harmonic $\widehat{S^{0,0}}/\tau$-module spectra.
\end{defn}

Here the MGL-nilpotent completion $X^{\wedge}_{\text{MGL}}$ is understood taken in $\C\text{-}\mathbf{mot}\text{-}\mathbf{Spectra}$. For a precise definition, see Section~\ref{Hcompletion} and \cite{DuggerIsaksenMASS, HKOrem}. One could also define the $\textup{MU}^{\textup{mot}}$-completion $X^{\wedge}_{\textup{MU}^{\textup{mot}}}$ in $\widehat{S^{0,0}}\text{-}\mathbf{Mod}$ for any $X$ in $\widehat{S^{0,0}}$-$\mathbf{Mod}_{\textup{cell}}$. By adjunction, it is clear that the two completions $X^{\wedge}_{\text{MGL}}$ and $X^{\wedge}_{\textup{MU}^{\textup{mot}}}$ are equivalent.
\begin{displaymath}
\xymatrix{
X^{\wedge}_{\text{MGL}} \ar[r]^-{\simeq} & X^{\wedge}_{\textup{MU}^{\textup{mot}}}
}.	
\end{displaymath}
So we may equivalently define a cellular $\widehat{S^{0,0}}/\tau$-module to be harmonic, if the map to its $\textup{MU}^{\textup{mot}}$-completion induces an isomorphism on $\pi_{*,*}$.

It is clear that the spectrum $\textup{MU}^{\textup{mot}}/\tau$ is harmonic. See Section 4.1 for more examples and non-examples.

We will define $t$-structures on certain stable $\infty$-categories of motivic spectra, such as $\textup{MU}^{\textup{mot}}/\tau$-$\mathbf{Mod}_{\textup{cell}}$ and $\widehat{S^{0,0}}/\tau$-$\mathbf{Mod}_{\textup{harm}}$. Recall that by Lurie's definition \cite[Definition 1.2.1.4]{HA}, a $t$-structure on a stable $\infty$-category is a $t$-structure on its homotopy category, which is a triangulated category. To describe these $t$-structures, we define the Chow-Novikov degree of an element that belongs to the bigraded homotopy groups of a motivic spectrum.

 
\begin{defn}
For any motivic spectrum $X$, consider its bigraded motivic homotopy groups $$\pi_{s,w}X.$$ Here $s$ is the topological degree under the Betti realization, and $w$ is the motivic weight. The \emph{Chow-Novikov degree} of an element in $\pi_{s,w}X$ is defined as $s-2w$.
	
	We say that $\pi_{*,*}X$ is \emph{concentrated in Chow-Novikov degrees} $I$, where $I$ is a set of integers, if all nonzero elements in $\pi_{*,*}X$ are concentrated in Chow-Novikov degrees belonging to $I$.
\end{defn}

For example, the homotopy groups of $\textup{MU}^{\textup{mot}}/\tau$ are concentrated in Chow-Novikov degree 0, while the homotopy groups of $\textup{MU}^{\textup{mot}}$ are concentrated in non-negative even Chow-Novikov degrees.

\begin{defn}\leavevmode \label{def:thicksub}
\begin{enumerate}
	\item We define 
	$$\textup{MU}^{\textup{mot}}/\tau\text{-}\mathbf{Mod}_{\textup{cell}}^b$$ as the stable full subcategory of
	\hbox{$\textup{MU}^{\textup{mot}}/\tau\text{-}\mathbf{Mod}_{\textup{cell}}$} spanned by objects whose homotopy groups are concentrated in bounded Chow-Novikov degrees.
	\item We define $$\textup{MU}^{\textup{mot}}/\tau\text{-}\mathbf{Mod}_{\textup{cell}}^{b, \geq 0}, $$ $$ \textup{MU}^{\textup{mot}}/\tau\text{-}\mathbf{Mod}_{\textup{cell}}^{b, \leq 0}, $$ $$ \textup{MU}^{\textup{mot}}/\tau\text{-}\mathbf{Mod}_{\textup{cell}}^{\heartsuit}$$ as the full subcategories of $\textup{MU}^{\textup{mot}}/\tau$-$\mathbf{Mod}_{\textup{cell}}^b$ spanned by objects whose homotopy groups are concentrated in nonnegative, nonpositive and zero Chow-Novikov degrees respectively.
	\item We define $$\widehat{S^{0,0}}/\tau\text{-}\mathbf{Mod}_{\textup{harm}}^b$$ as the stable full subcategory of $\widehat{S^{0,0}}/\tau$-$\mathbf{Mod}_{\textup{harm}}$ spanned by objects whose $\textup{MU}^{\textup{mot}}$-homology groups are concentrated in bounded Chow-Novikov degrees.
	\item We define $$\widehat{S^{0,0}}/\tau\text{-}\mathbf{Mod}_{{\textup{harm}}}^{b, \geq 0},$$ $$\widehat{S^{0,0}}/\tau\text{-}\mathbf{Mod}_{{\textup{harm}}}^{b, \leq 0},$$  $$\widehat{S^{0,0}}/\tau\text{-}\mathbf{Mod}_{{\textup{harm}}}^{\heartsuit}$$ as the full subcategories of \hbox{$\widehat{S^{0,0}}/\tau$-$\mathbf{Mod}_{{\textup{harm}}}^b$} spanned by objects whose $\textup{MU}^{\textup{mot}}$-homology groups are concentrated in nonnegative, nonpositive and zero Chow-Novikov degrees respectively.
\end{enumerate}	
\end{defn}

\begin{defn}
We define 
	$$\textup{MU}_*\text{-}\mathbf{Mod}^\textup{ev}$$
as the abelian category of graded modules that are concentrated in even degrees over the $p$-completed ring $\textup{MU}_*$, and 
	$$\textup{MU}_*\textup{MU}\text{-}\mathbf{Comod}^\textup{ev}$$
	$$\textup{BP}_*\textup{BP}\text{-}\mathbf{Comod}^\textup{ev}$$ 
as the abelian categories of graded comodules that are concentrated in even degrees over the $p$-completed Hopf algebroids $\textup{MU}_*\textup{MU}$ and $\textup{BP}_*\textup{BP}$.
	
We define 
$$\mathcal{D}^b(\textup{MU}_*\text{-}\mathbf{Mod}^\textup{ev}),$$ 
$$\mathcal{D}^b(\textup{MU}_*\textup{MU}\text{-}\mathbf{Comod}^\textup{ev})$$
$$\mathcal{D}^b(\textup{BP}_*\textup{BP}\text{-}\mathbf{Comod}^\textup{ev})$$ 
as their bounded derived categories.
\end{defn}

\begin{prop} \cite[Proposition~1.2.3]{Mor} \label{prop:bpbpmumucomod}
At each prime p, the categories
${\textup{BP}_*\textup{BP}\text{-}\mathbf{Comod}}^{\textup{ev}}$ and  $\textup{MU}_*\textup{MU}\text{-}\mathbf{Comod}^\textup{ev}$ are equivalent as abelian categories.
\end{prop}

The abelian categories of modules over $p$-completed $\textup{MU}_*$ and $\text{BP}_*$ are \emph{not} equivalent. However, Proposition~\ref{prop:bpbpmumucomod} states that the abelian categories of even comodules over $p$-completed $\textup{MU}_*\textup{MU}$ and $\textup{BP}_*\textup{BP}$ are equivalent. We will work with $\textup{MU}$ and $\textup{MU}^{\textup{mot}}$ since they are $E_\infty$-algebras in the corresponding categories while $\textup{BP}$ is not, due to a recent result of Lawson \cite{Law}. 

\begin{thm}\leavevmode \label{thm:mumottaumod}
	\begin{enumerate}
		\item The full subcategories $\textup{MU}^{\textup{mot}}/\tau$-$\mathbf{Mod}_{\textup{cell}}^{b, \geq 0}$ and $\textup{MU}^{\textup{mot}}/\tau$-$\mathbf{Mod}_{\textup{cell}}^{b, \leq 0}$ define a $t$-structure on $\textup{MU}^{\textup{mot}}/\tau$-$\mathbf{Mod}_{\textup{cell}}^b$.
		\item The functor
		 $$\pi_{*,*}: \textup{MU}^{\textup{mot}}/\tau\text{-}\mathbf{Mod}_{\textup{cell}}^{\heartsuit} \longrightarrow \textup{MU}_*\text{-}\mathbf{Mod}^\textup{ev}$$
		 is an equivalence.
		 \item There exists an equivalence of stable $\infty$-categories 		
		 $$\textup{MU}^{\textup{mot}}/\tau\text{-}\mathbf{Mod}_{\textup{cell}}^b \longrightarrow \mathcal{D}^b(\textup{MU}_*\text{-}\mathbf{Mod}^\textup{ev}),$$
		that preserves the given $t$-structures and extends the functor $\pi_{*,*}$ on the heart.
	\end{enumerate}
\end{thm}

\begin{remark} \label{remark:bigradevs}
The functor $\pi_{*,*}$ naturally lands in the category of bigraded modules over the bigraded ring $\textup{MU}^{\textup{mot}}_{*,*}/\tau$. Since all elements of this bigraded ring are concentrated in Chow-Novikov degree 0, it can be identified as the single graded ring $\textup{MU}_*$ by forgetting the motivic weight. A similar comment applies to the following theorem as well.
\end{remark}

\begin{thm}\leavevmode \label{thm:ctaumod}
	\begin{enumerate}
		\item The full subcategories $\widehat{S^{0,0}}/\tau$-$\mathbf{Mod}_{{\textup{harm}}}^{b, \geq 0}$ and $\widehat{S^{0,0}}/\tau\text{-}\mathbf{Mod}_{{\textup{harm}}}^{b, \leq 0}$ define a $t$-structure on\\ $\widehat{S^{0,0}}/\tau$-$\mathbf{Mod}_{\textup{harm}}^b$.
		\item The functor
		 $$\textup{MU}^{\textup{mot}}_{*,*}: \widehat{S^{0,0}}/\tau\text{-}\mathbf{Mod}_{\textup{harm}}^{\heartsuit} \longrightarrow \textup{MU}_*\textup{MU}\text{-}\mathbf{Comod}^\textup{ev}$$
		 is an equivalence.
		 \item There exists an equivalence of stable $\infty$-categories 
		$$\widehat{S^{0,0}}/\tau\text{-}\mathbf{Mod}_{\textup{harm}}^b \longrightarrow \mathcal{D}^b(\textup{MU}_*\textup{MU}\text{-}\mathbf{Comod}^\textup{ev}),$$
		 that preserves the given $t$-structures and extends the functor $\textup{MU}^{\textup{mot}}_{*,*}$ on the heart.
	\end{enumerate}
\end{thm}

\begin{remark}
The statements in Theorem~\ref{thm:mumottaumod} and Theorem~\ref{thm:ctaumod} can be connected by the following commutative diagram of stable $\infty$-categories with $t$-structures.
\begin{displaymath}
\xymatrix{
 \widehat{S^{0,0}}/\tau\text{-}\mathbf{Mod}_{\textup{harm}}^b \ar[r] \ar[d]_{-\wedge_{\widehat{S^{0,0}}/\tau} \textup{MU}^{\textup{mot}}/\tau} & \mathcal{D}^b(\textup{MU}_*\textup{MU}\text{-}\mathbf{Comod}^\textup{ev}) \ar[d] \\
 \textup{MU}^{\textup{mot}}/\tau\text{-}\mathbf{Mod}_{\textup{cell}}^b \ar[r]  & 
	\mathcal{D}^b(\textup{MU}_*\text{-}\mathbf{Mod}^\textup{ev}) }
\end{displaymath}
The vertical functor on the right is the forgetful functor.
\end{remark}

\begin{remark}
From a deformation perspective, our Theorem~\ref{thm:ctaumod} gives a new connection between the moduli stack of formal groups and the classical stable homotopy theory.

From the deformation 
$$\widehat{S^{0,0}}/\tau \xleftarrow{\textup{special\ fiber}} \widehat{S^{0,0}} \xrightarrow{\textup{generic\ fiber}} \tau^{-1}\widehat{S^{0,0}} $$
parametrized by $\tau$,
we have two adjunctions of stable $\infty$-categories:
$$(\widehat{S^{0,0}}\text{-}\mathbf{Mod})_{\tau=0} \adjre \widehat{S^{0,0}}\text{-}\mathbf{Mod} \adj
(\widehat{S^{0,0}}\text{-}\mathbf{Mod})_{\tau=1}.$$
We call this deformation a ``motivic deformation" intuitively.

By Dugger-Isaksen \cite{DuggerIsaksenMASS}, on the generic fiber, the full subcategory of cellular objects in
$$(\widehat{S^{0,0}}\text{-}\mathbf{Mod})_{\tau=1} := \tau^{-1}\widehat{S^{0,0}}\text{-}\mathbf{Mod}$$
is equivalent to the classical stable homotopy category at the prime p. In fact, Dugger-Isaksen showed that the motivic homotopy groups of the $\tau$-inverted sphere spectrum are isomorphic to that of the classical sphere spectrum. By an inductive argument, one can show that a similar statement is true for all finite cellular objects. This shows that $\tau$-inverted Betti realization functor is fully faithful. It is also essentially surjective since the Betti realization functor admits a section with constant weight zero. An Ind-object argument (similar to the proof of Corollary~1.2) gives us the claim.

Our main theorem shows that, on the special fiber, the full subcategory of harmonic objects in the category 
$$(\widehat{S^{0,0}}\text{-}\mathbf{Mod})_{\tau=0} := \widehat{S^{0,0}}/\tau\text{-}\mathbf{Mod}$$ 
is equivalent to the derived category of comodules that are concentrated in even degrees over the p-completed Hopf algebroid $\textup{MU}_*\textup{MU}$. By Quillen's theorem \cite{Qui}, the latter can be identified with the derived category of quasi-coherent sheaves on the moduli stack of formal groups over $\mathbb{Z}_p$-algebras.
\end{remark}

\begin{remark}
In our proof of Theorem \ref{thm:ctaumod}, we set up a strongly convergent motivic Adams-Novikov spectral sequence in the category $\widehat{S^{0,0}}/\tau\text{-}\mathbf{Mod}_{\textup{cell}}^b$,
$$
\Ext^{*,*,*}_{\textup{MU}^{\textup{mot}}_{*,*}\textup{MU}^{\textup{mot}}/\tau}(\textup{MU}^{\textup{mot}}_{*,*}X, \textup{MU}^{\textup{mot}}_{*,*}Y) \Longrightarrow \left[\Sigma^{*,*} X, Y^{\smas}_{\textup{MU}^{\textup{mot}}} \right]_{\widehat{S^{0,0}}/\tau}.
$$
This is stated as Theorem \ref{thm:ANss} in Section 5. Classically, the Adams-Novikov spectral sequence is set up such that the first variable is the sphere spectrum. Our construction could be generalized to an abstract setting and applied to the classical situation when the first variable $X$ is arbitrary. We will discuss this case in a general framework in future work.
\end{remark}


\subsection{The motivic Adams spectral sequence and the algebraic Novikov spectral sequence}

The following Theorem \ref{thm:iso13} establishes an isomorphism between the algebraic Novikov spectral sequence and the motivic Adams spectral sequence for $\widehat{S^{0,0}}/\tau$.

\begin{thm} \label{thm:iso13}
At each prime $p$, there is an isomorphism of tri-graded spectral sequences: the motivic Adams spectral sequence for $\widehat{S^{0,0}}/\tau$, which converges to the motivic homotopy groups of $\widehat{S^{0,0}}/\tau$, and the re-graded algebraic Novikov spectral sequence, which converges to the Adams-Novikov $E_2$-page for the sphere.


The indexes are indicated in the following diagram:
\begin{displaymath}
    \xymatrix{
  \Ext_{\textup{BP}_*\textup{BP}/I}^{s,2w}(\mathbb{F}_p, I^{a-s}/I^{a-s+1}) \ar@{=>}[dd]|{\mathbf{Algebraic \ Novikov \ SS}} \ar[rr]^{\cong} & & \Ext_{{A}_{*,*}^{\textup{mot}}}^{a, 2w-s+a,w}(\mathbb{F}_p[\tau], \mathbb{F}_p) \ar@{=>}[dd]|{\mathbf{Motivic \ Adams \ SS}} \\
  & & \\
  \Ext_{\textup{BP}_*\textup{BP}}^{s,2w}(\textup{BP}_*, \textup{BP}_*) \ar[rr]^{\cong} & & \pi_{2w-s,w}(\widehat{S^{0,0}}/\tau).
    }
\end{displaymath}

Here $I=(p,v_1,v_2,\dots)$ is the augmentation ideal of $\textup{BP}_*$ and ${A}_{*,*}^\textup{{mot}}$ is the motivic mod $p$ dual Steenrod algebra.
\end{thm}

Recall that both of the two spectral sequences in Theorem~\ref{thm:iso13} are multiplicative, so there are multiplicative filtrations on the abutments.  

\begin{thm} \label{thm:iso14}
There is an isomorphism between $\Ext_{\textup{BP}_*\textup{BP}}^{*,*}(\textup{BP}_*, \textup{BP}_*)$ and $\pi_{*,*}(\widehat{S^{0,0}}/\tau)$ that preserves the multiplicative filtrations, composition products and higher compositions in the respective categories.
\end{thm}

\begin{proof}
	The multiplicative structure on the abutments comes from composition of morphisms in both categories $\widehat{S^{0,0}}/\tau\text{-}\mathbf{Mod}_{\textup{harm}}^b$ and $\mathcal{D}^b({\textup{BP}_*\textup{BP}\text{-}\mathbf{Comod}}^{\textup{ev}})$. The isomorphism on abutments is induced by the equivalence of categories in Theorem~\ref{thm:ctaumod}, in particular respects compositions.
\end{proof}


The isomorphism between the abutments is known to Isaksen \cite[Proposition 6.2.5]{StableStems} and the first author \cite[Corollary 3.14]{GheCt}. Our Theorem~\ref{thm:iso14} further states that the isomorphism preserves the multiplicative filtrations on the abutments. We do not prove that the group isomorphism on the $E_2$-pages is also a ring isomorphism. We shall prove it in future work.

There has been huge interest in obtaining information on the stable homotopy groups of spheres by comparing the Adams spectral sequence with the Adams-Novikov spectral sequence. In fact, this is a dream entertained by Novikov \cite[Section 12]{Novikov}. See also \cite{Ravenel, MillerAdams, MillerThesis} for example. An important connection and technique of studying both spectral sequences is the following the Miller square \cite{MillerAdams}.
 \begin{displaymath}
    \xymatrix{
  & & \Ext^{s,t}_{P_*}(\mathbb{F}_p, I^{a-s}/I^{a-s+1}) \ar@{=>}[ddll]|{\mathbf{Cartan\textup{-}Eilenberg \ SS}} \ar@{=>}[ddrr]|{\mathbf{Algebraic \ Novikov \ SS}} & &  \\
  & & & & \\
 \Ext_{A_*}^{a,t}(\mathbb{F}_p, \mathbb{F}_p) \ar@{=>}[ddrr]|{\mathbf{Adams \ SS}} & & & & \Ext^{s,t}_{\textup{BP}_*\textup{BP}}(\textup{BP}_*,\textup{BP}_*) \ar@{=>}[ddll]|{\mathbf{Adams\textup{-}Novikov \ SS}}\\
  & & & & \\
  & & \pi_{*}\widehat{S^0}  & &
    }
\end{displaymath}
By a change-of-ring isomorphism, the $E_2$-page of the Cartan-Eilenberg spectral sequence, which computes the Adams $E_2$-page, is isomorphic to the algebraic Novikov spectral sequence, which computes the Adams-Novikov $E_2$-page. For $p$ odd, the Cartan-Eilenberg spectral sequence collapses for degree reasons.

To explore this square, Miller \cite{MillerAdams} smashes together the Adams resolution and the Adams-Novikov resolution, and gets a comparison theorem on the $d_2$-differentials in the algebraic Novikov spectral sequence and the Adams spectral sequence. The following theorem is due to Novikov \cite{Novikov}, Miller \cite[Theorem 4.2]{MillerAdams} and Andrews-Miller \cite[Theorem 9.3.3]{AM16}.


\begin{thm} \label{thm:mil}
Suppose $z'$ is an element in $\Ext_{A_*}^{a,t}(\mathbb{F}_p, \mathbb{F}_p)$ that has Cartan-Eilenberg filtration $s$. Then $d_2^{\mathbf{ASS}}z'$ has higher Cartan-Eilenberg filtration.

Moreover, if $z'$ is detected in the Cartan-Eilenberg spectral sequence by\\ $z$ in $\Ext^{s,t}_{P_*}(\mathbb{F}_p, I^{a-s}/I^{a-s+1})$, then  
	$$d_2^{\mathbf{ASS}}z' \ \text{is detected by} \ d_2^{\mathbf{algNSS}} z,$$
where $d_2^{\mathbf{algNSS}} z$ is in $\Ext^{s+1,t}_{P_*}(\mathbb{F}_p, I^{a-s+1}/I^{a-s+2})$.
\end{thm}

\begin{remark}
We regraded the algebraic Novikov spectral sequence so our $d_2^{\mathbf{algNSS}}$	is $d_1^{\mathbf{algNSS}}$ in \cite[Theorem 9.3.3]{AM16}.
\end{remark}

Based on the Miller square and Theorem \ref{thm:mil}, Miller \cite{MillerAdams} proves the Telescope Conjecture at chromatic height 1 at odd primes.
 
To understand the connection between $\mathbf{higher}$ differentials in the Adams and algebraic Novikov spectral sequences, it would be desirable to establish new connections between them. 

For example, suppose in general that we have two spectral sequences 
$$E_2\Rightarrow E_\infty, \ E_2'\Rightarrow E_\infty'$$ 
that are not necessarily connected by a homomorphism of spectral sequences. To compare them, it would be useful to have a third spectral sequence 
$$E_2''\Rightarrow E_\infty''$$ 
making a zig-zag diagram of spectral sequences.
\begin{displaymath}
 \xymatrix{
 E_2  \ar@{=>}[d]& E_2'' \ar[l]\ar[r]\ar@{=>}[d] & E_2' \ar@{=>}[d] \\
 E_\infty & E_\infty'' \ar[l]\ar[r] & E_\infty'
 }
\end{displaymath}
This is the one of the major techniques used by the second and third authors in \cite{WX} to explore the Mahowald square \cite{Mah} and compute differentials in the Adams spectral sequences. 
 
Following this philosophy, for the Miller square \cite{MillerAdams}, a basic question would be: Which spectral sequence can we put in between these two spectral sequences and have a zig-zag diagram?
\begin{displaymath}
  \xymatrix{
 \Ext^{s,t}_{P_*}(\mathbb{F}_p, I^{a-s}/I^{a-s+1})   \ar@{=>}[dd]|{\mathbf{Algebraic \ Novikov \ SS}} & & ?  \ar[ll] \ar[rr]\ar@{=>}[dd] & & \Ext_{A_*}^{a,t}(\mathbb{F}_p, \mathbb{F}_p) \ar@{=>}[dd]|{\mathbf{Adams \ SS}} \\
 & & & & \\
\Ext^{s,t}_{\textup{BP}_*\textup{BP}}(\textup{BP}_*, \textup{BP}_*) & & ? \ar[ll] \ar[rr] & & \pi_{*}\widehat{S^0} 
}
\end{displaymath}
 
Our Theorem \ref{thm:iso13} shows that we can achieve a zig-zag diagram in the motivic world.
 
In fact, consider the H$\mathbb{F}_p^{\textup{mot}}$-completed motivic sphere spectrum $\widehat{S^{0,0}}$. Inverting $\tau$, we get the classical $p$-completed sphere $\widehat{S^0}$ by Dugger-Isaksen \cite{DuggerIsaksenMASS}, in the sense that the corresponding Adams and Adams-Novikov spectral sequences have equivalent data. On the other hand, reducing mod $\tau$, we get $\widehat{S^{0,0}}/\tau$. Then the naturality of the Adams spectral sequences gives us a zig-zag diagram.
\begin{displaymath}
  \xymatrix{
 \Ext^{s,t}_{A_{*,*}^{mot}}(\mathbb{F}_p[\tau], \mathbb{F}_p)   \ar@{=>}[dd]|{\mathbf{Motivic \ Adams \ SS}} & &  \Ext^{s,t}_{A_{*,*}^{mot}}(\mathbb{F}_p[\tau], \mathbb{F}_p[\tau]) \ar[ll]\ar[rr]\ar@{=>}[dd]|{\mathbf{Motivic \ Adams \ SS}} & & \Ext_{A_*}^{a,t}(\mathbb{F}_p, \mathbb{F}_p) \ar@{=>}[dd]|{\mathbf{Adams \ SS}} \\
 & & & & \\
\pi_{*,*}\widehat{S^{0,0}}/\tau & & \pi_{*,*}\widehat{S^{0,0}} \ar[ll]\ar[rr]^-{\mathbf{Re}} & & \pi_{*}\widehat{S^0} 
}
\end{displaymath}
By Theorem \ref{thm:iso13}, the left side spectral sequence, which is the motivic Adams spectral sequence for $\widehat{S^{0,0}}/\tau$, is isomorphic to the algebraic Novikov spectral sequence. 
 
More generally, we have the following motivic square.
\begin{displaymath}
    \xymatrix{
  & & 
  \Ext_{A_{*,*}^{mot}}^{*,*,*}(\mathbb{F}_p[\tau], \mathbb{F}_p)[\tau] \ar@{=>}[ddll]|{\mathbf{Algebraic \ }\tau\textup{-}\mathbf{Bockstein \ SS}} \ar@{=>}[ddrr]
  |{\mathbf{Motivic \ Adams \ SS}}
  & &  \\
  & & & & \\
 \Ext_{A_{*,*}^{mot}}^{*,*,*}(\mathbb{F}_p[\tau], \mathbb{F}_p[\tau]) \ar@{=>}[ddrr]|{\mathbf{Motivic \ Adams \ SS}}  & & & & \pi_{*,*}\widehat{S^{0,0}}/\tau[\tau] \ar@{=>}[ddll]|{\tau\textup{-}\mathbf{Bockstein \ SS}}\\
  & & & & \\
  & & 
  \pi_{*,*}\widehat{S^{0,0}}  
  & &
    }
\end{displaymath}
Let's compare the motivic square with the Miller square.

For the lower right side, it is proved by Isaksen \cite{StableStems} that the motivic Adams-Novikov spectral sequence for $\widehat{S^{0,0}}$ is isomorphic to the $\tau$-Bockstein spectral sequence, and that it is rigid, in the sense that it contains the same information as the classical Adams-Novikov spectral sequence. Each nontrivial differential in the classical Adams-Novikov spectral sequence corresponds to a family of nontrivial differentials in the motivic Adams-Novikov spectral sequence, that are connected to each other by multiplication by $\tau$. We can recover all nonzero differentials in the motivic Adams-Novikov spectral sequence by knowing all nonzero differentials in the classical Adams-Novikov spectral sequence, and vice versa. 

We'd like to point out that the above isomorphism between the motivic Adams-Novikov spectral sequence for the motivic sphere and the $\tau$-Bockstein spectral sequence for the motivic sphere does not come from a map between two towers. In particular, the Chow-Novikov degree  compresses motivic Adams-Novikov $d_{2r+1}$-differentials to $\tau$-Bockstein $d_r$-differentials, and the motivic Adams-Novikov $E_2$-page is isomorphic to the $\tau$-Bockstein $E_1$-page, which is isomorphic to the homotopy groups of a motivic spectrum. Moreover, the edge map of the $\tau$-Bockstein spectral sequence has the advantage of being induced by an actual map between motivic spectra, so naturality applies to motivic homotopy groups. Naturality also gives us a map of the motivic Adams spectral sequences from the motivic sphere to $\widehat{S^{0,0}}/\tau$, where the latter is isomorphic to the classical algebraic Novikov spectral sequence.

For the upper left side, the relation of the two spectral sequences in the motivic square and the Miller square is the same as the relation on the lower right side. The algebraic $\tau$-Bockstein spectral sequence can be thought as a motivic version of the Cartan-Eilenberg spectral sequence, and contains the same information, in the same sense as the lower right side situation.

For the upper right side, our Theorem \ref{thm:iso13} says that the two spectral sequences are isomorphic.

Therefore, for three out of the four sides, the motivic square contains exactly the same information as the ones in the Miller square. 

For the remaining lower left side, Dugger-Isaksen \cite{DuggerIsaksenMASS} show that the $\tau$-inverted motivic Adams spectral sequence is isomorphic to the $\tau$-inverted classical Adams spectral sequence. This means that the difference between the motivic square and the Miller square lies in the $\tau$-torsion information. Therefore, when comparing the higher differentials in the classical and motivic Adams spectral sequences, the $\tau$-torsion information is necessary to make the zig-zag strategy work. 

Now, to compute a nontrivial classical Adams differential, for any $r$, start with an algebraic Novikov $d_r$-differential. Theorem \ref{thm:iso13} gives us a motivic Adams $d_r$-differential for $\widehat{S^{0,0}}/\tau$. Pulling back to the bottom cell of $\widehat{S^{0,0}}/\tau$ of the source element gives us a motivic Adams $d_{r'}$-differential for the motivic sphere with $r' \leq r$. Using the Betti realization functor, we then obtain a classical Adams $d_{r'}$-differential!

In practice, Isaksen, the second and the third authors \cite{IWX} extend the computation of classical and motivic stable stems into a large range using the following steps.

\begin{enumerate}
	\item Use a computer to carry out the entirely algebraic computation of the cohomology of the $\mathbb{C}$-motivic Steenrod algebra. These groups serve as the input to the $\mathbb{C}$-motivic Adams spectral sequence.
	\item Use a computer to carry out the entirely algebraic computation of the algebraic Novikov spectral sequence that converges to the cohomology of the Hopf algebroid $(\textup{BP}_*, \textup{BP}_*\textup{BP})$. This includes all differentials, and the multiplicative structure of the cohomology of $(\textup{BP}_*, \textup{BP}_*\textup{BP})$.
	\item Use Theorem \ref{thm:iso13} to identify the algebraic Novikov spectral sequence with the motivic Adams spectral sequence that computes the homotopy groups of $\widehat{S^{0,0}}/\tau$. This includes an identification of the cohomology of $(\textup{BP}_*, \textup{BP}_*\textup{BP})$ with the homotopy groups of $\widehat{S^{0,0}}/\tau$.
	\item Use the inclusion of the bottom cell and the projection to the top cell to pull back and push forward Adams differentials for $\widehat{S^{0,0}}/\tau$ to Adams differentials for the motivic sphere.
	\item Apply a variety of ad hoc arguments to deduce additional Adams differentials for the motivic sphere. The most important method involves shuffling Toda brackets.
	\item Use a long exact sequence in homotopy groups to deduce hidden $\tau$-extensions in the motivic Adams spectral sequence for the sphere.
	\item Invert $\tau$ to obtain the classical Adams spectral sequence and the classical stable homotopy groups.
\end{enumerate}

We'd like to highlight a few consequences of our stem-wise computations.

\begin{example}
Consider the following four differentials in the classical Adams spectral sequence for the $2$-completed sphere. 
\begin{enumerate}
\item There is a $d_3$ differential in the 15-stem
$$d_3(h_0h_4) = h_0d_0.$$
This is proved by May and Mahowald-Tangora in \cite{May2, MahowaldTangora} by comparing with Toda's unstable computations \cite{Todacomp}. 

\item There is a $d_4$ differential in the 38-stem
$$d_4(h_3h_5) = h_0x.$$
This is proved in Mahowald-Tangora \cite{MahowaldTangora} by an ad-hoc method using a certain finite CW spectrum. 

\item There is a $d_3$ differential in the 38-stem
$$d_3(e_1) = h_1t.$$
This is proved by Bruner in \cite{Br1} by power operations in the Adams spectral sequence.

\item
There is a $d_3$ differential in the 61-stem
$$d_3(D_3) = B_3.$$
This is proved by the second and third authors \cite{WX} using the $RP^\infty$-technique. The proof of this differential in \cite{WX} is a significant part of the proof that the 61-sphere has a unique smooth structure.
\end{enumerate}
It turns out that all these four differentials can be proved by our method. They all correspond to nontrivial differentials in the algebraic Novikov spectral sequence with the same length, and therefore are all consequences of purely algebraic computations and our Theorem \ref{thm:iso13}. 
\end{example}

\begin{remark}
For some of the differentials computed by Isaksen, the second and the third authors \cite{IWX} using Theorem \ref{thm:iso13},
our method gives the only known proof. For example, we prove an Adams $d_3$-differential in the 68-stem
$$d_3(d_2) = h_0^2Q_3,$$
which shows the non-existence of the homotopy class $\kappa_2$ in $\pi_{68}$. As another example, we prove an Adams $d_5$-differential in the 92-stem
$$d_5(g_3) = h_6d_0^2,$$
which shows the non-existence of the homotopy class $\overline{\kappa}_3$ in $\pi_{92}$. Since both the elements $d_2$ and $g_3$ lie in a nonzero $Sq^0$-family in the 4-line of the classical Adams $E_2$-page, the two new nontrivial differentials serve as new evidence of Minami's New Doomsday Conjecture.
\end{remark}

\begin{remark}
Theorem \ref{thm:iso13} can also be used to compute nontrivial extensions and Toda brackets. For example, there is an $\eta$-extension from $h_3d_1$ to $N$ in the 46-stem. This is proved by the second and third authors \cite[Proposition 1.3(2)]{WX2} using the $RP^\infty$-technique. As another example, there is a Toda bracket 
$$\langle \theta_4, 2, \sigma^2 \rangle$$
	in the 45-stem. It is computed by Isaksen in \cite[Lemma 4.2.91]{StableStems} by ad hoc methods. This Toda bracket computation is crucial in the third author's proof \cite{Xu1} that 
	$$2\theta_5 = 0$$ 
	in the 62-stem. Both the nontrivial $\eta$-extension and the Toda bracket computations are present in the motivic homotopy groups of $\widehat{S^{0,0}}/\tau$. By Theorem \ref{thm:iso13}, they can be computed by the product and Massey product structure on the classical Adams-Novikov $E_2$-page. In particular, the corresponding 3-fold Massey product can be verified in the algebraic Novikov spectral sequence using May's convergence theorem \cite{May}. Therefore, both the nontrivial $\eta$-extension and the Toda bracket computations are consequences of purely algebraic computations and our Theorem \ref{thm:iso13}.
\end{remark}

\subsection{Organization}

This paper is organized in two parts. 

In Part 1, we prove the equivalence of stable $\infty$-categories in Theorem \ref{thm:mumottaumod} and Theorem \ref{thm:ctaumod}. Our proofs use a theorem of Lurie in Higher Algebra \cite{HA} on the relation between a stable $\infty$-category with a $t$-structure and the derived category of its heart. We recall Lurie's theorem in Section 2, and prove Theorem \ref{thm:mumottaumod} and Theorem \ref{thm:ctaumod} in Section 3 and 4. We also prove Corollary~1.2 in the end of Section 4. We introduce the absolute Adams-Novikov spectral sequence in the category $\widehat{S^{0,0}}/\tau\text{-}\mathbf{Mod}_{\textup{harm}}^b$ in Section~5, which is necessary for our proof of Theorem \ref{thm:ctaumod}. We propose some further questions in Section~6. We discuss the H$\mathbb{F}_p^\textup{mot}$-completion in Section~7.

In Part 2, we prove the isomorphism of spectral sequences in Theorem \ref{thm:iso13}.  In Section~\ref{eight}, we check that, through the equivalence of stable $\infty$-categories in Theorem \ref{thm:ctaumod}, the algebraic Novikov tower in the derived category of $\textup{BP}_*\textup{BP}$-comodules corresponds to the motivic Adams tower of $\widehat{S^{0,0}}/\tau$ in the category of $\widehat{S^{0,0}}/\tau$-modules. In Section 10, we re-compute certain low filtration and historically more difficult differentials in the range up to the 45-stem at the prime 2, as an illustration of the power of the isomorphism of spectral sequences in Theorem \ref{thm:iso13}.  


\subsection{Conventions and Notations}
All colimits and limits in a stable $\infty$-category of spectra mean homotopy colimit and homotopy limit in the classical sense.

All modules over graded rings are graded modules. 

Here is a summary of a list of the notations we use in this paper.
\begin{enumerate}
\item $S^{1,0}$: the simplicial sphere $S^1$.
\item  $S^{1,1}$: the multiplicative group $\mathbb{G}_m = \mathbb{A}^1 - 0$.
\item H$\mathbb{F}_p^\textup{mot}$: the mod $p$ motivic Eilenberg-Mac Lane spectrum.
\item $\widehat{S^{0,0}}$: the motivic H$\mathbb{F}_p^\textup{mot}$-completed sphere spectrum.
\item $\widehat{S^{s,w}}$: the motivic H$\mathbb{F}_p^\textup{mot}$-completed sphere spectrum in bidegree $(s,w)$.
\item $\widehat{S^{0,0}}/\tau$: the cofiber of $\tau$.
\item $\C\text{-}\mathbf{mot}\text{-}\mathbf{Spectra}$: the stable $\infty$-category of motivic spectra over $\Spec \C$.
\item $\widehat{S^{0,0}}\text{-}\mathbf{Mod}$: the stable $\infty$-category of motivic module spectra over $\widehat{S^{0,0}}$.
\item $\Sigma^{s,w}( - )$: the suspension functor $\widehat{S^{s,w}} \wedge_{\widehat{S^{0,0}}} -$.
\item $R\text{-}\mathbf{Mod}$: the stable $\infty$-category of left modules over $R$ in $\widehat{S^{0,0}}\text{-}\mathbf{Mod}$, for an $E_\infty$-algebra $R \in \widehat{S^{0,0}}\text{-}\mathbf{Mod}$.
\item $R\text{-}\mathbf{Mod}_{\textup{cell}}$: the stable $\infty$-category of cellular objects in $R\text{-}\mathbf{Mod}$.
\item MGL: the cellular motivic algebraic cobordism spectrum.
\item $\textup{MU}^\textup{mot}$: the $E_\infty$-algebra $\textup{MGL} \wedge_{S^{0,0}} \widehat{S^{0,0}}$ in $\widehat{S^{0,0}}\text{-}\mathbf{Mod}_\textup{cell}$.
\item $\textup{MU}^{\textup{mot}}/\tau$: the $E_\infty$-algebra $\widehat{S^{0,0}}/\tau \wedge_{\widehat{S^{0,0}}} \textup{MU}^{\textup{mot}}$ in $\widehat{S^{0,0}}/\tau\text{-}\mathbf{Mod}_\textup{cell}$.
\item $\widehat{S^{0,0}}/\tau\text{-}\mathbf{Mod}_{\textup{harm}}$: the stable $\infty$-category of harmonic $\widehat{S^{0,0}}/\tau$-module spectra.
\item The Chow-Novikov degree of an element in $\pi_{s,w}X$: $s-2w$, for $X$ a motivic spectrum.
\item $\textup{MU}^{\textup{mot}}/\tau\text{-}\mathbf{Mod}_{\textup{cell}}^b$: the stable $\infty$-category of
	cellular $\textup{MU}^{\textup{mot}}/\tau$-modules whose homotopy groups are concentrated in bounded Chow-Novikov degrees.
\item $\textup{MU}^{\textup{mot}}/\tau\text{-}\mathbf{Mod}_{\textup{cell}}^{b, \geq 0}, \ \textup{MU}^{\textup{mot}}/\tau\text{-}\mathbf{Mod}_{\textup{cell}}^{b, \leq 0}, \ \textup{MU}^{\textup{mot}}/\tau\text{-}\mathbf{Mod}_{\textup{cell}}^{\heartsuit}$: full subcategories of\\ $\textup{MU}^{\textup{mot}}/\tau\text{-}\mathbf{Mod}_{\textup{cell}}^b$ spanned by objects whose homotopy groups are concentrated in nonnegative, nonpositive and zero Chow-Novikov degrees. 
\item $\widehat{S^{0,0}}/\tau\text{-}\mathbf{Mod}_{\textup{harm}}^b$: the stable $\infty$-category of harmonic $\widehat{S^{0,0}}/\tau$-modules whose $\textup{MU}^{\textup{mot}}$-homology groups are concentrated in bounded Chow-Novikov degrees.
\item $\widehat{S^{0,0}}/\tau\text{-}\mathbf{Mod}_{{\textup{harm}}}^{b, \geq 0}, \ \widehat{S^{0,0}}/\tau\text{-}\mathbf{Mod}_{{\textup{harm}}}^{b, \leq 0}, \ \widehat{S^{0,0}}/\tau\text{-}\mathbf{Mod}_{{\textup{harm}}}^{\heartsuit}$: full subcategories of\\ \hbox{$\widehat{S^{0,0}}/\tau$-$\mathbf{Mod}_{{\textup{harm}}}^b$} spanned by objects whose $\textup{MU}^{\textup{mot}}$-homology groups are concentrated in nonnegative, nonpositive and zero Chow-Novikov degrees.
\item $\textup{MU}_*\text{-}\mathbf{Mod}^\textup{ev}$: the abelian category of graded modules that are concentrated in even degrees over the $p$-completed ring $\textup{MU}_*$.
\item  $\textup{MU}_*\textup{MU}\text{-}\mathbf{Comod}^\textup{ev}$: the abelian category of graded comodules that are concentrated in even degrees over the $p$-completed Hopf algebroid $\textup{MU}_*\textup{MU}$.
\item  $\textup{BP}_*\textup{BP}\text{-}\mathbf{Comod}^\textup{ev}$: the abelian category of graded comodules that are concentrated in even degrees over the $p$-completed Hopf algebroid $\textup{BP}_*\textup{BP}$.
\item $\textup{MU}^{\textup{mot}}_{*,*}/\tau\text{-}\mathbf{Mod}$:
the abelian category of graded left modules over $\textup{MU}^{\textup{mot}}_{*,*}/\tau$.
\item $\textup{MU}^{\textup{mot}}_{*,*}/\tau\text{-}\mathbf{Mod}^0$: 
the abelian category of graded left modules over $\textup{MU}^{\textup{mot}}_{*,*}/\tau$ that are concentrated in Chow-Novikov degree 0.
\item $\textup{MU}^{\textup{mot}}_{*,*}\textup{MU}^{\textup{mot}}/\tau\text{-}\mathbf{Comod}$: 
the abelian category of graded left comodules over the Hopf algebroid $\textup{MU}^{\textup{mot}}_{*,*}\textup{MU}^{\textup{mot}}/\tau$.
\item $\textup{MU}^{\textup{mot}}_{*,*}\textup{MU}^{\textup{mot}}/\tau\text{-}\mathbf{Comod}^0$:  
the abelian category of graded left comodules over the Hopf algebroid $\textup{MU}^{\textup{mot}}_{*,*}\textup{MU}^{\textup{mot}}/\tau$ that are concentrated in Chow-Novikov degree 0.
\item $\mathcal{D}^b(\mathcal{A})$: the bounded derived category of an abelian category $\mathcal{A}$ as a stable $\infty$-category.
\item $\mathbf{Stable}(\textup{BP}_*\textup{BP})$: the underlying stable $\infty$-category of Hovey's unbounded derived category of $\textup{BP}_*\textup{BP}$-comodules.
\end{enumerate}

\subsection{Acknowledgements}
We thank Dan Isaksen for proposing a triangulated category version of Theorem \ref{intro:mainthmpart1} as a conjecture and suggesting the authors to prove it. We also thank him for producing the charts in the appendix. We thank Mark Behrens for his encouragement and comments regarding Theorem \ref{intro:thmequivss}. We thank Lars Hesselholt, Dan Isaksen and Haynes Miller for many helpful comments on the early draft of the paper. We thank S{\o}ren Eiles and MiG (Mini\textup{MU}m intrusion Grid) for helping arranging computing resources, which lead to the discovery of Theorem \ref{intro:thmequivss}. We thank Linquan Ma for pointing out the reference regarding Lemma \ref{lem:tech}. We thank Ben Antieau, Tom Bachmann, Tobi Barthel, Agnes Beaudry, Eva Belmont, Robert Burklund, Paul Goerss, Jesper Grodal, Jeremy Hahn, Mike Hill, Mike Hopkins, Hana Jia Kong, Tyler Lawson, Zhi L\"{u}, Akhil Mathew, Peter May, Doug Ravenel, Nicolas Ricka, John Rognes, Jonathan Rubin, XiaoLin Danny Shi, Donald Stanley, Weiping Zhang and Yifei Zhu for conversations regarding this subject. We thank three anonymous referees for their detailed suggestions and comments. The second author was partially supported by the Danish National Research Foundation through the Centre for Symmetry and Deformation (DNRF92), and by grant NSFC-11801082. The third author was partially supported by the National Science Foundation under Grant No. DMS-1810638.

\part{Equivalence of stable $\infty$-categories} \label{chap:Ctmod}

The question of when the homotopy category of module spectra over a certain ring spectrum is equivalent to the derived category of an abelian category as a triangulated category has been studied in many context by many people. For example, Schwede and Shipley \cite{SS} studied the case for the Eilenberg-Mac Lane spectrum H$R$, where $R$ is a commutative ring, Patchkoria \cite{Pat} studied the case for the complex periodic $K$-theory localized at an odd prime, Greenlees \cite{Gre} studied the case for the rational $S^1$-equivarant sphere spectrum, and Deligne and Goncharov \cite{DG} studied the case for the rational motivic Eilenberg-Mac Lane spectrum H$\mathbb{Q}^{\textup{mot}}$. The answers are positive in these cases. On the other hand, Schwede \cite{Schwede} showed that the classical stable homotopy category is not a derived category.

The goal of Part 1 is to prove that the homotopy category of harmonic $\widehat{S^{0,0}}/\tau$-spectra whose $\textup{MU}^\textup{mot}$-homology are concentrated in bounded Chow-Novikov degrees is equivalent to the bounded derived category of $\textup{MU}_*\textup{MU}$-comodules that are concentrated in even degrees. In fact, we prove Theorem \ref{thm:ctaumod} that there exists an equivalence of stable $\infty$-categories that preserves the given $t$-structures
		$$\widehat{S^{0,0}}/\tau\text{-}\mathbf{Mod}_{\textup{harm}}^b \longrightarrow \mathcal{D}^b(\textup{MU}_*\textup{MU}\text{-}\mathbf{Comod}^\textup{ev}).$$
		
We apply a theorem of Lurie in Higher Algebra \cite[Proposition 1.3.3.7]{HA} on the relation between a stable $\infty$-category with a $t$-structure and the derived category of its heart. As a warm-up, we prove Theorem \ref{thm:mumottaumod} that there exists an equivalence of stable $\infty$-categories that preserves the given $t$-structures
		$$\textup{MU}^{\textup{mot}}/\tau\text{-}\mathbf{Mod}_{\textup{cell}}^b \longrightarrow \mathcal{D}^b(\textup{MU}_*\text{-}\mathbf{Mod}^\textup{ev}).$$		

\section{Lurie's Theorem on $t$-structures} \label{sec:tstructures}

In \cite[Proposition 1.3.3.7]{HA}, Lurie proves a theorem on the relation between a stable $\infty$-category with a $t$-structure and the derived category of its heart. In this section, we state a corollary of Lurie's theorem as Proposition~\ref{prop:projcat}, and its dual version Proposition~\ref{prop:injcat}. Both propositions are used in Section 3 and 4. We will first recall relevant definitions and Lurie's theorem and then prove Proposition~\ref{prop:projcat}.


Let $\catC$ be an $\infty$-category. Denote by $h\catC$ its homotopy category, and by $[-,-]_{\catC}$ the abelian group of homotopy classes of maps in $\catC$. When it is clear from the context, we will also denote it by $[-,-]$. If $\catC$ is further a stable $\infty$-category, denote by $\Sigma$ its translation automorphism.

We recall from \cite[Definition 1.2.1.4]{HA} that a $t$-structure on a stable $\infty$-category $\catC$ is defined as a $t$-structure on its homotopy category $h\catC$, which is a triangulated category. More precisely, we have the following definition.

\begin{defn} \label{def 2.1}
	A \emph{t-structure} on a stable $\infty$-category $\catC$ is a pair of two full subcategories $\Cgeq, \ \Cleq$ that are stable under equivalences, satisfying the following three properties
\begin{enumerate}
\item for $X \in \Cgeq$ and $Y \in \Sigma^{-1} \Cleq$, we have $[X,Y]_{\catC} = 0$,
\item there are inclusions $\Sigma \Cgeq \subseteq \Cgeq, \ \Sigma^{-1} \Cleq \subseteq \Cleq$,
\item for any $X \in \catC$, there exists a fiber sequence
\begin{equation*}
X_{\geq 0} \lto X \lto X_{\leq -1},
\end{equation*}
with $X_{\geq 0} \in \Cgeq$ and $X_{\leq -1} \in \Sigma^{-1}\Cleq$.
\end{enumerate}
\end{defn}

As in \cite{HA}, we use homological indexing convention.

\begin{defn}
Let $\catC$ and $\catC'$ be stable $\infty$-categories equipped with $t$-structures. We say that an exact functor $f:\catC \rightarrow \catC'$ is \emph{right $t$-exact}, if it carries $\catC_{\geq 0}$ to $\catC'_{\geq 0}$. An exact functor $f:\catC \rightarrow \catC'$ is \emph{left $t$-exact}, if it carries $\catC_{\leq 0}$ to $\catC'_{\leq 0}$. A functor is \emph{$t$-exact} if it is both left and right $t$-exact.
\end{defn}

\begin{defn}
Denote by $\Cgeqn$ and $\Cleqn$ the $\infty$-categories $\Sigma^n \Cgeq$ and $\Sigma^n \Cleq$ respectively. For every integer $n$, the subcategories $\Cgeqn$ and $\Cleqn$ sit in adjunctions
\begin{equation*}
\Cgeqn \overunder{}{\tau_{\geq n}}{\adj} \catC \qquad \text{and} \qquad \catC \overunder{\tau_{\leq n}}{}{\adj} \Cleqn,
\end{equation*}
where $\tau_{\geq n}$ and $\tau_{\leq n}$ are called the \emph{$n^{\text{th}}$-truncation functors}.
\end{defn}
Sometimes the truncation functors are post-composed with the inclusion functors, so they land in $\catC$.
\begin{defn}
Denote by $\catC^{+}$ and $\catC^{-}$ the stable full subcategories spanned by \emph{left-bounded} and \emph{right-bounded} objects in $\catC$:
$$\catC^+ = \bigcup_{n \geq 0} \catC_{\leq n},$$
$$\catC^- = \bigcup_{n \geq 0} \catC_{\geq -n},$$
and by $\catCb \coloneqq \catC^+ \cap \catC^-$ be the stable subcategory of \emph{bounded objects}. We say that the $t$-structure is \emph{left-bounded}, \emph{right-bounded}, or \emph{bounded}, if the inclusions of $\catC^+$, $\catC^-$ or $\catCb$ respectively, in $\catC$, is an equivalence.

The intersection $\Cheart = \Cgeq \cap \Cleq$ is called the \emph{heart} of the $t$-structure.
\end{defn}

The $\infty$-category $\Cheart$ is always equivalent to (the nerve of) its homotopy category $h\Cheart$, which is an abelian category (see Remark~1.2.1.12 in \cite{HA}). Following \cite{HA}, we abuse the notation by identifying $\Cheart$ with the abelian category $h\Cheart$.

\begin{defn}
	Let $\catC$ be a stable $\infty$-category equipped with a $t$-structure. We define the \emph{left completion} $\widehat{\catC}$ of $\catC$ to be the limit of the tower
\begin{displaymath}
	\xymatrix{
	\cdots \ar[r]  & \catC_{\leq 2} \ar[r]^{\tau_{\leq 1}} & \catC_{\leq 1} \ar[r]^{\tau_{\leq 0}} & \catC_{\leq 0} \ar[r]^{\tau_{\leq -1}} & \cdots
	}
\end{displaymath}
We say that $\catC$ is \emph{left-complete} if the functor $\catC \lto \widehat{\catC}$ is an equivalence. 
\end{defn}

By Proposition~1.2.1.17 of \cite{HA}, the left completion $\widehat{\catC}$ is again a stable $\infty$-category, inherits a $t$-structure from $\catC$, and is left-complete.

Two important examples of stable $\infty$-categories with $t$-structures are the $\infty$-category of spectra (as discussed in Section 1.4 of \cite{HA}) and the derived $\infty$-category of an abelian category (as discussed in Section 1.3 of \cite{HA}).

\begin{example}
Denote by $\mathbf{Spectra}$ the $\infty$-category of spectra and the two full subcategories
\begin{equation*}
\mathbf{Spectra}_{\geq 0} = \{ X \in \mathbf{Spectra} \st \pi_{n}X = 0 \text{ for } n < 0 \},
\end{equation*}
\begin{equation*}
\mathbf{Spectra}_{\leq 0} = \{ X \in \mathbf{Spectra} \st \pi_{n}X = 0 \text{ for } n > 0 \}.
\end{equation*}
define a $t$-structure. Left and right bounded objects correspond to connective and co-connective spectra respectively, and its heart can be identified with the abelian category of abelian groups. Moreover, as proved in \cite[Proposition 1.4.3.6]{HA}, it is left-complete.
\end{example}

\begin{example}
	Suppose that $\catA$ is an abelian category with enough projective objects. There exists an associated \emph{derived $\infty$-category} $\DD^-(\catA)$, whose objects can be identified with right-bounded chain complexes with values in $\catA$. This $\infty$-category $\DD^-(\catA)$ is stable and its homotopy category $h\DD^-(\catA)$ can be identified as the usual derived category as triangulated categories.



It admits a natural $t$-structure defined by
\begin{itemize}
\item $\DD^-(\catA)_{\geq 0}$ is the full subcategory spanned by the complexes whose homology vanishes in negative degrees,
\item $\DD^-(\catA)_{\leq 0}$ is the full subcategory spanned by the complexes whose homology vanishes in positive degrees.
\end{itemize}

As proved in \cite[Proposition 1.3.3.16]{HA}, this $t$-structure is left complete and right bounded. Moreover, as proved in \cite[Proposition 1.3.3.12]{HA}, the derived $\infty$-category $\DD^-(\catA)$ has a universal property in the sense that if $\catC$ is any stable $\infty$-category equipped with a left complete $t$-structure, then any right exact functor $\catA \rightarrow \catC^{\heartsuit}$ extends (in an essentially unique way) to a right $t$-exact functor $\DD^-(\catA) \rightarrow \catC$.
\end{example}

We have the following recognition criterion due to Lurie \cite[Proposition 1.3.3.7]{HA}.

\begin{prop} \label{prop:Lurieproj}
Let $\catC$ be a stable $\infty$-category equipped with a left complete $t$-structure, whose heart $\catA = h\Cheart$ has enough projective objects. Then there exists an essentially unique $t$-exact functor
$$F:\catD^-(\catA) \longrightarrow \catC$$
extending the inclusion $N(\catA) \simeq \catC^{\heartsuit} \subseteq \catC$. Here $N(\catA)$ is the nerve of the abelian category $\catA$.

Moreover, the following two conditions are equivalent:
\begin{itemize}
	\item The functor $F$ is fully faithful.
	\item For any pair of objects $X, Y \in \catA$, if $X$ is projective, then the abelian groups $[\Sigma^{-i} X, Y]_{\catC}$ vanish for $i>0$.
\end{itemize}
If the conditions are satisfied, then the essential image of $F$ is the full subcategory $\catC^- = \bigcup_n \catC_{\geq n}$ of right bounded objects in $\catC$.
\end{prop}

\begin{remark}
	It is clear that if we restrict the functor $F$ on the bounded stable subcategory $\catD^b(\catA)$, then it gives an equivalence of stable $\infty$-categories
	$$F:\catD^b(\catA) \longrightarrow \catC^b$$
	that preserves $t$-structures.
\end{remark}

\begin{remark}
	Lurie's theorem is exactly the reason we are working with stable $\infty$-categories instead of triangulated categories. Given a triangulated category equipped with a $t$-structure, there in general does not exist a functor from the derived category of the heart to the original triangulated category extending the identity functor on the heart (see Remark~IV~4.13 of \cite{GM} for a more detailed explanation for example). However, if the triangulated category comes from the homotopy category of a stable $\infty$-category, then such a functor always exists. Moreover, Lurie's theorem gives us a recognition criterion in terms of homological algebra to see when such a functor is also an equivalence and preserves $t$-structures.
\end{remark}

Now we use Lurie's theorem to prove the main result of this section.

\begin{prop} \label{prop:projcat}
	Let $\catC$ be a stable $\infty$-category with a given bounded $t$-structure. Suppose that
	\begin{enumerate}
	
		\item the abelian category $\catA = h\catC^{\heartsuit}$ has enough projective objects,
		\item for any pair of objects $X, Y \in \catA$, if $X$ is projective, then the abelian groups $[\Sigma^{-i} X, Y]_{\catC}$ vanish for $i>0$.
	\end{enumerate}
Then there exists an equivalence of stable $\infty$-categories
$$F:\catD^b(\catA) \longrightarrow \catC$$
extending the inclusion $N(\catA) \simeq \catC^{\heartsuit} \subseteq \catC$, and which preserves $t$-structures. Here $N(\catA)$ is the nerve of the abelian category $\catA$ and $\catD^b(\catA)$ is the bounded derived category of $\catA$. 
\end{prop}

\begin{proof}
As explained in Remark 1.2.1.18 in \cite{HA}, for any stable $\infty$-category $\catC$ with a $t$-structure, the functor $\catC \rightarrow \widehat{\catC}$ induces an equivalence
$$\catC^+ \rightarrow {(\widehat{\catC})}^{+}.$$
For the stable $\infty$-category $\catC$ with a bounded $t$-structure in the statement of Proposition~\ref{prop:projcat}, we consider its left-completion $\widehat{\catC}$ so we could apply Proposition 2.10.
Therefore, the equivalence in the statement of Proposition \ref{prop:projcat} comes from the following zigzag of equivalences:
$$\catC \longleftarrow \catC^+ \longrightarrow {(\widehat{\catC})}^{+} \longleftarrow  {(\widehat{\catC})}^{b} \longleftarrow \catD^b(\catA),$$
where the first equivalence comes from the fact that the $t$-structure on $\catC$ is bounded, the third equivalence comes from the fact that the $t$-structure on ${\widehat{\catC}}$ is right bounded since $\catC$ is, and the last equivalence comes from Lurie's theorem and Remark 2.9.
\end{proof}

Considering the opposite category, we have the following dual version of Proposition~\ref{prop:projcat}.

\begin{prop} \label{prop:injcat}
	Let $\catC$ be a stable $\infty$-category with a given bounded $t$-structure. Suppose that
	\begin{enumerate}
		
		\item the abelian category $\catA = h\catC^{\heartsuit}$ has enough injective objects,
		\item for any pair of objects $X, Y \in \catA$, if $Y$ is injective, then the abelian groups $[\Sigma^{-i} X, Y]_{\catC}$ vanish for $i>0$.
	\end{enumerate}
Then there exists an equivalence of stable $\infty$-categories
$$G:\catD^b(\catA) \longrightarrow \catC$$
extending the inclusion $N(\catA) \simeq \catC^{\heartsuit} \subseteq \catC$, and which preserves $t$-structures. Here $N(\catA)$ is the nerve of the abelian category $\catA$ and $\catD^b(\catA)$ is the bounded derived category of $\catA$.
\end{prop}


\section{An algebraic model for cellular $\textup{MU}^{\textup{mot}}/\tau$-modules} \label{subsec:MLGmod}


In this section, we use Proposition \ref{prop:projcat} to prove Theorem \ref{thm:mumottaumod}. Namely, there exists a $t$-exact equivalence of stable $\infty$-categories
		$$\textup{MU}^{\textup{mot}}/\tau\text{-}\mathbf{Mod}_{\textup{cell}}^b \longrightarrow \mathcal{D}^b(\textup{MU}_*\text{-}\mathbf{Mod}^\textup{ev}), $$
		whose restriction on the heart is given by
	$$\pi_{*,*}: \textup{MU}^{\textup{mot}}/\tau\text{-}\mathbf{Mod}_{\textup{cell}}^{\heartsuit} \longrightarrow \textup{MU}_*\text{-}\mathbf{Mod}^\textup{ev}.$$
In Section \ref{subsec:mutaumoducss}, we first recall  the universal coefficient spectral sequence in the category $\textup{MU}^{\textup{mot}}/\tau\text{-}\mathbf{Mod}_{\textup{cell}}$, which is constructed by Dugger-Isaksen \cite{DuggerIsaksen}. This is stated as Theorem \ref{thm:UCTMGL}. Using this spectral sequence, we prove the equivalence on the heart as Proposition \ref{prop:equivMGLheart} in Section 3.2. Then, using this spectral sequence again, we show in Section 3.3 that the full subcategories 
$$\textup{MU}^{\textup{mot}}/\tau\text{-}\mathbf{Mod}_{\textup{cell}}^{b, \geq 0}, \ \textup{MU}^{\textup{mot}}/\tau\text{-}\mathbf{Mod}_{\textup{cell}}^{b, \leq 0}$$
define a $t$-structure. In the end of this section, we prove the equivalence of stable $\infty$-categories as Theorem \ref{thm:equivMGLCellwithalg}.

We will use in Section \ref{subsubsec:ANSS} the above equivalence of stable $\infty$-categories to construct enough motivic spectra to build  $\textup{MU}^{\textup{mot}}/\tau$-based Adams resolutions in the category \hbox{$\widehat{S^{0,0}}/\tau\text{-}\mathbf{Mod}_{\textup{cell}}$}.

\subsection{The category $\textup{MU}^\textup{mot}/\tau\text{-}\textbf{Mod}_\textup{cell}$ and the universal coefficient spectral
sequence} \label{subsec:mutaumoducss}

We begin with two adjunctions. The first adjunction
\begin{equation} \label{eq:adjSCellCtCell}
\widehat{S^{0,0}}\text{-}\mathbf{Mod}_{\textup{cell}} \overunder{\widehat{S^{0,0}}/\tau \smas -}{U}{\ladj} \widehat{S^{0,0}}/\tau\text{-}\mathbf{Mod}_{\textup{cell}},
\end{equation}
between cellular $\widehat{S^{0,0}}$-modules and cellular $\widehat{S^{0,0}}/\tau$-modules is induced by the $E_{\infty}$-ring map $$\widehat{S^{0,0}} \lto \widehat{S^{0,0}}/\tau.$$

Since $\textup{MU}^{\textup{mot}}/\tau$ is an $E_{\infty}$-algebra that is cellular over $\widehat{S^{0,0}}/\tau$, 
 the above adjunction \eqref{eq:adjSCellCtCell} extends to
\begin{equation} \label{eq:adjSCellCtCellMGLCell}
\widehat{S^{0,0}}\text{-}\mathbf{Mod}_{\textup{cell}} \overunder{\widehat{S^{0,0}}/\tau \smas -}{U}{\ladj} \widehat{S^{0,0}}/\tau\text{-}\mathbf{Mod}_{\textup{cell}} \overunder{\textup{MU}^{\textup{mot}} \smas -}{U}{\ladj} \textup{MU}^{\textup{mot}}/\tau\text{-}\mathbf{Mod}_{\textup{cell}}.
\end{equation}

\begin{defn}
Denote by
$$\textup{MU}^{\textup{mot}}_{*,*}/\tau\text{-}\mathbf{Mod}$$
the abelian category of graded left modules over $\textup{MU}^{\textup{mot}}_{*,*}/\tau$, and by
$$\textup{MU}^{\textup{mot}}_{*,*}/\tau\text{-}\mathbf{Mod}^0$$ 
as the full subcategory of $\textup{MU}^{\textup{mot}}_{*,*}/\tau\text{-}\mathbf{Mod}$ spanned by all graded modules $M_{*,*}$ that are concentrated in Chow-Novikov degree 0, i.e.,  $M_{s,w} = 0$ whenever $s \neq 2w$.
\end{defn}

We thus have a commutative diagram
\begin{equation*}
\begin{tikzpicture}
\matrix (m) [matrix of math nodes, row sep=3em, column sep=5em]
 {\textup{MU}^{\textup{mot}}/\tau\text{-}\mathbf{Mod}_{\textup{cell}} &  \textup{MU}^{\textup{mot}}_{*,*}/\tau\text{-}\mathbf{Mod}  \\
 \textup{MU}^{\textup{mot}}/\tau\text{-}\mathbf{Mod}_{\textup{cell}}^\heartsuit & \textup{MU}^{\textup{mot}}_{*,*}/\tau\text{-}\mathbf{Mod}^0. \\};
\path[thick, -stealth, font=\small]
(m-1-1) edge node[above] {$ \piss $}  (m-1-2);
\path[thick, -stealth, font=\small]
(m-2-1) edge node[above] {$  \piss $} (m-2-2);
\path[thick, -stealth, font=\small, right hook->]
(m-2-1) edge (m-1-1)
(m-2-2) edge (m-1-2);
\end{tikzpicture}
\end{equation*}

Since $\textup{MU}^{\textup{mot}}_{*,*}/\tau$ is concentrated in Chow-Novikov degree 0, forgetting the motivic weight we have an equivalence
\begin{equation*}
\textup{MU}^{\textup{mot}}_{*,*}/\tau\text{-}\mathbf{Mod}^0 \cong \textup{MU}_*\text{-}\mathbf{Mod}^\textup{ev}.
\end{equation*}

To show that the restriction of $\piss$ to the heart 
induces an equivalence
\begin{equation*}
\piss \colon \textup{MU}^{\textup{mot}}/\tau\text{-}\mathbf{Mod}_{\textup{cell}}^\heartsuit \stackrel{\cong}{\ltoo} \textup{MU}_*\text{-}\mathbf{Mod}^\textup{ev},
\end{equation*}
we recall the universal coefficient spectral sequence constructed by Dugger-Isaksen in \cite{DuggerIsaksen}. This spectral sequence is our main tool to compute homotopy classes of maps in the stable $\infty$-category $\textup{MU}^{\textup{mot}}/\tau\text{-}\mathbf{Mod}_{\textup{cell}}$.

\begin{thm}[Universal Coefficient spectral sequence] \label{thm:UCTMGL}
For any $\ X, Y \in  \textup{MU}^{\textup{mot}}/\tau\text{-}\mathbf{Mod}_{\textup{cell}}$, there is a conditionally convergent spectral sequence
\begin{equation*}
E_2^{s,t,w}=\Ext^{s,t,w}_{\textup{MU}^{\textup{mot}}_{*,*}/\tau}(\piss X, \piss Y) \Longrightarrow \left[ \Sigma^{t-s,w} X, Y \right]_{\textup{MU}^{\textup{mot}}/\tau}.
\end{equation*}
Moreover, if both $\pi_{*,*}X$ and $\pi_{*,*}Y$ are concentrated in bounded Chow-Novikov degrees, then the spectral sequence convergences strongly and collapses at a finite page.
\end{thm}

\begin{proof}
We refer to \cite[Propositions 7.7 and 7.10]{DuggerIsaksen} for the precise construction of the spectral sequence and the proof of conditional convergence. For the second statement of the theorem, we recall a few facts from the proof of  \cite[Propositions 7.7 and 7.10]{DuggerIsaksen}.

The $E_1$-page arises from a free resolution over $\textup{MU}^{\textup{mot}}_{*,*}/\tau$:
\begin{equation*} 
0 \ltoback \piss X \surjback \piss F_0 \ltoback \piss F_1 \ltoback \cdots,
\end{equation*}
and is given by
\begin{equation*}
E_1^{s,t,w} \coloneqq \Hom_{\textup{MU}^{\textup{mot}}_{*,*}/\tau}(\piss( \Sigma^{t,w} F_s), \piss Y).
\end{equation*}
The $E_2$-page is the cohomology of this chain complex, giving the claimed $\Ext$ groups. 

Suppose that $\piss X$ and $\piss Y$ are concentrated in Chow-Novikov degrees $[a, b]$ and $[c, d]$ respectively, where $a\leq b$ and $c\leq d$. Since $\textup{MU}^{\textup{mot}}_{*,*}/\tau$ is concentrated in Chow-Novikov degree 0, we can choose all $\piss(F_s)$ such that they are concentrated in Chow-Novikov degrees $[a, b]$. Therefore $\piss(\Sigma^{t, w} F_s)$ is concentrated in Chow-Novikov degrees 
$$[a + (t-2w), \ b + (t-2w)]$$ 
for all $s \geq 0$.

In order for the group $E_1^{s,t,w}$ to be nonzero, we must have 
$$c \leq b + (t-2w), \ \ d \geq a + (t-2w).$$

For a fixed weight $w$, this gives that
\begin{equation*}
t \in [c-b + 2w, d-a + 2w].
\end{equation*}
Since later pages $E_r^{s,t,w}$ are iterated subquotients of $E_1^{s,t,w}$, their $t$-degrees are all concentrated in $[c-b + 2w, d-a + 2w]$. 

Recall that the $d_r$-differential has the form
\begin{equation*}
E_r^{s,t,w} \stackrel{d_r}{\lto} E_r^{s+r, t+r-1,w}.
\end{equation*}
In particular, it changes the $t$-degrees by $r-1$. Since the $t$-degrees of all possible nonzero elements in the $E_1$-page satisfy $t \in [c-b + 2w, d-a + 2w]$, we must have $d_r = 0$ when
$$r-1 > (d-a+2w) - (c-b+2w) = (b-a)+(d-c)$$
for degree reasons. In other words, the spectral sequence collapses at the $E_{(b-a)+(d-c)+2}$ page. 

Therefore, under the condition that both $\pi_{*,*}X$ and $\pi_{*,*}Y$ are concentrated in bounded Chow-Novikov degrees, this spectral sequence convergences strongly and collapses at a finite page.
\end{proof}

Recall from Definition \ref{def:thicksub} that 
$$\textup{MU}^{\textup{mot}}/\tau\text{-}\mathbf{Mod}_{\textup{cell}}^{b, \geq 0}, $$ $$ \textup{MU}^{\textup{mot}}/\tau\text{-}\mathbf{Mod}_{\textup{cell}}^{b, \leq 0}, $$ $$ \textup{MU}^{\textup{mot}}/\tau\text{-}\mathbf{Mod}_{\textup{cell}}^{\heartsuit}$$ 
are the full subcategories of $\textup{MU}^{\textup{mot}}/\tau$-$\mathbf{Mod}_{\textup{cell}}^b$ that are spanned by objects whose homotopy groups are concentrated in nonnegative, nonpositive and zero Chow-Novikov degrees respectively.

\begin{cor}\label{C47}
Given $ X\in\textup{MU}^{\textup{mot}}/\tau$-$\mathbf{Mod}_{\textup{cell}}^{b, \geq 0}$ and $ Y\in\textup{MU}^{\textup{mot}}/\tau$-$\mathbf{Mod}_{\textup{cell}}^{b, \leq 0}$. The abelian group of homotopy classes of bidegree $(0, 0)$ can be computed algebraically by the isomorphism
$$ [X,Y]_{{\textup{MU}^{\textup{mot}}/\tau}} \longrightarrow \textup{Hom}_{\textup{MU}^{\textup{mot}}{*,*}/\tau}(\pi_{*,*}X,\pi_{*,*}Y)$$
that is induced by applying $\piss$.
\end{cor}

\begin{proof}
Consider the the $E_2$-page of the universal coefficient spectral sequence, the tri-degrees that converge to the bidegree $(0, 0)$ are of the form $(t, t, 0)$ for $t \geq 0$, i.e., the terms $E_2^{s, t, w} = E_2^{t, t, 0}$.

By the proof of Theorem \ref{thm:UCTMGL}, the $t$-degrees of all possible nonzero elements in the $E_1$-page and therefore $E_2$-page satisfy $t \leq d-a + 2w = d-a$. Since $\pi_{*,*}X$ and $\pi_{*,*}Y$ are concentrated in nonnegative and nonpositive bounded Chow-Novikov degrees, we have $d=a=0$. Therefore, we have $t \leq 0$.

Combining both facts, we have established that the only possible nonzero elements in the $E_2$-page that converge to the bidegree $(0, 0)$ are in 
$$E_2^{0,0,0} = \textup{Hom}_{\textup{MU}^{\textup{mot}}{*,*}/\tau}(\pi_{*,*}X,\pi_{*,*}Y).$$
To show that all elements in $E_2^{0,0,0}$ survive in the spectral sequence, first note that they are not targets of any nonzero differentials since they are in $s$-degree $0$. Second, all $d_r$-differentials for $r \geq 2$ increase the $t$-degree. Since the $t$-degrees of all nonzero elements are non-positive, the elements in $E_2^{0,0,0}$ do not support nonzero differentials. This completes the proof.
\end{proof}



\subsection{The equivalence on the heart}

Now we are ready to show that the functor $\piss$ induces an equivalence on the heart. The following is a special case of Corollary \ref{C47}.

\begin{cor} \label{cor:MGLCellheartff}
The functor
\begin{equation*} 
\piss \colon    \textup{MU}^{\textup{mot}}/\tau\text{-}\mathbf{Mod}_{\textup{cell}}^\heartsuit \ltoo \textup{MU}^\textup{mot}_{*,*}/\tau\text{-}\mathbf{Mod}^0
\end{equation*}
is fully faithful. Here the right hand side is understood as a discrete $\infty$-category.
\end{cor}

As a consequence, Corollary~\ref{cor:MGLCellheartff} shows that $\textup{MU}^{\textup{mot}}/\tau\text{-}\mathbf{Mod}_{\textup{cell}}^\heartsuit$ is also a discrete $\infty$-category.

\begin{proof}
For $n \geq 0$ and two objects $X,Y \in \textup{MU}^{\textup{mot}}/\tau\text{-}\mathbf{Mod}_{\textup{cell}}^\heartsuit$, 
by Corollary \ref{C47}, the edge homomorphism
\begin{equation*}
\left[ \Sigma^{n,0} X, Y \right]_{\textup{MU}^{\textup{mot}}/\tau} \stackrel{\piss}{\ltoo} \Hom_{\textup{MU}^{\textup{mot}}_*/\tau}(\piss \Sigma^{n,0} X, \piss Y)
\end{equation*}
is an isomorphism. When $n>0$, the bigraded module $\piss \Sigma^{n,0} X$ is concentrated in positive Chow-Novikov degree. So the right hand side of the above isomorphism is concentrated in the case $n=0$. This shows that $\piss$ is fully faithful on $\textup{MU}^{\textup{mot}}/\tau\text{-}\mathbf{Mod}_{\textup{cell}}^\heartsuit$.
\end{proof}

To show the equivalence on the heart, we only need to show the essential surjectivity of $\piss$. 

\begin{prop} \label{prop:equivMGLheart}
The functor
\begin{equation*}
\piss \colon \textup{MU}^{\textup{mot}}/\tau\text{-}\mathbf{Mod}_{\textup{cell}}^\heartsuit  \stackrel{}{\ltoo} \textup{MU}^\textup{mot}_{*,*}/\tau\text{-}\mathbf{Mod}^0
\end{equation*}
is an equivalence of $\infty$-categories.
\end{prop}

\begin{proof}
We need to show that any object $M \in \textup{MU}^\textup{mot}_{*,*}/\tau\text{-}\mathbf{Mod}^0$ can be realized as the homotopy groups of an object in $\textup{MU}^{\textup{mot}}/\tau\text{-}\mathbf{Mod}_{\textup{cell}}^\heartsuit $. 

Suppose that $M$ is a free $\textup{MU}^{\textup{mot}}_{*,*}/\tau$-module that is concentrated in Chow-Novikov degree 0.  
$$M \cong \bigoplus_{i \in I} \Sigma^{2k_i,k_i}\textup{MU}^{\textup{mot}}_{*,*}/\tau$$
Here $\Sigma^{2k_i, k_i}\textup{MU}^{\textup{mot}}_{*,*}/\tau$ is a free bigraded rank 1 module over $\textup{MU}^{\textup{mot}}_{*,*}/\tau$ with a generator in bidegree $(2k_i, k_i)$. We can realize $M$ as the homotopy groups of the wedge 
$$\bigvee_{i \in I} \Sigma^{2k_i,k_i}\textup{MU}^{\textup{mot}}/\tau$$ 
with the same index set, which is cellular.

For an arbitrary $M \in \textup{MU}^\textup{mot}_{*,*}/\tau\text{-}\mathbf{Mod}^0$, we can pick a free resolution
\begin{equation} \label{eqpf:resofMbyfree}
0 \ltoback M \ltoback F_0 \stackrel{f_1}{\ltoback} F_1 \stackrel{f_2}{\ltoback} F_2 \ltoback \cdots
\end{equation}
in $\textup{MU}^\textup{mot}_{*,*}/\tau\text{-}\mathbf{Mod}^0$. 

Each $F_i$ can be realized by  
$$Z_i \in \textup{MU}^{\textup{mot}}/\tau\text{-}\mathbf{Mod}_{\textup{cell}}^\heartsuit$$ and by Corollary \ref{cor:MGLCellheartff}, each map $f_i$ can be realized by a map $g_i \in \textup{MU}^{\textup{mot}}/\tau\text{-}\mathbf{Mod}_{\textup{cell}}^\heartsuit$ as in 
\begin{equation*}
Z_0 \stackrel{g_1}\ltoback Z_1 \stackrel{g_2}\ltoback Z_2 \ltoback \cdots.
\end{equation*}
Using the standard method of lifting an Adams resolution to an Adams tower, we claim that we can construct a tower
\begin{equation*} 
X_1 \lto X_2 \lto \cdots,
\end{equation*}
with the property that the homotopy groups of $X_i$ are given by the following groups
\begin{equation*}
\bigoplus_{l= -\infty}^{+\infty}\pi_{2l + k, l} (X_i) = \left\{
  \begin{array}{ll}
    M = \textup{Coker} f_1 & \text{ if } k = 0 \\
    \Sigma^{i, 0}\textup{Ker} f_i & \text{ if } k = i \\
    0 & \text{ otherwise}.
  \end{array}
\right.
\end{equation*}
and that the maps in this tower induce an isomorphism on the Chow-Novikov degree 0 part of these homotopy groups, which is $M$.


We prove this claim inductively.

In fact, we can choose $X_1$ to be the cofiber of
$$g_1: Z_1 \longrightarrow Z_0.$$
This gives us a long exact sequence on homotopy groups
\begin{displaymath}
	\xymatrix{
	\cdots \ar[r] & \pi_{*+1,*}X_1 \ar[r] & \piss Z_1 \ar[r]^{f_1} & \piss Z_0 \ar[r] & \piss X_1 \ar[r] & \cdots.
	}
\end{displaymath}

Since both $\piss Z_0$ and $\piss Z_1$ are concentrated in Chow-Novikov degree 0, we must have that $\piss X_1$ is concentrated in Chow-Novikov degree 0 and 1. We can compute directly from the long exact sequence that the Chow-Novikov degree 0 and 1 parts of $\piss X_1$ are isomorphic to $M = \textup{Coker} f_1$ and $\Sigma^{1, 0}\textup{Ker} f_1$ respectively.

Suppose now that we have constructed the tower up to $X_{i}$. We have a homomorphism
\begin{equation*}
\piss Z_{i+1} \cong F_{i+1} \surj \textup{Im}f_{i+1} \cong \textup{Ker}f_{i} \inj \piss(\Sigma^{-i,0} X_{i})
\end{equation*}
in $\textup{MU}^\textup{mot}_{*,*}/\tau\text{-}\mathbf{Mod}^0$. Here the first map is induced by $f_{i+1}$ and the second map corresponds to the Chow-Novikov degree $i$ part of $\piss X_i$.

By Corollary \ref{C47}, this homomorphism can be realized as a map 
$$Z_{i+1} \lto \Sigma^{-i,0} X_{i}.$$ 
Define $X_{i+1}$ as the $\Sigma^{i,0}$-suspension of its cofiber, so we have a cofiber sequence
\begin{equation*}
\Sigma^{i,0} Z_{i+1} \lto X_{i} \lto X_{i+1}.
\end{equation*}

By the associated long exact sequence in homotopy groups, we have that $\piss X_{i+1}$ is concentrated in Chow-Novikov degrees 0 and $i+1$. The Chow-Novikov degree 0 part is isomorphic to $M$, and the Chow-Novikov degree $i+1$ part is isomorphic to $\Sigma^{i+1, 0}\textup{Ker} f_{i+1}$ as required.

Having the tower 
\begin{equation*} 
X_1 \lto X_2 \lto \cdots,
\end{equation*}
we define $X$ as its colimit
\begin{equation*}
X \coloneqq \colim \left( X_1 \lto X_2 \lto \cdots \right).
\end{equation*}
The homotopy groups of $X$ are computed by the colimit
\begin{equation*}
\piss X \cong \colim \left( \piss X_1 \lto \piss X_2 \lto \cdots \right) = M
\end{equation*}
and are in particular concentrated in Chow-Novikov degree 0. 

Therefore we have proved that any module $M \in \textup{MU}^\textup{mot}_{*,*}/\tau\text{-}\mathbf{Mod}^0$ can be realized as a spectrum $X \in \textup{MU}^{\textup{mot}}/\tau\text{-}\mathbf{Mod}_{\textup{cell}}^\heartsuit$.
\end{proof}

\subsection{The $t$-structure, and the equivalence of categories} \label{subsubsec:tstructforMGLmod}

We prove in this subsection that the full subcategories previously defined satisfy the required axioms for the $t$-structure.

We first prove a general proposition regarding the existence of a $t$-structure.

\begin{prop} \label{sim t str}
Let $\mathcal{C}$ be a stable $\infty$-category, $\mathcal{C}^{\geq0}$ and $\mathcal{C}^{\leq0}$ be a pair of full subcategories of $\mathcal{C}$. 
Let $\mathcal{C}^{\geq n} = \Sigma^n \mathcal{C}^{\geq0}$, $\mathcal{C}^{\leq n} = \Sigma^{n}\mathcal{C}^{\leq0}$.
Suppose that
\begin{enumerate}
\item $\mathcal{C}^{\geq0}$ is closed under extensions and suspensions.
$\mathcal{C}^{\leq0}$ is closed under extensions and desuspensions.

\item 
$$\mathcal{C} = \bigcup_{n \in \mathbb{Z}} \mathcal{C}^{\geq n}.$$

\item
For any $X\in \mathcal{C}^{\geq0}$ and $Y\in \mathcal{C}^{\leq -1}$, we have
$$[X,Y]_{\mathcal{C}} = 0$$

\item 
For any $X\in \mathcal{C}^{\geq 0}$, there is an object $X_0\in \mathcal{C}^{\geq0}\cap \mathcal{C}^{\leq0}$ and a morphism
$$ X\longrightarrow X_0$$ 
such that its fiber lies in $\mathcal{C}^{\geq1}$.
\end{enumerate}
Then this pair of subcategories $\mathcal{C}^{\geq0}$ and $\mathcal{C}^{\leq0}$ defines a $t$-structure on $\mathcal{C}$.
\end{prop}

\begin{proof}
In view of Definition~\ref{def 2.1}, it remains to check that for any $X\in \mathcal{C}$, there exists a fiber sequence
$$X_{\geq0}\longrightarrow X \longrightarrow X_{\leq -1}$$
such that
$X_{\geq 0}\in \mathcal{C}^{\geq0}$ and $X_{\leq-1} \in \mathcal{C}^{\leq-1}$.

By the second assumption, there exists $n$ such that $X\in \mathcal{C}^{\geq n}$. 

If $n\geq 0$, we can take the fiber sequence to be
$$X\longrightarrow X \longrightarrow *$$
If $n<0$, using the fourth assumption there exits an object $X_n\in \mathcal{C}^{\geq n}\cap \mathcal{C}^{\leq n}$ and a fiber sequence
$$X_{\geq n+1}\longrightarrow X\longrightarrow X_n$$
such that $X_{\geq n+1} \in \mathcal{C}^{\geq n+1}$.

For $n < 0$, iterating this process, we get a finite sequence of morphsims
$$X_{\geq 0}\rightarrow X_{\geq -1}\rightarrow\dots \rightarrow X_{\geq n+1}\rightarrow X$$
such that $X_{\geq i} \in \mathcal{C}^{\geq i}$ and that the cofiber $X_{i-1}$ of the morphism
$$X_{\geq i}\longrightarrow X_{\geq i-1}$$
lies in $\mathcal{C}^{\geq i-1}\cap \mathcal{C}^{\leq i-1}$, where $i \in [n+1, 0]$.

Now we can take $X_{\leq -1}$ to be the cofiber of the morphism
$$X_{\geq 0}\longrightarrow X.$$
The object $X_{\leq -1}$ can be build up by finite extensions from the objects $X_{i-1}$ for\\ $i \in [n+1, 0]$. Therefore, by the first assumption, we have $X_{\leq -1} \in \mathcal{C}^{\leq -1}$.
\end{proof}

\begin{prop}
\label{prop:tstructonMGLCellb}
The pair of full subcategories 
$$\textup{MU}^{\textup{mot}}/\tau\text{-}\mathbf{Mod}_{\textup{cell}}^{b, \geq 0}, \ \textup{MU}^{\textup{mot}}/\tau\text{-}\mathbf{Mod}_{\textup{cell}}^{b, \leq 0}$$ 
defines a bounded $t$-structure on 
$$\textup{MU}^{\textup{mot}}/\tau\text{-}\mathbf{Mod}_{\textup{cell}}^b.$$
\end{prop}

\begin{proof}
We check the four conditions in Proposition~\ref{sim t str}. 

The first two conditions follow
directly from the definition of the Chow-Novikov degree.

For the third condition,
we need to show that
$$[X,Y]_{\textup{MU}^{\textup{mot}}/\tau} = 0$$
for any objects $$X \in \textup{MU}^{\textup{mot}}/\tau\text{-}\mathbf{Mod}_{\textup{cell}}^{b, \geq 0}, \ Y \in \Sigma^{-1,0} \textup{MU}^{\textup{mot}}/\tau\text{-}\mathbf{Mod}_{\textup{cell}}^{b, \leq 0}.$$

By Corollary~\ref{C47}, we have that
\begin{equation*}
[X,Y]_{\textup{MU}^{\textup{mot}}/\tau} \cong \Hom_{\textup{MU}^{\textup{mot}}_{*,*}/\tau}(\piss X, \piss Y)
\end{equation*}
Since $\piss X$ is concentrated in non-negative Chow-Novikov degrees and $\piss Y$ is concentrated in negative Chow-Novikov degrees, the right hand side is zero.


For the fourth condition, we need to show that for any 
$$X \in \textup{MU}^{\textup{mot}}/\tau\text{-}\mathbf{Mod}_{\textup{cell}}^{b, \geq 0},$$ 
there exists a fiber sequence
\begin{equation*}
X_{\geq 1} \lto X \lto X_0
\end{equation*}
such that 
$$X_{0} \in \textup{MU}^{\textup{mot}}/\tau\text{-}\mathbf{Mod}_{\textup{cell}}^{\heartsuit}, \ X_{\geq 1} \in \Sigma^{1, 0} (\textup{MU}^{\textup{mot}}/\tau\text{-}\mathbf{Mod}_{\textup{cell}}^{b, \geq 0}).$$




Consider the Chow-Novikov degree 0 part of $\piss (X)$, namely 
$$\pi_{*,*} (X)^{=0} := \bigoplus_k \pi_{2k, k}(X).$$
By Proposition \ref{prop:equivMGLheart} there is a spectrum 
$$X_0 \in \textup{MU}^{\textup{mot}}/\tau\text{-}\mathbf{Mod}_{\textup{cell}}^\heartsuit$$ realizing this bigraded $\textup{MU}^{\textup{mot}}_{*,*}/\tau$-module
$$\piss X_0 \cong \pi_{*,*} (X)^{=0}.$$
Consider the projection map
\begin{equation*}
\piss(X) \surj \pi_{*,*} (X)^{=0} \cong  \piss X_0.
\end{equation*}
By Corollary \ref{C47}, the projection map can be realized by a map 
$$ X \lto X_0.$$ 
Denote by $X_{\geq 1}$ its fiber. Then from the long exact sequence in homotopy groups, we have that 
$$X_{\geq 1} \in \Sigma^{1, 0} (\textup{MU}^{\textup{mot}}/\tau\text{-}\mathbf{Mod}_{\textup{cell}}^{b, \geq 0}).$$ 
\end{proof}

Having this $t$-structure on $\textup{MU}^{\textup{mot}}/\tau\text{-}\mathbf{Mod}_{\textup{cell}}^b$, the main result of this section follows from Proposition \ref{prop:projcat}.

\begin{thm} \label{thm:equivMGLCellwithalg}
There is a $t$-exact equivalence of stable $\infty$-categories
\begin{equation*}
\mathcal{D}^b(\textup{MU}_*\text{-}\mathbf{Mod}^\textup{ev}) \stackrel{\cong}{\ltoo}\textup{MU}^{\textup{mot}}/\tau\text{-}\mathbf{Mod}_{\textup{cell}}^b.
\end{equation*}
\end{thm}

\begin{proof}

By Proposition~\ref{prop:tstructonMGLCellb}, the $t$-structure is bounded.  By Proposition~\ref{prop:equivMGLheart}, and the equivalence 
\begin{equation*}
\textup{MU}^{\textup{mot}}_{*,*}/\tau\text{-}\mathbf{Mod}^0 \cong \textup{MU}_*\text{-}\mathbf{Mod}^\textup{ev},
\end{equation*}
the heart can be identified as modules over $\textup{MU}_*$. Therefore, it has enough projective objects. 

It remains to show that for any two objects $$X,Y \in \textup{MU}^{\textup{mot}}/\tau\text{-}\mathbf{Mod}_{\textup{cell}}^\heartsuit$$ 
with $\piss X$ projective over $\textup{MU}^{\textup{mot}}_{*,*}/\tau$, we have that
$$[\Sigma^{-i,0}X, Y]_{\textup{MU}^{\textup{mot}}/\tau}=0$$ 
for $i > 0$. 

We apply the Universal Coefficient spectral sequence in Theorem \ref{thm:UCTMGL}, 
\begin{equation*}
\Ext_{\textup{MU}^{\textup{mot}}_{*,*}/\tau}^{s,t,w}(\piss X, \piss Y) \Longrightarrow [ \Sigma^{t-s,w}X,Y]_{\textup{MU}^{\textup{mot}}/\tau}.
\end{equation*}

Since $\piss X$ is projective over $\textup{MU}^{\textup{mot}}_{*,*}/\tau$, the $E_2$-page of the spectral sequence is concentrated on the line $s=0$, and therefore collapses at the $E_2$-page.

Moreover, since both $\piss X$ and $\piss Y$ are concentrated in Chow-Novikov degree 0,
the $E_2$-page is also concentrated in Chow-Novikov degree 0, namely $t - 2w = 0$ in this case. 

We are interested in the case $t-s = -i<0$ and $w=0$. By the above analysis, the corresponding tri-degrees in the $E_2$-page are all 0 in our case. Therefore, we must have that
$$[\Sigma^{-i,0}X, Y]_{\textup{MU}^{\textup{mot}}/\tau}=0.$$
This completes the proof.
\end{proof}


\section{An algebraic model for harmonic $\widehat{S^{0,0}}/\tau$-modules} \label{subsec:Ctmod}

After the warmup in Section 3, we use Proposition \ref{prop:injcat} to prove Theorem \ref{thm:ctaumod}. Namely, there exists a $t$-exact equivalence of stable $\infty$-categories
		$$\widehat{S^{0,0}}/\tau\text{-}\mathbf{Mod}_{\textup{harm}}^b \longrightarrow \mathcal{D}^b(\textup{MU}_*\textup{MU}\text{-}\mathbf{Comod}^\textup{ev}),$$
		whose restriction on the heart is given by
	$$\textup{MU}^{\textup{mot}}_{*,*}: \widehat{S^{0,0}}/\tau\text{-}\mathbf{Mod}_{\textup{harm}}^{\heartsuit} \longrightarrow \textup{MU}_*\textup{MU}\text{-}\mathbf{Comod}^\textup{ev}.$$
The structure of this section is similar to that of Section 3. 
	
In Section \ref{subsec:ctauharmodanss}, we discuss the category of harmonic $\widehat{S^{0,0}}/\tau$-modules. We will also recall certain facts on the category of $\textup{MU}_*\textup{MU}$-comodules, such as Landweber's Filtration Theorem. Instead of using the universal coefficient spectral sequence in the category $\textup{MU}^{\textup{mot}}/\tau\text{-}\mathbf{Mod}_{\textup{cell}}$, we will use the absolute Adams-Novikov spectral sequence in the category of harmonic $\widehat{S^{0,0}}/\tau$-modules. This spectral sequence is constructed in Section 5. Using this spectral sequence, we prove the equivalence on the heart as Proposition \ref{prop:CtCellequivheart} in Section 4.2. Then again using this spectral sequence, we show in Section 4.3 that the full subcategories 
$$\widehat{S^{0,0}}/\tau\text{-}\mathbf{Mod}_{\textup{harm}}^{b,\geq0}, \  \widehat{S^{0,0}}/\tau\text{-}\mathbf{Mod}_{\textup{harm}}^{b,\leq0}$$
define a $t$-structure and conclude the equivalence of stable $\infty$-categories as Proposition \ref{prop:tstrCtCellb} and Theorem \ref{thm:equivCtCellcompalg}.


\subsection{The categories $\widehat{S^{0,0}}/\tau\text{-}\mathbf{Mod}_{\textup{harm}}$ and $\textup{MU}_*\textup{MU}\text{-}\mathbf{Comod}^\textup{ev}$} \label{subsec:ctauharmodanss}


We first recall from Definition~\ref{def:pisscomplete} that a $\widehat{S^{0,0}}/\tau$-module spectrum $Y$ is \emph{harmonic} if it is $\widehat{S^{0,0}}/\tau$-cellular and the natural map 
$$Y \lto Y^{\smas}_{\textup{MU}^{\textup{mot}}}$$ 
is an isomorphism on $\piss$. As pointed out after Definition~\ref{def:pisscomplete}, the two completions $X^{\wedge}_{\text{MGL}}$ in the category $\C\text{-}\mathbf{mot}\text{-}\mathbf{Spectra}$ and $X^{\wedge}_{\textup{MU}^{\textup{mot}}}$ in the category $\widehat{S^{0,0}}$-$\mathbf{Mod}$ are equivalent for any $X$ in $\widehat{S^{0,0}}$-$\mathbf{Mod}_{\textup{cell}}$. It is clear that in the category $\widehat{S^{0,0}}/\tau\text{-}\mathbf{Mod}_{\textup{cell}}$, being harmonic is closed under taking suspensions, finite products and fibers. The category of harmonic $\widehat{S^{0,0}}/\tau$-module spectra is denoted by $\widehat{S^{0,0}}/\tau\text{-}\mathbf{Mod}_{\textup{harm}}$.


We have the following examples and non-examples of harmonic $\widehat{S^{0,0}}/\tau$-module spectra.

\begin{example}\leavevmode
	\begin{enumerate}
		\item Any finite cellular object in $\widehat{S^{0,0}}/\tau$-$\mathbf{Mod}$ is harmonic. In fact, by discussion in Section~7 of Dugger-Isaksen \cite{DuggerIsaksenMASS}, the H$\mathbb{F}_p^\textup{mot}$-completed sphere $\widehat{S^{0,0}}$ is MGL-complete and is therefore ${\textup{MU}^{\textup{mot}}}$-complete. Then the claim follows from an induction argument.
				\item Any finite cellular object in $\textup{MU}^{\textup{mot}}/\tau$-$\mathbf{Mod}$ is harmonic.
		\item The $\eta$-inverted cofiber of $\tau$ is $\widehat{S^{0,0}}/\tau$-cellular but \emph{not} harmonic. \\ Here $\eta$ is the Hopf map in $\pi_{1,1}\widehat{S^{0,0}}$. Post-composing with the unit map $\widehat{S^{0,0}} \rightarrow \widehat{S^{0,0}}/\tau$, we also denote its Hurewicz image in $\pi_{1,1}\widehat{S^{0,0}}/\tau$ by $\eta$. It is non-nilpotent in the ring $\pi_{*,*}\widehat{S^{0,0}}/\tau$. One way to see this fact is to identify $\pi_{*,*}\widehat{S^{0,0}}/\tau$ as $\Ext^{*,*}_{\textup{BP}_*\textup{BP}}(\textup{BP}_*, \textup{BP}_*)$ by Gheorghe-Isaksen and to use the fact from Miller-Ravenel-Wilson \cite{MRW} that the element that detects $\eta$ in $\Ext^{*,*}_{\textup{BP}_*\textup{BP}}(\textup{BP}_*, \textup{BP}_*)$ is non-nilpotent. The $\eta$-inverted cofiber of $\tau$
\begin{equation*}
\eta^{-1} \widehat{S^{0,0}}/\tau \coloneqq \textup{colim} \Big( \widehat{S^{0,0}}/\tau \rightarrow \Sigma^{-1,-1}\widehat{S^{0,0}}/\tau \rightarrow \Sigma^{-2,-2}\widehat{S^{0,0}}/\tau \rightarrow \cdots \Big)
\end{equation*}
is a cellular object in $\widehat{S^{0,0}}/\tau$-$\mathbf{Mod}$. Since $\eta$ maps to zero in $\pi_{*,*}\textup{MU}^{\textup{mot}}$, we have that $$\textup{MU}^{\textup{mot}}_{*,*}(\eta^{-1} \widehat{S^{0,0}}/\tau) = 0.$$
Therefore, the completion $({\eta^{-1} \widehat{S^{0,0}}/\tau})^{\wedge}_{\textup{MU}^{\textup{mot}}}$ is contractible. This shows that the spectrum $\eta^{-1} \widehat{S^{0,0}}/\tau$ is not harmonic.
	\end{enumerate}
\end{example}

The following Lemma~\ref{lemma:pisscompclosedunderubfilcolim} will be used in the proof of Proposition~\ref{prop:CtCellequivheart}. We postpone the proof of Lemma~\ref{lemma:pisscompclosedunderubfilcolim} until the end of Section 5.

\begin{lemma} \label{lemma:pisscompclosedunderubfilcolim} 
Suppose that $\{Y_{\alpha}\}$ is a filtered system in $\widehat{S^{0,0}}/\tau\text{-}\mathbf{Mod}_{\textup{cell}}^\heartsuit$ such that each $Y_{\alpha}$ is harmonic. Then the colimit of $\{Y_{\alpha}\}$ in $\widehat{S^{0,0}}/\tau\text{-}\mathbf{Mod}_{\textup{cell}}^\heartsuit$ is also harmonic.
\end{lemma}

Recall that for a $\widehat{S^{0,0}}/\tau$-module $X$, its
$\textup{MU}^{\textup{mot}}$-homology can also be described as
\begin{equation*}
\textup{MU}_{*,*}^{\textup{mot}} X = \piss( \textup{MU}^{\textup{mot}} \smas_{\widehat{S^{0,0}}} X) \cong \piss(\textup{MU}^{\textup{mot}}/\tau  \smas_{\widehat{S^{0,0}}/\tau} X).
\end{equation*}

Following computations of $\textup{MU}^{\textup{mot}}_{*,*}\textup{MU}^{\textup{mot}}$ from Hu-Kriz-Ormsby \cite{HKOrem} and Dugger-Isaksen \cite{DuggerIsaksenMASS}, we have the $\textup{MU}^{\textup{mot}}$-homology of $\textup{MU}^{\textup{mot}}/\tau$
\begin{equation*}
\pi_{*,*}(\textup{MU}^{\textup{mot}}/\tau\wedge_{\widehat{S^{0,0}}/\tau}\textup{MU}^{\textup{mot}}/\tau) \cong \textup{MU}^{\textup{mot}}_{*,*}/\tau [b_1, b_2, \ldots] \cong  {\textup{MU}^{\textup{mot}}_{*,*}\textup{MU}^{\textup{mot}}}/\tau
\end{equation*}
where $|b_i| = (2i,i)$, and is in Chow-Novikov degree 0. Since $\tau$ can be realized as a map $\widehat{S^{0,-1}} \rightarrow \widehat{S^{0,0}}$, it is primitive in $\textup{MU}^\textup{mot}$. Therefore, ${\textup{MU}^{\textup{mot}}_{*,*}\textup{MU}^{\textup{mot}}}/\tau$ is a Hopf algebroid.

\begin{defn}
Denote by 
$$\textup{MU}^{\textup{mot}}_{*,*}\textup{MU}^{\textup{mot}}/\tau\text{-}\mathbf{Comod}$$ 
the abelian category of graded left comodules over the Hopf algebroid $\textup{MU}^{\textup{mot}}_{*,*}\textup{MU}^{\textup{mot}}/\tau$, and by 
$$\textup{MU}^{\textup{mot}}_{*,*}\textup{MU}^{\textup{mot}}/\tau\text{-}\mathbf{Comod}^0$$ 
its full subcategory spanned by all graded comodules
whose underlying $\textup{MU}^{\textup{mot}}_{*,*}/\tau$-modules are concentrated in Chow-Novikov degree 0.
\end{defn}

We therefore have a commutative diagram
\begin{equation*}
\begin{tikzpicture}
\matrix (m) [matrix of math nodes, row sep=3em, column sep=5em]
{ \widehat{S^{0,0}}/\tau\text{-}\mathbf{Mod}_{\textup{harm}} & \textup{MU}^{\textup{mot}}_{*,*}\textup{MU}^{\textup{mot}}/\tau\text{-}\mathbf{Comod} \\
  \widehat{S^{0,0}}/\tau\text{-}\mathbf{Mod}_{\textup{harm}}^\heartsuit & \textup{MU}^{\textup{mot}}_{*,*}\textup{MU}^{\textup{mot}}/\tau\text{-}\mathbf{Comod}^0. \\};
\path[thick, -stealth, font=\small]
(m-1-1) edge node[above] {$ \textup{MU}^{\textup{mot}}_{*,*} $}  (m-1-2)
(m-2-1) edge node[above] {$  \textup{MU}^{\textup{mot}}_{*,*}$} (m-2-2);
\path[thick, -stealth, font=\small, right hook->/]
(m-2-1) edge (m-1-1)
(m-2-2) edge (m-1-2);
\end{tikzpicture}
\end{equation*}

Forgetting the motivic weight, we have the equivalence 
\begin{equation*}
\textup{MU}^{\textup{mot}}_{*,*}\textup{MU}^{\textup{mot}}/\tau\text{-}\mathbf{Comod}^0 \cong \textup{MU}_*\textup{MU}\text{-}\mathbf{Comod}^\textup{ev}.
\end{equation*}

Recall that we have the adjunction between modules and comodules
\begin{equation} \label{eq:adjModComod}
U \colon \textup{MU}^{\textup{mot}}_{*,*}\textup{MU}^{\textup{mot}}/\tau\text{-}\mathbf{Comod} \ladj \textup{MU}^{\textup{mot}}_{*,*}/\tau\text{-}\mathbf{Mod} \colon  \textup{MU}^{\textup{mot}}_{*,*}\textup{MU}^{\textup{mot}}/\tau\otimes_{\textup{MU}^{\textup{mot}}_{*,*}/\tau}-.
\end{equation}
The forgetful functor is a \emph{left} adjoint, while the tensor-up functor is a \emph{right} adjoint.
We refer to \cite[Section 1.1]{Hovey} for more details.

Using the ring map $\widehat{S^{0,0}}/\tau \lto \textup{MU}^{\textup{mot}}/\tau$, we can form the commutative diagram
\begin{equation*}
\begin{tikzpicture}
\matrix (m) [matrix of math nodes, row sep=3em, column sep=5em]
 {\textup{MU}^{\textup{mot}}/\tau\text{-}\mathbf{Mod}_{\textup{cell}} &  \textup{MU}^{\textup{mot}}_{*,*}/\tau\text{-}\mathbf{Mod} \\
  \widehat{S^{0,0}}/\tau\text{-}\mathbf{Mod}_{\textup{cell}}&  \textup{MU}^{\textup{mot}}_{*,*}\textup{MU}^{\textup{mot}}/\tau\text{-}\mathbf{Comod}. \\};
\path[thick, -stealth, font=\small]
(m-2-1) edge node[left] {$\textup{MU}^{\textup{mot}}/\tau \wedge_{\widehat{S^{0,0}}/\tau} - $} (m-1-1)
(m-2-2) edge node[right] { $\textup{U}$ }(m-1-2)
(m-1-1) edge node[above] {$ \piss $}  (m-1-2)
(m-2-1) edge node[above] {$   \textup{MU}^{\textup{mot}}_{*,*} $} (m-2-2);
\end{tikzpicture}
\end{equation*}
For the category of comodules over $\textup{MU}_{*}\textup{MU}$, we recall the Landweber's Filtration Theorem. Recall from \cite{Landweber1, Landweber2} that there are elements $v_n \in \textup{MU}_{\ast}$ with $v_0=p$, giving the invariant prime ideals $I_n = (v_0, \ldots, v_n) \unlhd \textup{MU}_*$. Moreover, these elements satisfy the formula
\begin{equation*}
\eta_R(v_n)\equiv v_n \mod I_{n-1},
\end{equation*}
and so $\textup{MU}_*/I_n$ is canonically a comodule over $\textup{MU}_*\textup{MU}$. This gives a short exact sequence of comodules
\begin{equation*}
0 \lto \textup{MU}_* /I_{n} \stackrel{\cdot v_n}{\lto} \textup{MU}_* /I_{n} \lto \textup{MU}_* /I_{n+1} \lto 0,
\end{equation*}
for every $n \geq 0$. Landweber's Filtration Theorem (\cite{Landweber1}, \cite{Landweber2}) states that any comodule $M$ over $\textup{MU}_{*}\textup{MU}$ whose underlying $\textup{MU}_*$-module is finitely presented, can be reconstructed by finitely many extensions of suspensions of $\textup{MU}_*/I_{n}$'s.\\

\begin{thm}[Landweber's Filtration Theorem] 
Suppose that $S$ is a class of comodules over $\textup{MU}_{*}\textup{MU}$ such that 
\begin{enumerate}
	\item it contains  $\textup{MU}_*$ and
 $\textup{MU}_*/I_n$'s for all $n\geq 0$,
 \item and it is closed under suspensions and extensions.
\end{enumerate}   
Then $S$ contains all comodules over $\textup{MU}_{*}\textup{MU}$ whose underlying $\textup{MU}_*$-modules are finitely presented.
\end{thm}

There are two more facts that we will use on the category of comodules over $\textup{MU}_{*}\textup{MU}$. The first one is the following lemma. For a proof, see Miller-Ravenel \cite[Lemma 2.11]{MillerRavenel} and Hovey \cite{Hovey} for example.

\begin{lemma} \label{lem:mumucolim}
	Any comodule over $\textup{MU}_*\textup{MU}$ is a filtered colimit of finitely presented comodules.
\end{lemma}

The second one is a consequence of the fact that the forgetful functor from\\ $\textup{MU}_*\textup{MU}\text{-}\mathbf{Comod}^\textup{ev}$ to $\textup{MU}_*\text{-}\mathbf{Mod}^\textup{ev}$ has a right adjoint. For a precise argument, see the proof of Lemma~\ref{lemma:injresalg} for its motivic analogue. 

\begin{lemma} \label{lem:injmumu}
	The category $\textup{MU}_*\textup{MU}\text{-}\mathbf{Comod}^\textup{ev}$ has enough injective objects.
\end{lemma}

We will construct the absolute Adams-Novikov spectral sequence, namely, for any two objects $X$ and $Y$ in this category, there is a strongly convergent spectral sequence that collapses at a finite page.
\begin{equation*}
\Ext^{s,t,w}_{\textup{MU}^{\textup{mot}}_{*,*} \textup{MU}^{\textup{mot}}/\tau}(\textup{MU}^{\textup{mot}}_{*,*} X, \textup{MU}^{\textup{mot}}_{*,*} Y) \Longrightarrow [\Sigma^{t-s, w} X, Y ]_{\widehat{S^{0,0}}/\tau},
\end{equation*}
with differentials
$$d_r \colon E_r^{s, t, w} \lto E_r^{s+r, t+r-1, w}.$$
The existence of this absolute Adams-Novikov spectral sequence in the category $\widehat{S^{0,0}}/\tau\text{-}\mathbf{Mod}_{\textup{harm}}^b$ is proved as Theorem \ref{cor:ANss} in Section 5.


Using the absolute Adams-Novikov spectral sequence,  we will prove the following Corollaries \ref{cor:Cttstruct1} and \ref{cor:L45} in Section 5.3.

\begin{cor} \label{cor:Cttstruct1}
For $X \in \widehat{S^{0,0}}/\tau\text{-}\mathbf{Mod}_{\textup{harm}}^{b,\geq0}$ and $Y \in \widehat{S^{0,0}}/\tau\text{-}\mathbf{Mod}_{\textup{harm}}^{b,\leq0}$, 
the following map induced by applying the functor $\textup{MU}^\textup{mot}_{*,*}$ is an isomorphism.
\begin{equation*}
[X,Y]_{\widehat{S^{0,0}}/\tau} \longrightarrow \Hom_{\textup{MU}^{\textup{mot}}_{*,*}\textup{MU}^{\textup{mot}}/\tau}( \textup{MU}^{\textup{mot}}_{*,*} X,\textup{MU}^{\textup{mot}}_{*,*} Y ).
\end{equation*}
\end{cor}

\begin{cor}\label{cor:L45}
Given $X, Y \in \widehat{S^{0,0}}/\tau\text{-}\mathbf{Mod}_{{\textup{harm}}}^\heartsuit$, for any bidegree $(t,w)$, there is an isomorphism
\begin{equation*}
[\Sigma^{t,w} X,Y]_{\widehat{S^{0,0}}/\tau} \cong \Ext^{2w -t, 2w, w}_{\textup{MU}^{\textup{mot}}_{*,*}\textup{MU}^{\textup{mot}}/\tau}(\textup{MU}^{\textup{mot}}_{*,*}X, \textup{MU}^{\textup{mot}}_{*,*}Y).
\end{equation*}
\end{cor}

\subsection{The equivalence on the heart} 

Now we are ready to show that the functor $\textup{MU}^\textup{mot}_{*,*}$ induces an equivalence on the heart. The following is a special case of Corollary \ref{cor:Cttstruct1}.

\begin{cor} \label{cor:CtCellheartff}
The functor
\begin{equation*} 
\textup{MU}^{\textup{mot}}_{*,*} \colon  \widehat{S^{0,0}}/\tau\text{-}\mathbf{Mod}_{\textup{harm}}^\heartsuit  \stackrel{}{\ltoo} \textup{MU}^{\textup{mot}}_{*,*}\textup{MU}^{\textup{mot}}/\tau\text{-}\mathbf{Comod}^0
\end{equation*}
is fully faithful. Here the right hand side is understood as a discrete $\infty$-category.
\end{cor}

As a consequence, Corollary~\ref{cor:CtCellheartff} shows that $\widehat{S^{0,0}}/\tau\text{-}\mathbf{Mod}_{\textup{harm}}^\heartsuit$ is also a discrete $\infty$-category.

\begin{proof}
For $n \geq 0$ and two objects $X,Y \in \widehat{S^{0,0}}/\tau\text{-}\mathbf{Mod}_{\textup{harm}}^\heartsuit$, by Corollary~\ref{cor:Cttstruct1}, the edge homomorphism
\begin{equation*}
\left[\Sigma^{n,0} X, Y \right]_{\widehat{S^{0,0}}/\tau} \stackrel{\textup{MU}^{\textup{mot}}_{*,*}}{\ltoo} \Hom_{\textup{MU}^{\textup{mot}}_{*,*}\textup{MU}^{\textup{mot}}/\tau}\left( \textup{MU}^{\textup{mot}}_{*,*} \Sigma^{n,0}X, \textup{MU}^{\textup{mot}}_{*,*} Y \right)
\end{equation*}
is an isomorphism. When $n>0$, the bigraded module $\textup{MU}^{\textup{mot}}_{*,*} \Sigma^{n,0}X$ is concentrated in positive Chow-Novikov degree. So the right hand side of the above isomorphism is concentrated in the case $n=0$. This shows that $\textup{MU}^{\textup{mot}}_{*,*}$ is fully faithful on $\widehat{S^{0,0}}/\tau\text{-}\mathbf{Mod}_{\textup{harm}}^\heartsuit$.
\end{proof}

To show the equivalence on the heart, we only need to show the essential surjectivity of $\textup{MU}^{\textup{mot}}_{*,*}$.

Unlike the case for modules over $\textup{MU}^{\textup{mot}}_{*,*}/\tau$, we do not have free resolutions for comodules over $\textup{MU}^{\textup{mot}}_{*,*}\textup{MU}^{\textup{mot}}/\tau$. We will instead use Landweber's Filtration Theorem to realize all comodules that are finitely presented, and then extend the result using filtered colimits. In particular, all Smith-Toda complexes exist in $\widehat{S^{0,0}}/\tau\text{-}\mathbf{Mod}$. 

We start with the following 2-out-of-3 Lemma.

\begin{lemma} \label{lemma:CtCellesssurj2outof3}
Consider any short exact sequence in $\textup{MU}^{\textup{mot}}_{*,*}\textup{MU}^{\textup{mot}}/\tau\text{-}\mathbf{Comod}^0$
\begin{equation} \label{eq:pfheartses}
0 \lto M' \stackrel{f'}{\lto} M \stackrel{f''}{\lto} M'' \lto 0.
\end{equation}
If any two of the three comodules $M', \ M, \ M''$ are realizable in $\widehat{S^{0,0}}/\tau\text{-}\mathbf{Mod}_{\textup{harm}}^\heartsuit$, then so is the third.
\end{lemma}

\begin{proof}
There are three cases that we need to prove.
\begin{enumerate}
	\item Suppose that both comodules $M'$ and $M$ are realizable by
$$M' \cong \textup{MU}^{\textup{mot}}_{*,*}X', \ M \cong \textup{MU}^{\textup{mot}}_{*,*}X.$$  
By Corollary \ref{cor:CtCellheartff}, the algebraic map $f'$ is also realizable as the $\textup{MU}^{\textup{mot}}_{*,*}$-homology of a map
\begin{displaymath}
	\xymatrix{
X' \ar[r]^{F'} & X.
} 
\end{displaymath}
Since $\widehat{S^{0,0}}/\tau\text{-}\mathbf{Mod}_{\textup{harm}}^b$ is closed under taking cofibers, we can realize the comodule $M''$ by the $\textup{MU}^{\textup{mot}}_{*,*}$-homology of the cofiber of $F'$. In fact, the associated long exact sequence on the  $\textup{MU}^{\textup{mot}}_{*,*}$-homology tells us  
$$M'' \cong \textup{MU}^{\textup{mot}}_{*,*}X'',$$
where $X''$ is the the cofiber of $F'$.

\item Suppose that both comodules $M$ and $M''$ are realizable. Then we realize the algebraic map and take the fiber instead. The same argument shows that it realizes $M'$.

\item Suppose that both comodules $M'$ and $M''$ are realizable by
$$M' \cong \textup{MU}^{\textup{mot}}_{*,*}X', \ M'' \cong \textup{MU}^{\textup{mot}}_{*,*}X''.$$ 
In this case, the short exact sequence \eqref{eq:pfheartses} corresponds to an element in 
$$\Ext^{1,0,0}_{\textup{MU}^{\textup{mot}}_{*,*}\textup{MU}^{\textup{mot}}/\tau}(M'',M').$$ 
By Corollary \ref{cor:L45}, this algebraic element can be realized by a map 
$$F : \Sigma^{-1,0} X'' \lto X'.$$ 
Define $X$ to be the cofiber of the map $F$. We claim that $X$ realizes $M$. 
In fact, the map $F$ has Adams-Novikov filtration 1, so it induces the zero homomorphism on $\textup{MU}^{\textup{mot}}_{*,*}$. Therefore, the cofiber sequence that defines $X$ induces a short exact sequence on $\textup{MU}^{\textup{mot}}_{*,*}$. Since the isomorphism in Corollary~\ref{cor:L45} comes from the collapse of the absolute Adams-Novikov spectral sequence on the $E_2$-page, this shows that this short exact on $\textup{MU}^{\textup{mot}}_{*,*}$ is isomorphic to the the short exact sequence \eqref{eq:pfheartses}. Therefore,  
$$M \cong \textup{MU}^{\textup{mot}}_{*,*}X.$$
\end{enumerate}
This completes the proof.
\end{proof}

Now we prove the equivalence on the heart.

\begin{prop} \label{prop:CtCellequivheart}
The functor
\begin{equation*}
\textup{MU}^{\textup{mot}}_{*,*} \colon  \widehat{S^{0,0}}/\tau\text{-}\mathbf{Mod}_{\textup{harm}}^\heartsuit \stackrel{\cong}{\ltoo} \textup{MU}^{\textup{mot}}_{*,*}\textup{MU}^{\textup{mot}}/\tau\text{-}\mathbf{Comod}^0
\end{equation*}
is an equivalence of categories.
\end{prop}

\begin{proof}
We only need to show that the functor $\textup{MU}^{\textup{mot}}_{*,*}$ is essentially surjective. In other words, for any comodule $M \in \textup{MU}^{\textup{mot}}_{*,*}\textup{MU}^{\textup{mot}}/\tau\text{-}\mathbf{Comod}^0$, we show that it can be realized as a harmonic $\widehat{S^{0,0}}/\tau$-module $X$, whose $\textup{MU}^{\textup{mot}}_{*,*}$-homology is $M$. This follows from Lemmas \ref{lemma:CtCellesssurj2outof3}, \ref{lem:mumucolim}, \ref{lemma:pisscompclosedunderubfilcolim} and Landweber's Filtration Theorem via the equivalence
\begin{equation*}
\textup{MU}^{\textup{mot}}_{*,*}\textup{MU}^{\textup{mot}}/\tau\text{-}\mathbf{Comod}^0 \cong \textup{MU}_*\textup{MU}\text{-}\mathbf{Comod}^\textup{ev}.
\end{equation*}


In fact, $\textup{MU}_*$ corresponds $\textup{MU}^{\textup{mot}}_{*,*}/\tau$, and is therefore realized by $\widehat{S^{0,0}}/\tau$. By Lemma \ref{lemma:CtCellesssurj2outof3}, we can inductively realize comodules $\textup{MU}_*/I_n$ for all $n \geq 0$. Then by Landweber's Filtration Theorem and  
Lemma \ref{lemma:CtCellesssurj2outof3}, we can realized all finitely presented comodules. 

For any comodule $M \in \textup{MU}^{\textup{mot}}_{*,*}\textup{MU}^{\textup{mot}}/\tau\text{-}\mathbf{Comod}^0$, or equivalently a comodule over $\textup{MU}_*\textup{MU}$ that is concentrated in even degrees, by Lemma~\ref{lem:mumucolim}, we can write it as a filtered colimit of finitely presented ones $M_{\alpha}$,
$$M \cong \colim M_{\alpha}.$$
By the above discussion, we can realize each $M_{\alpha}$ by $X_{\alpha} \in \widehat{S^{0,0}}/\tau\text{-}\mathbf{Mod}_{\textup{harm}}^\heartsuit$. Moreover, by Corollary \ref{cor:CtCellheartff}, we can realize the whole filtered system $\{ M_{\alpha} \}$ by a filtered system $\{ X_{\alpha} \}$. Taking the colimit, we define
$$X \coloneqq \colim X_{\alpha}.$$
By Lemma \ref{lemma:pisscompclosedunderubfilcolim}, $X$ is harmonic. Since $\textup{MU}^{\textup{mot}}_{*,*}$ commutes with filtered colimits, we have that the comodule $M$ is realized by $X$. This completes the proof.
\end{proof}

\subsection{The $t$-structure and the equivalence of categories} \label{subsubsec:CtCellprooftstruct}

We prove that two full subcategories satisfy the required axioms for the $t$-structure.

\begin{prop} \label{prop:tstrCtCellb}
The pair of full subcategories $$\widehat{S^{0,0}}/\tau\text{-}\mathbf{Mod}_{\textup{harm}}^{b,\geq0}, \  \widehat{S^{0,0}}/\tau\text{-}\mathbf{Mod}_{\textup{harm}}^{b,\leq0}$$ 
defines a bounded $t$-structure on $\widehat{S^{0,0}}/\tau\text{-}\mathbf{Mod}_{\textup{harm}}^b$.
\end{prop}
{
\begin{proof}
The proof is exactly analogous to the proof of Proposition~\ref{prop:tstructonMGLCellb}, with Corollary~\ref{cor:Cttstruct1} replacing Corollary~\ref{C47}, and Proposition~\ref{prop:CtCellequivheart} replacing Proposition~\ref{prop:equivMGLheart}.

\end{proof}

Having this $t$-structure on $\widehat{S^{0,0}}/\tau\text{-}\mathbf{Mod}_{{\textup{harm}}}^b$, the main result of this section follows from Proposition \ref{prop:injcat}.


\begin{thm} \label{thm:equivCtCellcompalg}
There is a $t$-exact equivalence of stable $\infty$-categories
\begin{equation*}
\mathcal{D}^b(\textup{MU}_*\textup{MU}\text{-}\mathbf{Comod}^\textup{ev}) \stackrel{\cong}{\ltoo}\widehat{S^{0,0}}/\tau\text{-}\mathbf{Mod}_{\textup{harm}}^b.
\end{equation*}
\end{thm}

\begin{proof}
The proof is analogous to the proof of Theorem \ref{thm:equivMGLCellwithalg}. It is clear that the $t$-structure is bounded.  By Proposition \ref{prop:CtCellequivheart}, and the equivalence 
\begin{equation*}
\textup{MU}^{\textup{mot}}_{*,*}\textup{MU}^{\textup{mot}}/\tau\text{-}\mathbf{Comod}^0 \cong \textup{MU}_*\textup{MU}\text{-}\mathbf{Comod}^\textup{ev},
\end{equation*}
the heart can be identified as comodules over $\textup{MU}_*\textup{MU}$. By Lemma \ref{lem:injmumu}, it has enough injective objects.

It remains to show that for objects 
$$X, \ Y \in \widehat{S^{0,0}}/\tau\text{-}\mathbf{Mod}_{\textup{harm}}^\heartsuit$$
 with $\textup{MU}^{\textup{mot}}_{*,*}Y$ injective over $\textup{MU}^{\textup{mot}}_{*,*}\textup{MU}^{\textup{mot}}/\tau$, we have that
$$[\Sigma^{-i,0}X, Y]_{\widehat{S^{0,0}}/\tau}=0$$ 
for any $i > 0$.

We apply the absolute Adams-Novikov spectral sequence 
\begin{equation*}
\Ext^{s,t,w}_{\textup{MU}^{\textup{mot}}_{*,*}\textup{MU}^{\textup{mot}}/\tau}\left( \textup{MU}^{\textup{mot}}_{*,*}X, \textup{MU}^{\textup{mot}}_{*,*}Y \right) \Longrightarrow \left[\Sigma^{t-s,w} X,Y \right]_{\widehat{S^{0,0}}/\tau}
\end{equation*}
in the category $\widehat{S^{0,0}}/\tau\text{-}\mathbf{Mod}_{\textup{harm}}^b$, as in Corollary \ref{cor:ANss}.

Since $\textup{MU}^{\textup{mot}}_{*,*}Y$ is an injective $\textup{MU}^{\textup{mot}}_{*,*}\textup{MU}^{\textup{mot}}/\tau$-comodule, the $E_2$-page of the spectral sequence is concentrated on the line $s=0$, and therefore collapses at the $E_2$-page.

Moreover, since both $\textup{MU}^{\textup{mot}}_{*,*}X$ and $\textup{MU}^{\textup{mot}}_{*,*}Y$ are concentrated in Chow-Novikov degree 0,
the $E_2$-page is also concentrated in Chow-Novikov degree 0, namely $t - 2w = 0$ in this case.

We are interested in the case $t-s = -i<0$ and $w=0$. By the above analysis, the corresponding tri-degrees in the $E_2$-page are all 0 in our case. Therefore, we must have that
$$[\Sigma^{-i,0}X, Y]_{\widehat{S^{0,0}}/\tau}=0.$$ 
This completes the proof.
\end{proof}

\begin{remark} \label{rem:bigrading}
	We comment on the bi-grading in the equivalence of stable $\infty$-categories in Theorem \ref{thm:equivCtCellcompalg} through some examples.
	\begin{enumerate}
		\item 
	It is clear that $\widehat{S^{0,0}}/\tau$ corresponds	 to $\textup{MU}_*$ in the derived category of $\textup{MU}_*\textup{MU}$-comodules. 
\item	
	Consider $\Sigma^{2,1}\widehat{S^{0,0}}/\tau$. Since its $\textup{MU}^{\textup{mot}}$-homology is concentrated in Chow-Novikov degree 0, it lives in the heart. Therefore, by the $t$-exactness, it corresponds to a cochain complex that is concentrated in cohomological degree 0. A direct computation show that it corresponds to $\Sigma^{2}\textup{MU}_*$. We also denote this object in the category $\mathcal{D}^b(\textup{MU}_*\textup{MU}\text{-}\mathbf{Comod}^\textup{ev})$ by $\Sigma^{2,1}\textup{MU}_*$.
	\item Consider $\Sigma^{1,0}\widehat{S^{0,0}}/\tau$. Its $\textup{MU}^{\textup{mot}}$-homology is concentrated in Chow-Novikov degree $1$. By the $t$-exactness, it corresponds to the cochain complex that is concentrated in cohomological degree $-1$, with the comodule $\textup{MU}_*$ in that cohomological degree. We also denote this object by $\Sigma^{1,0}\textup{MU}_*$.
	\item In general, denote by $\Sigma^{m,n}\textup{MU}_*$ the object in the category $\mathcal{D}^b(\textup{MU}_*\textup{MU}\text{-}\mathbf{Comod}^\textup{ev})$ that $\Sigma^{m,n}\widehat{S^{0,0}}/\tau$ corresponds to. Then $\Sigma^{m,n}\textup{MU}_*$ is a cochain complex that is concentrated in cohomological degree $2n-m$, with the comodule $\Sigma^{2n}\textup{MU}_*$ in that cohomological degree.
	\end{enumerate}

\end{remark}

By Proposition~\ref{prop:bpbpmumucomod}, there exists an exact equivalence of categories between\\
${\textup{BP}_*\textup{BP}\text{-}\mathbf{Comod}}^{\textup{ev}}$ and $ \textup{MU}_*\textup{MU}\text{-}\mathbf{Comod}^\textup{ev}$. Therefore, Theorem~\ref{thm:equivCtCellcompalg} implies Theorem~1.1.

\begin{remark} \label{symmo}
	The equivalence of stable $\infty$-categories in Theorem~1.1 is actually symmetric monoidal. In fact, the equivalence preserves colimits, so we have the following commutative diagram
\begin{displaymath}
\xymatrix{
\mathcal{D}^{-}({{\textup{BP}_*\textup{BP}\text{-}\mathbf{Comod}}^{\textup{ev}}})_{\geq 0} \ar[r] & (\widehat{S^{0,0}}/\tau\text{-}\mathbf{Mod}_{\textup{harm}}^{b,\geq 0})^\wedge \\
\mathbf{s}(\textup{BP}_*\textup{BP}\text{-}\mathbf{Comod}^\textup{ev}_{\textup{rel proj}}) \ar[r] \ar[u]^{F_1} & \mathbf{s}(\widehat{S^{0,0}}/\tau\text{-}\mathbf{Mod}_{\textup{harm}}^{\heartsuit}) \ar[u]^{F_2}
}	
\end{displaymath}
Here $(\widehat{S^{0,0}}/\tau\text{-}\mathbf{Mod}_{\textup{harm}}^{b,\geq 0})^\wedge$ is the left completion of $\widehat{S^{0,0}}/\tau\text{-}\mathbf{Mod}_{\textup{harm}}^{b, \geq 0}$ with respective to its t-structure, $\mathbf{s}(\textup{BP}_*\textup{BP}\text{-}\mathbf{Comod}^\textup{ev}_{\textup{rel proj}})$ is the category of simplicial objects of relative projective $\textup{BP}_*\textup{BP}$-comodules that are concentrated in even degrees (see \cite[Definition~2.1.2]{Hovey}), and $\mathbf{s}(\widehat{S^{0,0}}/\tau\text{-}\mathbf{Mod}_{\textup{harm}}^{\heartsuit})$ is the category of simplicial objects in $\widehat{S^{0,0}}/\tau\text{-}\mathbf{Mod}_{\textup{harm}}^{\heartsuit}$. The horizontal arrows are induced by the equivalence in Theorem~1.1, and the vertical arrows are geometric realizations.

By Proposition~5.5.9.14 of \cite{HTT}, geometric realization functor is symmetric monoidal. The lower horizontal arrow is also symmetric monoidal, since it is level wise symmetric monoidal, which is implied by Proposition~\ref{prop:CtCellequivheart} and the universal coefficient theorem (see \cite[Proposition~7.10]{DI04}).

Let $W_1$ and $W_2$ denote the class of morphisms in the $\infty$-categories $\mathbf{s}(\textup{BP}_*\textup{BP}\text{-}\mathbf{Comod}^\textup{ev}_{\textup{rel proj}})$ and $\mathbf{s}(\widehat{S^{0,0}}/\tau\text{-}\mathbf{Mod}_{\textup{harm}}^{\heartsuit})$ that are sent to equivalences by the functors $F_1$ and $F_2$. Let $\mathbf{s}(\textup{BP}_*\textup{BP}\text{-}\mathbf{Comod}^\textup{ev}_{\textup{rel proj}})[W_1^{-1}]$ and $\mathbf{s}(\widehat{S^{0,0}}/\tau\text{-}\mathbf{Mod}_{\textup{harm}}^{\heartsuit})[W_2^{-1}]$ denote the localizations with respect to $W_1$ and $W_2$. The above commutative diagram factors through the following one.
\begin{displaymath}
\xymatrix{
\mathcal{D}^{-}({{\textup{BP}_*\textup{BP}\text{-}\mathbf{Comod}}^{\textup{ev}}})_{\geq 0} \ar[r] & (\widehat{S^{0,0}}/\tau\text{-}\mathbf{Mod}_{\textup{harm}}^{b,\geq 0})^\wedge \\
\mathbf{s}(\textup{BP}_*\textup{BP}\text{-}\mathbf{Comod}^\textup{ev}_{\textup{rel proj}})[W_1^{-1}] \ar[r] \ar[u] & \mathbf{s}(\widehat{S^{0,0}}/\tau\text{-}\mathbf{Mod}_{\textup{harm}}^{\heartsuit})[{W_2}^{-1}] \ar[u]
}	
\end{displaymath}

By the Dold-Kan correspondence and Lemma~1.4.6 of \cite{Hovey}, the left vertical arrow is an equivalence.

By Proposition~2.2.1.9 of \cite{HA}, localization functor preserves symmetric monoidal structure. Therefore, the upper horizontal arrow is symmetric monoidal. Restricting it to the bounded subcategory and by the universal property of stablization \cite[Theorem~2.14]{Robalo}, this gives us the claim.  	
	
\end{remark}

Now we prove Corollary 1.2.

\begin{proof}[Proof of Corollary 1.2]
Let $\widehat{S^{0,0}}/\tau\text{-}\mathbf{Mod}_\textup{fin}$ be the category of finite cellular motivic left module spectra over $\widehat{S^{0,0}}/\tau$, and $\mathcal{D}^b({{\textup{BP}_*\textup{BP}\text{-}\mathbf{Comod}}^{\textup{ev}}})_\textup{fin}$ be the full subcategory of\\ $\mathcal{D}^b({{\textup{BP}_*\textup{BP}\text{-}\mathbf{Comod}}^{\textup{ev}}})$ consisting of objects generated by $\textup{BP}_*$, under finite colimits and shifts by both homological and even internal degrees. 

Since $\widehat{S^{0,0}}/\tau$ is harmonic, the category $\widehat{S^{0,0}}/\tau\text{-}\mathbf{Mod}_\textup{fin}$ is the full subcategory of\\ $\widehat{S^{0,0}}/\tau\text{-}\mathbf{Mod}_{\textup{harm}}^{b}$ consisting of objects generated by $\widehat{S^{0,0}}/\tau$, under finite colimits and shifts by both the topological degree and the motivic weight.

Since $\widehat{S^{0,0}}/\tau$ corresponds to $\textup{BP}_*$ under the equivalence in Theorem \ref{intro:mainthmpart1}, we have an equivalence of stable $\infty$-categories  with given $t$-structures at each prime $p$
$$\mathcal{D}^b({{\textup{BP}_*\textup{BP}\text{-}\mathbf{Comod}}^{\textup{ev}}})_\textup{fin}\simeq \widehat{S^{0,0}}/\tau\text{-}\mathbf{Mod}_\textup{fin}.$$

By Theorem 5.3.5.11 of Lurie's Higher Topos Theory \cite{HTT}, if $\mathcal{D}$ is an $\infty$-category that admits filtered colimits, and $\mathcal{C}$ is an essentially small full subcategory of $\mathcal{D}$, whose elements are compact, and generate $\mathcal{D}$ under filtered colimits, then $\mathcal{D}$ is equivalent to the $\infty$-category $\textup{Ind}(\mathcal{C})$ of $\textup{Ind}$-objects of $\mathcal{C}$.

It follows that  
$$\widehat{S^{0,0}}/\tau\text{-}\mathbf{Mod}_\textup{cell}\simeq \textup{Ind}(\widehat{S^{0,0}}/\tau\text{-}\mathbf{Mod}_\textup{fin}).$$

On the other hand, $\textup{BP}_*$ generates \hbox{$\mathcal{D}^b({\textup{BP}_*\textup{BP}\text{-}\mathbf{Comod}}^{\textup{ev}})_\textup{fin}$} under finite colimits. Moreover, it is proved by Hovey in \cite[Section 6]{Hovey} that objects in the category\\ \hbox{$\mathcal{D}^b({\textup{BP}_*\textup{BP}\text{-}\mathbf{Comod}}^{\textup{ev}})_\textup{fin}$} are compact, and generate $\mathbf{Stable}(\textup{BP}_*\textup{BP}\text{-}\mathbf{Comod}^\textup{ev})$ under filtered colimits. It then follows from Theorem 5.3.5.11 of \cite{HTT} that 
$$\mathbf{Stable}(\textup{BP}_*\textup{BP}\text{-}\mathbf{Comod}^\textup{ev})\simeq \textup{Ind}(\mathcal{D}^b({\textup{BP}_*\textup{BP}\text{-}\mathbf{Comod}}^{\textup{ev}})_\textup{fin}).$$

Therefore, we have an equivalence of stable $\infty$-categories at each prime $p$
$$\mathbf{Stable}(\textup{BP}_*\textup{BP}\text{-}\mathbf{Comod}^\textup{ev})\simeq \widehat{S^{0,0}}/\tau\text{-}\mathbf{Mod}_\textup{cell}.$$
\end{proof}

\section{The absolute Adams-Novikov spectral sequence} \label{subsubsec:ANSS}

In Section \ref{subsec:MLGmod}, we used the universal coefficient spectral sequence 
\begin{equation*}
E_2^{s,t,w}=\Ext^{s,t,w}_{\textup{MU}^{\textup{mot}}_{*,*}/\tau}(\piss X, \piss Y) \Longrightarrow \left[ \Sigma^{t-s,w} X, Y \right]_{\textup{MU}^{\textup{mot}}/\tau}
\end{equation*}
of Theorem \ref{thm:UCTMGL} to compute homotopy classes of maps in $\textup{MU}^{\textup{mot}}/\tau\text{-}\mathbf{Mod}_{\textup{cell}}^b$. This is a very convenient tool since both the $t$-structure on $\textup{MU}^{\textup{mot}}/\tau\text{-}\mathbf{Mod}_{\textup{cell}}^b$ and the $E_2$-page of the universal coefficient spectral sequence are defined in terms of homotopy groups. The bounds in the $t$-structure correspond to vanishing areas in the spectral sequence. 

For the category $\widehat{S^{0,0}}/\tau\text{-}\mathbf{Mod}_{\textup{harm}}^b$, the $t$-structure is defined in terms of $\textup{MU}^{\textup{mot}}$-homology. We therefore need a version of the motivic Adams-Novikov spectral sequence that computes $\widehat{S^{0,0}}/\tau$-linear maps.

Recall from Dugger-Isaksen \cite[Section 8]{DuggerIsaksenMASS} or Hu-Kriz-Ormsby \cite{HKOrem} the usual $\textup{MU}^{\textup{mot}}$-based motivic Adams-Novikov spectral sequence
\begin{equation*}
\Ext^{*,*,*}_{\textup{MU}^{\textup{mot}}_{*,*}\textup{MU}^{\textup{mot}}}(\textup{MU}^{\textup{mot}}_{*,*}\widehat{S^{0,0}}, \textup{MU}^{\textup{mot}}_{*,*}Y) \Longrightarrow \piss Y^{\smas}_{\textup{MU}^{\textup{mot}}}. 
\end{equation*}

This spectral sequence is not what we need. We need a spectral sequence of the form
\begin{equation*}
\Ext_{\textup{MU}^{\textup{mot}}_{*,*}\textup{MU}^{\textup{mot}}/\tau}(\textup{MU}^{\textup{mot}}_{*,*}X, \textup{MU}^{\textup{mot}}_{*,*}Y) \Longrightarrow \left[ X, Y^{\smas}_{\textup{MU}^{\textup{mot}}} \right]_{\widehat{S^{0,0}}/\tau},
\end{equation*}
for the following two reasons. $Y^{\smas}_{\textup{MU}^{\textup{mot}}}$ has the natural structure of being a $\widehat{S^{0,0}}/\tau$-module. (See property (3) after Definition~\ref{d71} for this fact.)

First, we need a spectral sequence computing homotopy classes of maps in the category $\widehat{S^{0,0}}/\tau\text{-}\mathbf{Mod}_{\textup{cell}}$, instead of homotopy classes of maps between the underlying motivic spectra.

Second, we need the first variable $X$ to be a general cellular $\widehat{S^{0,0}}/\tau$-module than just the unit object $\widehat{S^{0,0}}/\tau$. Classically, when the first variable $X$ is the sphere spectrum, we can use the standard cosimplicial cobar Adams-Novikov resolution for the second variable $Y \in \widehat{S^{0,0}}/\tau\text{-}\mathbf{Mod}$ to set up this spectral sequence. This is done in \cite[Chapter 2]{Ravenel} classically and in \cite[Section 8]{DuggerIsaksenMASS} and \cite{HKOrem} motivically. Such a resolution induces a resolution of $\textup{MU}^{\textup{mot}}_{*,*}Y$ by relative injective comodules. It computes the $E_2$-page as an $\Ext$-group only when the first variable $\textup{MU}^{\textup{mot}}_{*,*}X$ is a projective module over $\textup{MU}^{\textup{mot}}_{*,*}/\tau$ \cite[Corollary A1.2.12]{Ravenel}. Since our first variable $X$ is arbitrary, the $E_2$-page in general does not have a description as a relative $\Ext$-group.


Instead of using the canonical Adams-Novikov tower that produces a resolution of $\textup{MU}^{\textup{mot}}_{*,*}Y$ by relative injectives, we construct an $\mathbf{absolute}$ Adams-Novikov tower that produces a resolution of $\textup{MU}^{\textup{mot}}_{*,*}Y$ by $\mathbf{absolute}$ injectives. The first step is Lemma \ref{lemma:injrestop1} and Lemma \ref{lemma:injrestop2}, where we produce enough $\widehat{S^{0,0}}/\tau$-modules whose $\textup{MU}^{\textup{mot}}$-homology are injective comodules.
The second step is Lemma \ref{lemma:injresalg}, where we show that we can algebraically resolve comodules in $\textup{MU}^{\textup{mot}}_{*,*}\textup{MU}^{\textup{mot}}/\tau\text{-}\mathbf{Comod}$ by these injective comodules. The third step is Proposition \ref{prop:ANtower}, where we topologically realize the algebraic construction to produce an absolute Adams-Novikov tower in the category $\widehat{S^{0,0}}/\tau\text{-}\mathbf{Mod}_{\textup{cell}}^b$. Finally, in Theorem \ref{thm:ANss}, we construct the absolute Adams-Novikov spectral sequence and analyze its convergence. 
{\color{red}
}

\subsection{The absolute Adams-Novikov tower}
We construct the absolute Adams-Novikov tower in this subsection.

Recall that the forgetful functor from the abelian category $\textup{MU}^{\textup{mot}}_{*,*}\textup{MU}^{\textup{mot}}/\tau\text{-}\mathbf{Comod}$ to the category of abelian groups reflects monomorphisms, epimorphisms and exactness.

The following Lemma \ref{lemma:injrestop1} is a consequence of Proposition \ref{prop:equivMGLheart} and the homology version of Dugger-Isaksen's the universal coefficient spectral sequence \cite[Propositions 7.7 and 7.10]{DuggerIsaksen}. 

\begin{lemma} \label{lemma:injrestop1}
For any injective module $N \in \textup{MU}^{\textup{mot}}_{*,*}/\tau\text{-}\mathbf{Mod}^0$, we have
\begin{enumerate}
\item 	$\textup{MU}^{\textup{mot}}_{*,*}\textup{MU}^{\textup{mot}}/\tau \otimes_{\textup{MU}^{\textup{mot}}_{*,*}/\tau} N$ is an injective $\textup{MU}^{\textup{mot}}_{*,*}\textup{MU}^{\textup{mot}}/\tau$-comodule,
\item there exists $I$ in $\textup{MU}^{\textup{mot}}/\tau\text{-}\mathbf{Mod}_{\textup{cell}}^\heartsuit$ such that
$$\piss I \cong N,$$
\item for any such an object $I$, 
$$\textup{MU}^{\textup{mot}}_{*,*} I \cong \textup{MU}^{\textup{mot}}_{*,*}\textup{MU}^{\textup{mot}}/\tau \otimes_{\textup{MU}^{\textup{mot}}_{*,*}/\tau} N.$$
\end{enumerate}

\end{lemma}

\begin{proof}
Part (1) is straightforward (see \cite[Lemma A1.2.2]{Ravenel} for example). Part (2) follows directly from Proposition \ref{prop:equivMGLheart}.

For Part (3), we have the equivalences
\begin{align*}
\textup{MU}^{\textup{mot}} \wedge I & \simeq \textup{MU}^{\textup{mot}}/\tau \wedge_{\widehat{S^{0,0}}/\tau} I\\ & \simeq  \textup{MU}^{\textup{mot}}/\tau \wedge_{\widehat{S^{0,0}}/\tau} \left(  \textup{MU}^{\textup{mot}}/\tau \smas_{ \textup{MU}^{\textup{mot}}/\tau} I \right) \\ 
 & \simeq  \left(  \textup{MU}^{\textup{mot}}/\tau \wedge_{\widehat{S^{0,0}}/\tau}  \textup{MU}^{\textup{mot}}/\tau \right) \smas_{\textup{MU}^{\textup{mot}}/\tau} I.
\end{align*}
Since $\textup{MU}^{\textup{mot}}/\tau$ is $\widehat{S^{0,0}}/\tau$-cellular, the homotopy groups of the last term can be computed by the homology version of Dugger-Isaksen's universal coefficient spectral sequence \cite[Proposition 7.10]{DuggerIsaksen}
\begin{equation*}
\Tor_{s,t,w}^{\textup{MU}^{\textup{mot}}_{*,*}/\tau}\left( \textup{MU}^{\textup{mot}}_{*,*}\textup{MU}^{\textup{mot}}/\tau, \piss I \right) \Longrightarrow \pi_{t+s,w}\left( \textup{MU}^{\textup{mot}}/\tau \wedge_{\widehat{S^{0,0}}/\tau} \textup{MU}^{\textup{mot}}/\tau \smas_{\textup{MU}^{\textup{mot}}/\tau} I \right)
\end{equation*}
in the category $\textup{MU}^{\textup{mot}}/\tau\text{-}\mathbf{Mod}_{\textup{cell}}$.

Since $\textup{MU}^{\textup{mot}}_{*,*}\textup{MU}^{\textup{mot}}/\tau$ is free, 
the spectral sequence is concentrated on the line $s=0$ and collapses at the $E_2$-page. This proves Part (3).

\end{proof}

\begin{remark}
The statement of Part (3) does not require that $N$ is injective.	
\end{remark}

Lemma \ref{lemma:injrestop1} is our source of motivic $\widehat{S^{0,0}}/\tau$-modules whose $\textup{MU}^{\textup{mot}}/\tau$-homology is injective as a comodule.

\begin{lemma} \label{lemma:injrestop2}
Suppose that $I$ is an object in $\textup{MU}^{\textup{mot}}/\tau\text{-}\mathbf{Mod}_{\textup{cell}}^\heartsuit$ such that $\pi_{*,*}I$ is an injective $\textup{MU}^{\textup{mot}}_{*,*}/\tau$-module.


Then for any $X \in \widehat{S^{0,0}}/\tau\text{-}\mathbf{Mod}_{\textup{cell}}^b$, we have 
$$
[X,  I]_{\widehat{S^{0,0}}/\tau} \cong \Hom_{\textup{MU}^{\textup{mot}}_{*,*}\textup{MU}^{\textup{mot}}/\tau}(\textup{MU}^{\textup{mot}}_{*,*}X, \textup{MU}^{\textup{mot}}_{*,*}I).
$$
\end{lemma}

\begin{proof}
The lemma follows from the following isomorphisms
\begin{align*}
[X,  I]_{\widehat{S^{0,0}}/\tau} & \cong [\textup{MU}^{\textup{mot}}/\tau \wedge_{\widehat{S^{0,0}}/\tau} X,  I]_{\textup{MU}^{\textup{mot}}/\tau}\\
& \cong \Hom_{\textup{MU}^{\textup{mot}}_{*,*}/\tau}(\textup{MU}^{\textup{mot}}_{*,*}X, \piss I) \\ 
& \cong \Hom_{\textup{MU}^{\textup{mot}}_{*,*}/\tau}(\textup{MU}^{\textup{mot}}_{*,*}X, N)\\
& \cong \Hom_{\textup{MU}^{\textup{mot}}_{*,*}\textup{MU}^{\textup{mot}}/\tau}(\textup{MU}^{\textup{mot}}_{*,*}X, \textup{MU}^{\textup{mot}}_{*,*}\textup{MU}^{\textup{mot}}/\tau \otimes_{\textup{MU}^{\textup{mot}}_{*,*}/\tau} N)\\
& \cong \Hom_{\textup{MU}^{\textup{mot}}_{*,*}\textup{MU}^{\textup{mot}}/\tau}(\textup{MU}^{\textup{mot}}_{*,*}X, \textup{MU}^{\textup{mot}}_{*,*}I).
\end{align*}
In fact, the first isomorphism follows from the adjunction \eqref{eq:adjSCellCtCellMGLCell} between $\widehat{S^{0,0}}/\tau$-modules and $\textup{MU}^{\textup{mot}}/\tau$-modules. The third and last isomorphisms follow from Lemma \ref{lemma:injrestop1}. The fourth isomorphism follows from a change-of-ring isomorphism. It remains to show the second isomorphism.

Since both $I$ and $\textup{MU}^{\textup{mot}}/\tau \wedge_{\widehat{S^{0,0}}/\tau} X$ belong to $\textup{MU}^{\textup{mot}}/\tau\text{-}\mathbf{Mod}_{\textup{cell}}$,
the set of homotopy classes of maps
$$[\textup{MU}^{\textup{mot}}/\tau \wedge_{\widehat{S^{0,0}}/\tau} X, I]_{\textup{MU}^{\textup{mot}}/\tau}$$
can be computed by the universal coefficient spectral sequence of Theorem \ref{thm:UCTMGL}
\begin{equation*}
\Ext_{\textup{MU}^{\textup{mot}}_{*,*}/\tau}^{s,t,w}( \textup{MU}^{\textup{mot}}_{*,*} X, \piss I) \Longrightarrow [\Sigma^{t-s,w} \textup{MU}^{\textup{mot}}/\tau \wedge_{\widehat{S^{0,0}}/\tau} X, I]_{\textup{MU}^{\textup{mot}}/\tau}.
\end{equation*}

Since $\piss I \cong N$ is an injective $\textup{MU}^{\textup{mot}}_{*,*}/\tau$-module, the spectral sequence is concentrated on the line $s=0$ and collapses at the $E_2$-page. This gives the second isomorphism.
\end{proof}


\begin{lemma} \label{lemma:injresalg}
For any $M \in \textup{MU}^{\textup{mot}}_{*,*}\textup{MU}^{\textup{mot}}/\tau\text{-}\mathbf{Comod}$ that is concentrated in Chow-Novikov degree $k$, there exists a monomorphism 
\begin{displaymath}
	\xymatrix{
	M \ar@{^{(}->}[rr] & & \textup{MU}^{\textup{mot}}_{*,*}\textup{MU}^{\textup{mot}}/\tau \otimes_{\textup{MU}^{\textup{mot}}_{*,*}/\tau} N,
	}
\end{displaymath}
where $N$ is injective in $\textup{MU}^{\textup{mot}}_{*,*}/\tau\text{-}\mathbf{Mod}$ and is concentrated in Chow-Novikov degree $k$.
\end{lemma}

\begin{proof}

Since there are enough injective objects in the category $\textup{MU}^{\textup{mot}}_{*,*}/\tau\text{-}\mathbf{Mod}$, we may choose an embedding $M \hookrightarrow N$ into an injective object in the category $\textup{MU}^{\textup{mot}}_{*,*}/\tau\text{-}\mathbf{Mod}$. Then the induced comodule map
\begin{displaymath}
	\xymatrix{
	M \ar@{^{(}->}[rr] & & \textup{MU}^{\textup{mot}}_{*,*}\textup{MU}^{\textup{mot}}/\tau \otimes_{\textup{MU}^{\textup{mot}}_{*,*}/\tau} N,
	}
\end{displaymath}
is also a monomorphism.
\end{proof}


\begin{prop}\label{prop:ANtower}
For any $Y \in   \widehat{S^{0,0}}/\tau\text{-}\mathbf{Mod}_{\textup{cell}}^b$, there exists a tower of the following form 
\begin{displaymath}
	\xymatrix{
	Y \ar@{=}[r] & Y_0 \ar[d]  & Y_1 \ar[d] \ar[l] & Y_2 \ar[d] \ar[l] &  \cdots \ar[l] \\
	& I_0 & I_1  &   I_2  &
	}
\end{displaymath} 
in the category $\widehat{S^{0,0}}/\tau\text{-}\mathbf{Mod}_{\textup{cell}}^b$, such that
\begin{enumerate}
\item each map $Y_{s} \lto Y_{s-1}$ induces the zero homomorphism in $\textup{MU}^{\textup{mot}}$-homology,
\item each cofiber $I_s$ is a finite product of suspensions of objects $I$ in $\textup{MU}^{\textup{mot}}/\tau\text{-}\mathbf{Mod}_{\textup{cell}}^\heartsuit$ such that $\pi_{*,*}I$ is an injective $\textup{MU}^{\textup{mot}}_{*,*}/\tau$-module.

\end{enumerate}
We call such a tower an $\mathbf{absolute}$ Adams-Novikov tower.

Moreover, any map $f \colon X \lto Y$ in $\widehat{S^{0,0}}/\tau\text{-}\mathbf{Mod}_{\textup{cell}}^b$ can be lifted to a map of absolute Adams-Novikov towers.
\end{prop}

\begin{proof}
Suppose that $\textup{MU}^{\textup{mot}}_{*,*}Y$ is concentrated in Chow-Novikov degrees $[a,  b]$, namely
\begin{equation*}
\textup{MU}^{\textup{mot}}_{*,*}Y \cong \bigoplus_{k = a}^b \bigoplus_{l =-\infty}^{+\infty} \textup{MU}^{\textup{mot}}_{2l+k, l} Y.
\end{equation*}
By Lemma \ref{lemma:injresalg}, for every $k \in [a, b]$, there exists a monomorphism
\begin{equation*}
\bigoplus_{l = -\infty}^{+\infty} \textup{MU}^{\textup{mot}}_{2l+k, l}( Y) \cong \textup{MU}^{\textup{mot}}_{*, *}( \Sigma^{-k, 0} Y)^{=0} \inj \textup{MU}^{\textup{mot}}_{*,*}\textup{MU}^{\textup{mot}}/\tau \otimes_{\textup{MU}^{\textup{mot}}_{*,*}/\tau} N_{0, k}
\end{equation*}
where $N_{0,k}$ is injective module that is concentrated in Chow-Novikov degree $0$. By Lemma \ref{lemma:injrestop1}, there exists a spectrum $I_{0,k} \in \textup{MU}^{\textup{mot}}/\tau\text{-}\mathbf{Mod}_{\textup{cell}}^\heartsuit$ such that
$$\piss I_{0, k} \cong N_{0,k},$$
and that 
$$\textup{MU}^{\textup{mot}}_{*,*} I_{0, k} \cong \textup{MU}^{\textup{mot}}_{*,*}\textup{MU}^{\textup{mot}}/\tau \otimes_{\textup{MU}^{\textup{mot}}_{*,*}/\tau} N_{0,k}.$$
By Lemma \ref{lemma:injrestop2}, we have that 
$$
[\Sigma^{-k, 0} Y,  I_{0, k}]_{\widehat{S^{0,0}}/\tau} 
\cong \Hom_{\textup{MU}^{\textup{mot}}_{*,*}\textup{MU}^{\textup{mot}}/\tau}(\textup{MU}^{\textup{mot}}_{*,*} (\Sigma^{-k, 0} Y), \textup{MU}^{\textup{mot}}_{*,*}\textup{MU}^{\textup{mot}}/\tau \otimes_{\textup{MU}^{\textup{mot}}_{*,*}/\tau} N_{0, k})$$ $$
\cong \Hom_{\textup{MU}^{\textup{mot}}_{*,*}\textup{MU}^{\textup{mot}}/\tau}(\textup{MU}^{\textup{mot}}_{*, *}( \Sigma^{-k, 0} Y)^{=0}, \textup{MU}^{\textup{mot}}_{*,*}\textup{MU}^{\textup{mot}}/\tau \otimes_{\textup{MU}^{\textup{mot}}_{*,*}/\tau} N_{0, k}).
$$
The second isomorphism follows from the fact that $N_{0, k}$ is concentrated in Chow-Novikov degree $0$. Therefore, the algebraic map of comodules
\begin{displaymath}
	\xymatrix{
 \textup{MU}^{\textup{mot}}_{\ast,\ast}( \Sigma^{-k,0} Y) \ar@{->>}[r] & \textup{MU}^{\textup{mot}}_{*, *}( \Sigma^{-k, 0} Y)^{=0} \ar@{^{(}->}[r] & \textup{MU}^{\textup{mot}}_{*,*}\textup{MU}^{\textup{mot}}/\tau \otimes_{\textup{MU}^{\textup{mot}}_{*,*}/\tau} N_{0,k},
	}
\end{displaymath}
where the first map is the project map to the Chow-Novikov degree $0$ part, can be realized as a $\widehat{S^{0,0}}/\tau$-linear map 
$$\Sigma^{-k,0} Y \longrightarrow I_{0,k}.$$
Combine these maps for all $k\in[a, b]$, we obtain a map
$$Y \longrightarrow \prod_{k=a}^b \Sigma^{k,0} I_{0,k}.$$
This map induces a monomorphism in $\textup{MU}^{\textup{mot}}$-homology.

Denote the finite product by 
$$I_0 \coloneqq \prod_{k=a}^b \Sigma^{k,0} I_{0,k},$$
and the fiber of the map $Y \rightarrow I_0$ by $Y_1$, as in
\begin{displaymath}
	\xymatrix{
 Y \ar[d]  & Y_1  \ar[l]  \\
	 I_0. &     
	}
\end{displaymath} 
By the associated long exact sequence in $\textup{MU}^{\textup{mot}}$-homology, the map $Y_1 \lto Y$ induces the zero map in $\textup{MU}^{\textup{mot}}$-homology, and $\textup{MU}^{\textup{mot}}_{*,*}Y_1$ is concentrated in Chow-Novikov degrees $[a- 1,b-1]$. So in particular we have 
$$Y_1 \in   \widehat{S^{0,0}}/\tau\text{-}\mathbf{Mod}_{\textup{cell}}^b.$$
We can repeat the procedure, producing an absolute Adams-Novikov tower
\begin{displaymath}
	\xymatrix{
 Y \ar[d]  & Y_1 \ar[d] \ar[l] & Y_2 \ar[d] \ar[l] &  \cdots \ar[l] \\
	 I_0 & I_1  &   I_2  &
	}
\end{displaymath} 
satisfying the desired properties.

We now prove the second claim of the theorem. For any $\widehat{S^{0,0}}/\tau$-linear map $f_0:X_0 \rightarrow Y_0$,
we may assume that $\textup{MU}^{\textup{mot}}_{*,*}X_0$ and $\textup{MU}^{\textup{mot}}_{*,*}Y_0$ are both concentrated in Chow-Novikov degrees $[a, b]$.
Denote the first step of their tower by
\begin{displaymath}
	\xymatrix{
 X_0 \ar[d] \ar[r]^{f_0}  & Y_0 \ar[d]  \\
	 I_0 & J_0,  
	 	}
\end{displaymath}
where $I_0$ and $J_0$ are the finite products of suspensions of objects that satisfy the conclusions of Lemma \ref{lemma:injrestop1}. Applying $\textup{MU}^{\textup{mot}}_{*,*}$, we have the following diagram of $\textup{MU}^{\textup{mot}}_{*,*}\textup{MU}^{\textup{mot}}/\tau$-comodules
\begin{displaymath}
	\xymatrix{
 \textup{MU}^{\textup{mot}}_{*,*}X_0 \ar@{^{(}->}[d] \ar[r]^{{f_0}_{*,*}}  & \textup{MU}^{\textup{mot}}_{*,*}Y_0 \ar@{^{(}->}[d]  \\
	 \textup{MU}^{\textup{mot}}_{*,*}I_0 \ar@{-->}[r]^{\phi} & \textup{MU}^{\textup{mot}}_{*,*}J_0.  
	 	}
\end{displaymath}
Here the existence of the homomorphism $\phi$ is due to the universal property of injective objects in the category $\textup{MU}^{\textup{mot}}_{*,*}\textup{MU}^{\textup{mot}}/\tau\text{-}\mathbf{Comod}$. 

We have
$$\textup{MU}^{\textup{mot}}_{*,*}I_0 = \textup{MU}^{\textup{mot}}_{*,*}(\prod_{k=a}^b \Sigma^{k,0} I_{0,k})= \prod_{k=a}^b \textup{MU}^{\textup{mot}}_{*,*}(\Sigma^{k,0} I_{0,k}),$$
$$\textup{MU}^{\textup{mot}}_{*,*}J_0 = \textup{MU}^{\textup{mot}}_{*,*}(\prod_{k=a}^b \Sigma^{k,0} J_{0,k})= \prod_{k=a}^b \textup{MU}^{\textup{mot}}_{*,*}(\Sigma^{k,0} J_{0,k}).$$
The Chow-Novikov degree $k$ parts of $\textup{MU}^{\textup{mot}}_{*,*}I_0$ and $\textup{MU}^{\textup{mot}}_{*,*}J_0$ are given by 
$$\textup{MU}^{\textup{mot}}_{*,*}(\Sigma^{k,0} I_{0,k}), \ \textup{MU}^{\textup{mot}}_{*,*}(\Sigma^{k,0} J_{0,k}).$$
Therefore, the homomorphism $\phi$ is given by the product of homomorphisms
$$\phi_k: \textup{MU}^{\textup{mot}}_{*,*}(\Sigma^{k,0} I_{0,k}) \longrightarrow \textup{MU}^{\textup{mot}}_{*,*}(\Sigma^{k,0} J_{0,k})$$
for each $k\in[a,b]$.

Since $J_{0,k}$ satisfies the conclusions of Lemma \ref{lemma:injrestop1}, we have that
\begin{align*}
[\Sigma^{k,0} I_{0,k},  \Sigma^{k,0} J_{0,k}]_{\widehat{S^{0,0}}/\tau} & \cong [I_{0,k},  J_{0,k}]_{\widehat{S^{0,0}}/\tau}\\
& \cong \Hom_{\textup{MU}^{\textup{mot}}_{*,*}\textup{MU}^{\textup{mot}}/\tau}(\textup{MU}^{\textup{mot}}_{*,*}I_{0,k}, \textup{MU}^{\textup{mot}}_{*,*}J_{0,k})\\
& \cong \Hom_{\textup{MU}^{\textup{mot}}_{*,*}\textup{MU}^{\textup{mot}}/\tau}(\textup{MU}^{\textup{mot}}_{*,*}(\Sigma^{k,0}I_{0,k}), \textup{MU}^{\textup{mot}}_{*,*}(\Sigma^{k,0}J_{0,k})),
\end{align*}
where the second isomorphism is given by Lemma \ref{lemma:injrestop2}. Therefore the homomorphism $\phi_k$ can be realized by a $\widehat{S^{0,0}}/\tau$-linear map 
$$g_{0,k}: \Sigma^{k,0}I_{0,k} \longrightarrow \Sigma^{k,0}J_{0,k}.$$
Taking the product of $g_{0,k}$ for all $k\in[a,b]$, we define a map $g_0: I_0 \rightarrow J_0$. Then $g_0$ realizes $\phi$, and we have the diagram
\begin{displaymath}
	\xymatrix{
 X_0 \ar[d] \ar[r]^{f_0}  & Y_0 \ar[d]  \\
	 I_0 \ar[r]^{g_0} & J_0. 
	 	}
\end{displaymath}
To see that the square commutes up to homotopy, we have
\begin{align*}
[X,  J_{0}]_{\widehat{S^{0,0}}/\tau} & \cong [X,  \prod_{k=a}^b \Sigma^{k,0} J_{0,k}]_{\widehat{S^{0,0}}/\tau}\\
& \cong \prod_{k=a}^b \Hom_{\textup{MU}^{\textup{mot}}_{*,*}\textup{MU}^{\textup{mot}}/\tau}(\textup{MU}^{\textup{mot}}_{*,*}X, \textup{MU}^{\textup{mot}}_{*,*}(\Sigma^{k,0}J_{0,k}))\\
& \cong \Hom_{\textup{MU}^{\textup{mot}}_{*,*}\textup{MU}^{\textup{mot}}/\tau}(\textup{MU}^{\textup{mot}}_{*,*}X, \textup{MU}^{\textup{mot}}_{*,*}J_0),
\end{align*}
where the second isomorphism is given by Lemma \ref{lemma:injrestop2}.
 
Therefore, the commutativity of this square follows from the commutativity of the corresponding square in $\textup{MU}^{\textup{mot}}$-homology. 

The commutative diagram in $\widehat{S^{0,0}}/\tau\text{-}\mathbf{Mod}_{\textup{cell}}^b$ induces a map $f_1:X_1\rightarrow Y_1$ between the fibers, so the following diagram commutes up to homotopy
\begin{equation*}
\begin{tikzpicture}
\matrix (m) [matrix of math nodes, row sep=2.5em, column sep=4em]
{ X   & X_1  & X_2 &  \cdots  \\
  I_0 & I_1  &     & \\
  Y   & Y_1  & Y_2 &  \cdots  \\
  J_0 & J_1.  &     & \\};
\path[thick, -stealth, font=\small]
(m-1-2) edge (m-1-1)
(m-1-3) edge (m-1-2)
(m-1-4) edge (m-1-3)
(m-1-1) edge (m-2-1)
(m-1-2) edge (m-2-2)

(m-3-2) edge (m-3-1)
(m-3-3) edge (m-3-2)
(m-3-4) edge (m-3-3)
(m-3-1) edge (m-4-1)
(m-3-2) edge (m-4-2)

(m-1-1) edge[bend right=45] node[left] {$ f_0 $} (m-3-1)
(m-2-1) edge[bend right=45] node[left] {$ g_0 $} (m-4-1);
\path[dashed, -stealth, font=\small]
(m-1-2) edge[bend right=45] node[left] {$ f_1 $} (m-3-2);
\end{tikzpicture}
\end{equation*}
Iterating this process produces the desired map of absolute Adams-Novikov towers.
\end{proof}

\subsection{The spectral sequence}

Every absolute Adams-Novikov tower gives rise to an absolute Adams-Novikov spectral sequence. In the following Theorem \ref{thm:ANss}, we identify the $E_2$-page of the spectral sequence and its abutment. We also show that it does not depend on the absolute Adams-Novikov tower, and converges strongly for objects with bounded Chow-Novikov degree.

\begin{thm} \label{thm:ANss}
For $X,Y \in \widehat{S^{0,0}}/\tau\text{-}\mathbf{Mod}_{\textup{cell}}^b$, there is an absolute Adams-Novikov spectral sequence
\begin{equation*}
E_2^{s,t,w} \cong \Ext^{s,t,w}_{\textup{MU}^{\textup{mot}}_{*,*}\textup{MU}^{\textup{mot}}/\tau}\left( \textup{MU}^{\textup{mot}}_{*,*}X, \textup{MU}^{\textup{mot}}_{*,*}Y\right) \Longrightarrow \left[\Sigma^{t-s,w} X,Y^{\smas}_{\textup{MU}^{\textup{mot}}} \right]_{\widehat{S^{0,0}}/\tau},
\end{equation*}
with differentials $$d_r \colon E_r^{s,t,w} \lto E_r^{s+r,t+r-1,w},$$
that does not depend on the absolute Adams-Novikov tower. Here $Y^{\smas}_{\textup{MU}^{\textup{mot}}}$ is the $\textup{MU}^{\textup{mot}}$-completion of $Y$. Moreover, this spectral sequence converges strongly and collapses at a finite page.
\end{thm}

\begin{proof}
The arguments for the existence of the absolute Adams-Novikov spectral sequence and its independence of the absolute Adams-Novikov tower are both standard. Under the hypotheses that both $\textup{MU}^{\textup{mot}}_{*,*}X$ and $\textup{MU}^{\textup{mot}}_{*,*}Y$ are concentrated in bounded Chow-Novikov degrees, the argument for the strongly convergence and collapse at a finite page is similar to the one given in the proof of Theorem \ref{thm:UCTMGL}: It follows from degree reasons.

To complete the proof, we only need to identify the abutment.

Let $Y/Y_s$ be the cofiber of the map $Y_s \rightarrow Y$ in the absolute Adams-Novikov tower, 
\begin{displaymath}
	\xymatrix{
 Y \ar[d]  & Y_1 \ar[d] \ar[l] & Y_2 \ar[d] \ar[l] &  \cdots \ar[l] \\
	 I_0 & I_1  &   I_2.  &
	}
\end{displaymath}
and define the limit in $\widehat{S^{0,0}}/\tau\text{-}\mathbf{Mod}$
$$\widehat{Y}=\textup{lim}(Y/Y_s).$$
The spectral sequence converges conditionally to
 $$[X, \widehat{Y}]_{\widehat{S^{0,0}}/\tau}.$$
See for example Sections~5 and 15 of \cite{Boardman} for a discussion of convergence issues of the Adams spectral sequence.

To identify it as $[X, Y^\wedge_{\textup{MU}^{\textup{mot}}}]_{\widehat{S^{0,0}}/\tau}$, since $X$ is $\widehat{S^{0,0}}/\tau$-cellular, we only need to show that $\widehat{Y}$ has the same homotopy groups as the $Y^\wedge_{\textup{MU}^{\textup{mot}}}$. 

Take $X = \widehat{S^{0,0}}/\tau$. Since $\textup{MU}^{\textup{mot}}_{*,*} \widehat{S^{0,0}}/\tau = \textup{MU}^{\textup{mot}}_{*,*}/\tau$ is free over itself, we can use the canonical $\textup{MU}^{\textup{mot}}/\tau$-Adams resolution \cite[Definition 2.2.10]{Ravenel} for $Y$ in this case. Now we compare the canonical $\textup{MU}^{\textup{mot}}/\tau$-based Adams-Novikov tower of $Y$ with the absolute Adams-Novikov tower of $Y$. 

As we did in the proof of Proposition \ref{prop:ANtower}, we have a map of towers from the canonical $\textup{MU}^{\textup{mot}}/\tau$-based one to the absolute one. The identity map on $Y$ induces a homomorphism from the canonical cobar resolution of $\textup{MU}^{\textup{mot}}_{*,*}Y$ to the absolute injective resolution of $\textup{MU}^{\textup{mot}}_{*,*}Y$, so in particular a homomorphism of relative injective resolutions.

This induces a homomorphism from the usual Adams-Novikov spectral sequence to the absolute Adams-Novikov spectral sequence, with an isomorphism on the $E_2$-page. It is therefore an isomorphism of spectral sequences and we have an isomorphism
$$\pi_{*,*}\widehat{Y} \stackrel{\cong}{\lto} \pi_{*,*}Y^{\smas}_{\textup{MU}^{\textup{mot}}}.$$
Since any cellular $\widehat{S^{0,0}}/\tau$-module $X$ can be written in terms of filtered colimits and cofibers of suspensions of $\widehat{S^{0,0}}/\tau$'s, there is an isomorphism
\begin{equation*}
[X,\widehat{Y}]_{\widehat{S^{0,0}}/\tau} \stackrel{\cong}{\ltoo} [X, Y^{\smas}_{\textup{MU}^{\textup{mot}}}]_{\widehat{S^{0,0}}/\tau}.
\end{equation*}
Therefore, the absolute Adams-Novikov spectral sequence computes $[X, Y^{\smas}_{\textup{MU}^{\textup{mot}}}]_{\widehat{S^{0,0}}/\tau}$.
\end{proof}

When $Y$ is harmonic, the isomorphism $\pi_{*,*}\widehat{Y} \stackrel{\cong}{\lto} \pi_{*,*}Y^{\smas}_{\textup{MU}^{\textup{mot}}}$ gives the following corollary.

\begin{cor} \label{cor:ANss}
For any $X,Y \in \widehat{S^{0,0}}/\tau\text{-}\mathbf{Mod}_{\textup{harm}}^b$, there is an absolute Adams-Novikov spectral sequence
\begin{equation*}
E_2^{s,t,w}=\Ext^{s,t,w}_{\textup{MU}^{\textup{mot}}_{*,*}\textup{MU}^{\textup{mot}}/\tau}\left(\textup{MU}^{\textup{mot}}_{*,*}X, \textup{MU}^{\textup{mot}}_{*,*}Y \right) \Longrightarrow \left[\Sigma^{t-s,w} X,Y \right]_{\widehat{S^{0,0}}/\tau},
\end{equation*}
with differentials 
$$d_r \colon E_r^{s,t,w} \lto E_r^{s+r,t+r-1,w},$$
that converges strongly and collapses at a finite page.
\end{cor}

\begin{remark}

The above arguments can be applied to more general situations. However, this construction of an absolute Adams-Novikov spectral sequence depends on realizability of categorical injectives, so the range of situations to which it applies may be rather limited.

In the case of the classical stable homotopy category, for spectra $X$ and $Y$, there is a conditionally convergent spectral sequence
$$ \Ext_{\textup{MU}_*\textup{MU}}^{s,t}(\textup{MU}_*X,\textup{MU}_*Y) \Longrightarrow [\Sigma^{t-s}X,Y^\wedge_{\textup{MU}}],$$
where $\textup{MU}_*X$ does \emph{not} have to be projective over $\textup{MU}_*$. 
We will discuss this case in a general framework in future work.
\end{remark}

\subsection{Proofs of Lemma \ref{lemma:pisscompclosedunderubfilcolim}, Corollary \ref{cor:Cttstruct1} and Corollary \ref{cor:L45}}

We give the proofs of Lemma \ref{lemma:pisscompclosedunderubfilcolim}, Corollary \ref{cor:Cttstruct1} and Corollary \ref{cor:L45} in this section.


Corollary \ref{cor:Cttstruct1} states that if $X \in \widehat{S^{0,0}}/\tau\text{-}\mathbf{Mod}_{\textup{harm}}^{b,\geq0}$ and $Y \in \widehat{S^{0,0}}/\tau\text{-}\mathbf{Mod}_{\textup{harm}}^{b,\leq0}$, then the abelian group of homotopy classes of degree $(0,0)$ maps can be computed algebraically by the isomorphism
\begin{equation*}
[X,Y]_{\widehat{S^{0,0}}/\tau} \longrightarrow \Hom_{\textup{MU}^{\textup{mot}}_{*,*}\textup{MU}^{\textup{mot}}/\tau}( \textup{MU}^{\textup{mot}}_{*,*} X,\textup{MU}^{\textup{mot}}_{*,*} Y )
\end{equation*}
that is induced by applying $\textup{MU}^\textup{mot}_{*,*}$.

\begin{proof}[Proof of Corollary \ref{cor:Cttstruct1}]
The proof is similar to the one of Corollary \ref{C47}.
 
Consider the the $E_2$-page of the absolute Adams-Novikov spectral sequence, the tri-degrees that converge to the bidegree $(0, 0)$ are of the form $(t, t, 0)$ for $t \geq 0$, i.e., the parts $E_2^{s, t, w} = E_2^{t, t, 0}$.

By the proof of Theorem \ref{thm:ANss}, the $t$-degrees of all possible nonzero elements in the $E_1$-page and therefore $E_2$-page satisfy $t \leq d-a + 2w = d-a$. Since $\textup{MU}^{\textup{mot}}_{*,*}X$ and $\textup{MU}^{\textup{mot}}_{*,*}Y$ are concentrated in nonnegative and nonpositive bounded Chow-Novikov degrees, we have $d=a=0$. Therefore, we have $t \leq 0$.

Combining both facts, we have established that the only possible nonzero elements in the $E_2$-page that converge to the bidegree $(0, 0)$ are in 
$$E_2^{0,0,0} = \Hom_{\textup{MU}^{\textup{mot}}_{*,*}\textup{MU}^{\textup{mot}}/\tau}( \textup{MU}^{\textup{mot}}_{*,*} X,\textup{MU}^{\textup{mot}}_{*,*} Y ).$$
To show that all elements in $E_2^{0,0,0}$ survive in the spectral sequence, note that they are not targets of any nonzero differentials since they are in $s$-degree $0$. Second, all $d_r$-differentials for $r \geq 2$ increase the $t$-degree. Since the $t$-degrees of all nonzero elements are non-positive, the elements in $E_2^{0,0,0}$ do not support nonzero differentials. There are no hidden extensions due to degree reasons. This completes the proof.
\end{proof}

Corollary \ref{cor:L45} states that given $X, Y \in \widehat{S^{0,0}}/\tau\text{-}\mathbf{Mod}_{{\textup{harm}}}^\heartsuit$, for any bidegree $(t,w)$, there is an isomorphism
\begin{equation*}
[\Sigma^{t,w} X,Y]_{\widehat{S^{0,0}}/\tau} \cong \Ext^{2w -t, 2w, w}_{\textup{MU}^{\textup{mot}}_{*,*}\textup{MU}^{\textup{mot}}/\tau}(\textup{MU}^{\textup{mot}}_{*,*}X, \textup{MU}^{\textup{mot}}_{*,*}Y).
\end{equation*}

\begin{proof}[Proof of Corollary \ref{cor:L45}]
Consider the the $E_2$-page of the absolute Adams-Novikov spectral sequence. Since both $\textup{MU}^{\textup{mot}}_{*,*}X$ and $\textup{MU}^{\textup{mot}}_{*,*}Y$ are concentrated in Chow-Novikov degree 0, the $E_2$-page
$$E_2^{s, t, w} = \Ext^{s,t,w}_{\textup{MU}^{\textup{mot}}_{*,*}\textup{MU}^{\textup{mot}}/\tau}\left( \textup{MU}^{\textup{mot}}_{*,*}X, \textup{MU}^{\textup{mot}}_{*,*}Y \right)$$
is concentrated in degrees $t=2w$. Since all differentials preserves the motivic weights $w$, this spectral sequence collapses at the $E_2$-page. There are no hidden extensions due to degree reasons. Therefore, we have the isomorphism
\begin{equation*}
[\Sigma^{t,w} X,Y]_{\widehat{S^{0,0}}/\tau} \cong \Ext^{2w -t, 2w, w}_{\textup{MU}^{\textup{mot}}_{*,*}\textup{MU}^{\textup{mot}}/\tau}(\textup{MU}^{\textup{mot}}_{*,*}X, \textup{MU}^{\textup{mot}}_{*,*}Y).
\end{equation*}
\end{proof}

We now prove Lemma \ref{lemma:pisscompclosedunderubfilcolim}, which states that if 
$\{Y_{\alpha}\}$ is a filtered system in $\widehat{S^{0,0}}/\tau\text{-}\mathbf{Mod}^\heartsuit$ such that each $Y_{\alpha}$ is harmonic, then the colimit of $\{Y_{\alpha}\}$ in $\widehat{S^{0,0}}/\tau\text{-}\mathbf{Mod}^\heartsuit$ is also harmonic.

\begin{proof}[Proof of Lemma \ref{lemma:pisscompclosedunderubfilcolim}]
Consider the absolute Adams-Novikov spectral sequence of Theorem \ref{thm:ANss}
\begin{equation*}
\Ext^{s,t,w}_{\textup{MU}^{\textup{mot}}_{*,*}\textup{MU}^{\textup{mot}}/\tau}\left(\textup{MU}^{\textup{mot}}_{*,*}\widehat{S^{0,0}}/\tau, \textup{MU}^{\textup{mot}}_{*,*}Y \right) \Longrightarrow \left[\Sigma^{t-s,w} \widehat{S^{0,0}}/\tau,Y^{\smas}_{\textup{MU}^{\textup{mot}}} \right]_{\widehat{S^{0,0}}/\tau} \cong \pi_{t-s,w}Y^{\smas}_{\textup{MU}^{\textup{mot}}}
\end{equation*}
in the case that $X = \widehat{S^{0,0}}/\tau$ and $Y = \textup{colim} \ Y_{\alpha}$. Since both $\widehat{S^{0,0}}/\tau$ and $Y$ are in $\widehat{S^{0,0}}/\tau\text{-}\mathbf{Mod}^\heartsuit$, the $E_2$-page is concentrated in degrees $t = 2w$. Since all differentials preserve the motivic weights $w$, this spectral sequence collapses at the $E_2$-page. There are no hidden extensions due to degree reasons. Therefore, we have the isomorphism
\begin{equation*} 
\pi_{t,w} Y^{\smas}_{\textup{MU}^{\textup{mot}}} \cong [\Sigma^{t,w} \widehat{S^{0,0}}/\tau ,Y^{\smas}_{\textup{MU}^{\textup{mot}}}]_{\widehat{S^{0,0}}/\tau} \cong \Ext^{2w -t, 2w, w}_{\textup{MU}^{\textup{mot}}_{*,*}\textup{MU}^{\textup{mot}}/\tau}(\textup{MU}^{\textup{mot}}_{*,*}\widehat{S^{0,0}}/\tau, \textup{MU}^{\textup{mot}}_{*,*}Y).
\end{equation*}

Since $\textup{MU}^{\textup{mot}}_{*,*}\widehat{S^{0,0}}/\tau \cong \textup{MU}^{\textup{mot}}_{*,*}/\tau$ is free over $\textup{MU}^{\textup{mot}}_{*,*}/\tau$, one can use the canonical cobar resolution for $\textup{MU}^{\textup{mot}}_{*,*}Y$. Since it is functorial and commutes with filtered colimits,
the isomorphism
\begin{equation*}
\colim \textup{MU}^{\textup{mot}}_{*,*}Y_{\alpha} \cong \textup{MU}^{\textup{mot}}_{*,*}(Y)
\end{equation*}
induces an isomorphism 
\begin{equation*} 
\colim \Ext^{*,*,*}_{\textup{MU}^{\textup{mot}}_{*,*}\textup{MU}^{\textup{mot}}/\tau}\left(\textup{MU}^{\textup{mot}}_{*,*}/\tau,\textup{MU}^{\textup{mot}}_{*,*}Y_{\alpha} \right) \cong \Ext^{*,*,*}_{\textup{MU}^{\textup{mot}}_{*,*}\textup{MU}^{\textup{mot}}/\tau}\left( \textup{MU}^{\textup{mot}}_{*,*}/\tau, \textup{MU}^{\textup{mot}}_{*,*}Y \right).
\end{equation*}

Therefore, we have the following isomorphisms
\begin{align*}
\pi_{t,w} Y & \cong \pi_{t,w} (\textup{colim} \ Y_\alpha)\\
& \cong \textup{colim} \ \pi_{t,w} Y_\alpha\\
& \cong \textup{colim} \ [\Sigma^{t,w} \widehat{S^{0,0}}/\tau,  Y_\alpha]_{\widehat{S^{0,0}}/\tau}\\
& \cong \textup{colim} \Ext^{2w-t,2w,w}_{\textup{MU}^{\textup{mot}}_{*,*}\textup{MU}^{\textup{mot}}/\tau}\left(\textup{MU}^{\textup{mot}}_{*,*}/\tau, \textup{MU}^{\textup{mot}}_{*,*}Y_{\alpha} \right)\\
& \cong \Ext^{2w-t,2w,w}_{\textup{MU}^{\textup{mot}}_{*,*}\textup{MU}^{\textup{mot}}/\tau}\left(\textup{MU}^{\textup{mot}}_{*,*}/\tau, \textup{MU}^{\textup{mot}}_{*,*}Y \right)\\
& \cong [\Sigma^{t,w} \widehat{S^{0,0}}/\tau ,Y^{\smas}_{\textup{MU}^{\textup{mot}}}]_{\widehat{S^{0,0}}/\tau}\\
& \cong \pi_{t,w} Y^{\smas}_{\textup{MU}^{\textup{mot}}},
\end{align*}
where the fourth isomorphism is given by Corollary \ref{cor:L45}, since each $Y_{\alpha}$ is harmonic. The composite is induced by the completion map $Y \rightarrow Y^{\smas}_{\textup{MU}^{\textup{mot}}}$. This shows that $Y$ is harmonic.
\end{proof}

\begin{remark}
Lemma~\ref{lemma:pisscompclosedunderubfilcolim} can be generalized to the case when there is a uniform bound on the Chow-Novikov degrees of $\textup{MU}^{\textup{mot}}_{*,*}Y_{\alpha}$ for all $\alpha$. Part $(3)$ of Example~4.1 shows that Lemma~\ref{lemma:pisscompclosedunderubfilcolim} cannot hold without any bound on Chow-Novikov degrees.

\end{remark}

\section{Further questions}
The category of cellular modules over $\widehat{S^{0,0}}/\tau$ measures the difference between cellular  modules over the H$\mathbb{F}_p^{\textup{mot}}$-completed motivic sphere spectrum $\widehat{S^{0,0}}$ and cellular modules over the classical $p$-completed sphere spectrum $\widehat{S^0}$.


\begin{defn}
Let $ \widehat{S^{0,0}}\text{-}\mathbf{Mod}_\textup{fin}$ be the category of finite cellular modules over $\widehat{S^{0,0}}$, and $ \widehat{S^0}\text{-}\mathbf{Mod}_\textup{fin}$ be the category of classical finite cellular modules over $\widehat{S^0}$.
Let $\widehat{S^{0,0}}\text{-}\mathbf{Mod}^{\tau\textup{-tor}}_\textup{fin}$ be the full subcategory of $\widehat{S^{0,0}}\text{-}\mathbf{Mod}_\textup{fin}$ that is generated by $\widehat{S^{0,0}}/\tau\text{-}\mathbf{Mod}_\text{fin}$ under cofibers, i.e. the smallest full subcategory containing objects of $\widehat{S^{0,0}}/\tau\text{-}\mathbf{Mod}_\text{fin}$ and closed under taking cofibers.
\end{defn}

The following Proposition \ref{prop:exactseqcat} can be proved from Dugger-Isaksen \cite[Sections 3.2 and 3.4]{DuggerIsaksenMASS} and Isaksen \cite[Proposition 3.0.2]{StableStems}.

\begin{prop} \label{prop:exactseqcat}
The sequence
\begin{displaymath}
	\xymatrix{
	\widehat{S^{0,0}}\text{-}\mathbf{Mod}^{\tau\textup{-tor}}_\textup{fin} \ar[r] & \widehat{S^{0,0}}\text{-}\mathbf{Mod}_\textup{fin} \ar[r]^{\widehat{\mathbf{Re}}} & \widehat{S^0}\textup{-}\mathbf{Mod}_\textup{fin}
	}
\end{displaymath}
is an exact sequence of stable $\infty$-categories in the sense of Blumberg-Gepner-Tabuada \cite[Section 5]{BGT},
where $\widehat{\mathbf{Re}}$ is the $p$-completed version of the Betti realization functor (\cite[Theorem 1.4]{DI04}) explained below. 
\end{prop}

Let $\mathbf{Re}$ be the Betti realization functor constructed in Dugger-Isaksen \cite[Theorem 1.4]{DI04}. It is symmetric monoidal and preserves colimits. It was shown by Dugger-Isaksen \cite{DuggerIsaksenMASS} that, $\mathbf{Re}$ sends the motivic Adams tower for $S^{0,0}$ to the classical Adams tower for $S^0$. Taking the limit, we get a map of $E_\infty$ spectra:
$$\mathbf{Re}(\widehat{S^{0,0}})\rightarrow \widehat{S^0}.$$

For any $\widehat{S^{0,0}}$-module $X$, we define the $p$-completed Betti realization functor to be
$$\widehat{\mathbf{Re}}(X) := \mathbf{Re}(X)\wedge_{\mathbf{Re}(\widehat{S^{0,0}})}\widehat{S^0}.$$
The $p$-completed Betti realization functor $\widehat{\mathbf{Re}}$ sends $\widehat{S^{0,0}}$ to $\widehat{S^0}$. It is symmetric monoidal, and preserves colimits.


In the sense of Proposition \ref{prop:exactseqcat}, our Theorem~\ref{intro:mainthmpart1} gives a decomposition of the cellular stable motivic category into more classical categories.

In particular, we can apply the non-connective algebraic $K$-theory functor $\mathbb{K}$ constructed in Blumberg-Gepner-Tabuada \cite[Section 9]{BGT}, and get a cofiber sequence of non-connective algebraic $K$-theory spectra, since the functor $\mathbb{K}$ sends exact sequence of stable $\infty$-categories into cofiber sequences:
\begin{displaymath}
	\xymatrix{
	\mathbb{K}(\widehat{S^{0,0}}\text{-}\mathbf{Mod}_\text{fin}^{\tau\text{-tor}}) \ar[r] &\mathbb{K}(\widehat{S^{0,0}}\text{-}\mathbf{Mod}_\text{fin}) \ar[r]^{\widehat{\mathbf{Re}}} &\mathbb{K}(\widehat{S^0}\text{-}\mathbf{Mod}_\text{fin})
	}
\end{displaymath}
Since the Betti realization functor admits a section, the above cofiber sequence actually splits
$$\mathbb{K}(\widehat{S^{0,0}}\text{-}\mathbf{Mod}_\text{fin}) \simeq \mathbb{K}(\widehat{S^0}\text{-}\mathbf{Mod}_\text{fin}) \vee \mathbb{K}(\widehat{S^{0,0}}\text{-}\mathbf{Mod}_\text{fin}^{\tau\text{-tor}}).$$
The spectrum $\mathbb{K}(\widehat{S^0}\text{-}\mathbf{Mod}_\text{fin})$ for the p-completed sphere spectrum is described by B\"{o}kstedt-Hsiang-Madsen \cite{BHM}.


To understand the spectrum $\mathbb{K}(\widehat{S^{0,0}}\text{-}\mathbf{Mod}^{\tau\text{-tor}}_\text{fin})$, we consider the inclusion functor
\begin{equation} \label{equ:in}
\widehat{S^{0,0}}/\tau\text{-}\mathbf{Mod}_\text{fin} \longrightarrow \widehat{S^{0,0}}\text{-}\mathbf{Mod}^{\tau\text{-tor}}_\text{fin}.
\end{equation}


We propose the following question.
\begin{ques}\label{devi}
Does this inclusion functor (\ref{equ:in})
induce an equivalence on non-connective algebraic $K$-theory spectra?
\end{ques}

Our Question~\ref{devi} is an example of the d\'evissage question for algebraic K-theories. It is known to be false in some situations (see Antieau-Barthel-Gepner \cite{ABG}). 

Let $\text{BP}_*\text{BP}\text{-}\mathbf{Comod}_\text{fin}^\textup{ev}$ be the subcategory of $\text{BP}_*\text{BP}\text{-}\mathbf{Comod}^\textup{ev}$ on those comodules whose underlying $\textup{BP}_*$-module is finitely presented.
If the answer to Question \ref{devi} is yes, then by the theorem of the heart due to Barwick \cite{Barwick12} and Theorem \ref{intro:mainthmpart1}, we have the following isomorphism for all $i\geq 0$ (we require this condition since Barwick's theorem only applies to the connective K-theories),
$$\mathbb{K}_i(\widehat{S^{0,0}}\text{-}\mathbf{Mod}_\textup{fin}) \cong \mathbb{K}_i(\widehat{S^0}\text{-}\mathbf{Mod}_\textup{fin})\oplus \mathbb{K}_i({\textup{BP}_*\textup{BP}\text{-}\mathbf{Comod}}^{\textup{ev}}_\textup{fin}).$$

If we further regard the category ${\textup{BP}_*\textup{BP}\text{-}\mathbf{Comod}}^{\textup{ev}}_\text{fin}$ as the category $\mathbf{Coh}(\mathcal{M}_{FG})$ of coherent sheaves over the moduli stack $\mathcal{M}_{FG}$ of formal groups \cite{Go}, and the answer to Question \ref{devi} is yes, 
then there is an isomorphism for all $i$,
$$ \mathbb{K}_i(\widehat{S^{0,0}}) \cong \mathbb{K}_i(\widehat{S^0})\oplus \mathbb{K}_i(\mathbf{Coh}(\mathcal{M}_{FG})).$$

\section{H$\mathbb{F}_p^{\textup{mot}}$-completion} \label{Hcompletion}

Let $R$ be an $E_\infty$-algebra in a symmetric monoidal stable $\infty$-category $\mathcal{C}$ with the unit object $S$.
\begin{defn} \label{d71}
For any object $Z$ in $\mathcal{C}$, we define its $R$-completion, denoted by $Z^\wedge_R$, as the totalization of the cosimplicial object
$$Z\otimes R^{\otimes \bullet},$$
where the co-face maps are induced by the unit map $S \rightarrow R$. 
\end{defn}

The $R$-completion $Z^\wedge_R$ in $\mathcal{C}$ has the following properties. 
\begin{enumerate}
\item It commutes with finite limits and finite colimits.
\item It commutes with suspensions and desuspensions.

\item If $Z$ is an $E_\infty$-algebra in $\mathcal{C}$, then $Z^\wedge_R$ is also an $E_\infty$-algebra in $\mathcal{C}$. Moreover, if $Y$ is an $Z$-module, then $Y^\wedge_R$ is an $Z^\wedge_R$-module. 
These are special cases of Corollary~3.2.2.5 in Higher Algebra \cite{HA}.
\end{enumerate}

Now we consider the category $\C\text{-}\mathbf{mot}\text{-}\mathbf{Spectra}$ and the H$\mathbb{F}_p^\textup{mot}$-completion of the sphere spectrum $S^{0,0}$:
$$\widehat{S^{0,0}} = {S^{0,0}}^\wedge_{\textup{H}\mathbb{F}_p^\textup{mot}}.$$
Both the sphere spectrum $S^{0,0}$ and the Eilenberg-Mac Lane spectrum H$\mathbb{F}_p^\textup{mot}$ are $E_\infty$-algebras in ${\C\text{-}\mathbf{mot}\text{-}\mathbf{Spectra}}$. Therefore, the H$\mathbb{F}_p^\textup{mot}$-completed sphere spectrum $\widehat{S^{0,0}}$ is an $E_\infty$-algebra.
 
It is a theorem of Hu-Kriz-Ormsby \cite{HKOrem, HKOcon} that, over any algebraic closed field of characteristic 0, the H$\mathbb{F}_p^\textup{mot}$-completion of the sphere spectrum and the usual $p$-completion of the motivic sphere spectrum have isomorphic motivic homotopy groups.

For the effect of the H$\mathbb{F}_p^\textup{mot}$-completion of the sphere spectrum on homotopy groups, Hu-Kriz-Ormsby \cite{HKOrem, HKOcon} pointed out that there is a short exact sequence on homotopy groups of the uncompleted sphere spectrum $S^{0,0}$ and the H$\mathbb{F}_p^\textup{mot}$-completed sphere spectrum $\widehat{S^{0,0}}$:
\begin{displaymath}
\xymatrix{
0 \ar[r] & \Ext^1_\mathbb{Z}(\mathbb{Z}/p^\infty, \pi_{s,w}{{S^{0,0}}}) \ar[r] & \pi_{s,w} \widehat{S^{0,0}} \ar[r] & \Hom_\mathbb{Z}(\mathbb{Z}/p^\infty, \pi_{s-1,w} {{S^{0,0}}}) \ar[r] & 0.
}	
\end{displaymath}
See \cite{HKOcon} for a general discussion regarding the effect of homotopy groups with respect to the H$\mathbb{F}_p^\textup{mot}$-completion.


Now we consider the cellular motivic spectrum MGL. Recall from \cite{Voe} that
\begin{displaymath}
	\textup{MGL} = \colim_{n,k \rightarrow \infty} \Sigma^{-2n,-n} \textup
{Thom}(V(k,n)).
\end{displaymath}
Here $V(k,n)$ is the tautological bundle over the Grassmannian $Gr(k,n)$ for $k\geq n$, which is the smooth scheme of complex n-planes in $\mathbb{C}^{k}$, and $\textup{Thom}(V(k,n))$ is its associated motivic Thom spectrum.

Recall that we have defined in Section~1.1 that
$$\textup{MU}^\textup{mot} = \textup{MGL} \wedge_{S^{0,0}} \widehat{S^{0,0}}.$$
By adjunction, for any $\widehat{S^{0,0}}$-module $X$, its MGL-completion in the category of $\C\text{-}\mathbf{mot}\text{-}\mathbf{Spectra}$ can be identified as its $\textup{MU}^\textup{mot}$-completion in the category $\widehat{S^{0,0}}\text{-}\mathbf{Mod}$.

The following proposition states that it has the same homotopy groups of the H$\mathbb{F}_p^\textup{mot}$-completion of MGL.

\begin{prop} \label{MUmot}
The natural map
$$\textup{MU}^\textup{mot} = \textup{MGL} \wedge_{S^{0,0}} \widehat{S^{0,0}} \rightarrow \textup{MGL}^\wedge_{\textup{H}\mathbb{F}_p^\textup{mot}}$$
induces an isomorphism on $\pi_{*,*}$.
\end{prop}


We prove Proposition~\ref{MUmot} by using Lemma~11 of \cite{HKOrem}. Recall from \cite{HKOrem} that a motivic cellular spectrum $X$ is $k$-connective, if $\pi_{s,w}X = 0$ for all $s$ and $w$ such that $s-w<k$. A cellular map $f:X\rightarrow Y$ between cellular motivic spectra is a $k$-equivalence, if its cofiber is $(k+1)$-connective. Lemma~11 of \cite{HKOrem} states that if $X$ is $k$-connective, then $X^\wedge_{\textup{H}\mathbb{F}_p^\textup{mot}}$ is also $k$-connective.

For MGL, recall from Schubert calculus (see Griffiths-Harris \cite[Section~1.5]{GH} in the classical setting and Wendt \cite[Proposition~2.2]{Wen} for adaption to the motivic setting) that the map
$$\Sigma^{-2n,-n} \textup{Thom}(V(k,n)) \rightarrow \textup{MGL}$$
is $m$-connective, where $m=\textup{min}\{n-1, \ k-n-1\}$.

For any finite cellular motivic spectrum $X$, we have
$$X \wedge_{S^{0,0}} \widehat{S^{0,0}} \simeq X^\wedge_{\textup{H}\mathbb{F}_p^\textup{mot}},$$
since both sides commute with finite colimits.

Because $\textup{Thom}(V(k,n))$ is a finite motivic spectrum for all $n$ and $k$, Proposition~\ref{MUmot} follows from the following Lemma~\ref{HFplemma}.

\begin{lemma} \label{HFplemma}
Suppose that $X$ is the colimit of motivic cellular spectra
\begin{displaymath}
\xymatrix{
X_1 \ar[r] & X_2 \ar[r] & \cdots.
}	
\end{displaymath}
Suppose further that there exists an increasing sequence $m_n$ of natural numbers
$$\lim_{n\rightarrow \infty} m_n =\infty$$ 
such that the map $X_n \rightarrow X$ is an $m_n$-equivalence for all $n$.
Then the map
$$ \colim ({(X_n)}^\wedge_{\textup{H}\mathbb{F}_p^\textup{mot}}) \rightarrow X^\wedge_{\textup{H}\mathbb{F}_p^\textup{mot}}$$
induces an isomorphism on $\pi_{*,*}$.
\end{lemma}

\begin{proof}



By assumption, the map $X_n \rightarrow X$ is an $m_n$-equivalence. By Lemma~11 of \cite{HKOrem} and above discussion, the map ${(X_n)}^\wedge_{\textup{H}\mathbb{F}_p^\textup{mot}} \rightarrow X^\wedge_{\textup{H}\mathbb{F}_p^\textup{mot}}$ is also an $m_n$-equivalence.

Taking the colimit, we have that
$$\pi_{*,*}\colim ((X_n)^\wedge_{\textup{H}\mathbb{F}_p^\textup{mot}} )\cong \colim \pi_{*,*}(X_n)^\wedge_{\textup{H}\mathbb{F}_p^\textup{mot}} \cong \pi_{*,*}X^\wedge_{\textup{H}\mathbb{F}_p^\textup{mot}}.$$
\end{proof}

\part{Equivalence of spectral sequences} \label{chap:equivss}

\section{Main theorem of Part 2}

The algebraic Novikov spectral sequence is introduced by Novikov \cite{Novikov} and Miller \cite{MillerThesis}. Ravenel's green book \cite{Ravenel} and Andrews-Miller's paper \cite{AM16} are also good references for this material.

\begin{thm}(Novikov \cite{Novikov}, Miller \cite{MillerThesis})
There exists a tri-graded spectral sequence with
$$E_1^{s,i,t} = \Ext^{s,t}_{\textup{BP}_*\textup{BP}/I}(\textup{BP}_*/I, I^{i}/I^{i+1}),$$
$$d_r: E_r^{s,i,t} \longrightarrow E_r^{s+1,i+r,t},$$
converging to
$$\Ext^{s,t}_{\textup{BP}_*\textup{BP}}(\textup{BP}_*,\textup{BP}_*).$$
Here $I=(p,v_1,v_2,\cdots)$ is the augmentation ideal of $\textup{BP}_*$.
\end{thm}

To compare it with the motivic Adams spectral spectral, which is studied by Morel, Dugger-Isaksen and Hu-Kriz-Ormsby \cite{MorelAdams, DuggerIsaksenMASS, HKOrem}, we re-grade the algebraic Novikov spectral sequence.

\begin{defn} \label{defn:afil}
We define the $a$-filtration of the algebraic Novikov spectral sequence
$$a = i + s.$$
\end{defn}

The goal of Part 2 of this paper is to prove the following Theorem \ref{thm:isoofss}.

\begin{thm} \label{thm:isoofss}
At each prime $p$, there is an isomorphism of tri-graded spectral sequences between the motivic Adams spectral sequence for $\widehat{S^{0,0}}/\tau$, which converges to the motivic homotopy groups of $\widehat{S^{0,0}}/\tau$, and the regraded algebraic Novikov spectral sequence, which converges to the Adams-Novikov $E_2$-page for the sphere spectrum.

The indexes are indicated in the following diagram:
\begin{displaymath}
    \xymatrix{
  \Ext_{\textup{BP}_*\textup{BP}/I}^{s,2w}(\textup{BP}_*/I, I^{a-s}/I^{a-s+1}) \ar@{=>}[dd]|{\mathbf{Algebraic \ Novikov \ SS}} \ar[rr]^-{\cong} & & \Ext_{{A}^{mot}_{*,*}}^{a, 2w-s+a, w}(\mathbb{F}_p[\tau], \mathbb{F}_p) \ar@{=>}[dd]|{\mathbf{Motivic \ Adams \ SS}} \\
  & & \\
  \Ext_{\textup{BP}_*\textup{BP}}^{s,2w}(\textup{BP}_*, \textup{BP}_*) \ar[rr]^{\cong} & & \pi_{2w-s,w}\widehat{S^{0,0}}/\tau.
    }
\end{displaymath}
Here ${A}_{*,*}^{mot}$ is the motivic mod $p$ dual Steenrod algebra.
\end{thm}

    }


\section{The equivalence to the motivic Adams spectral sequence} \label{eight}

\subsection{The algebraic Novikov tower}
In this subsection, we write down the algebraic Novikov tower using the newly defined $a$-filtration explicitly.

	
Recall from \cite[Definition~A.1.2.7]{Ravenel} that a $\textup{BP}_*\textup{BP}$-comodule is \emph{relative injective} if it is a direct summand of a $\textup{BP}_*\textup{BP}$-comodule of the form $\textup{BP}_*\textup{BP}\otimes_{\textup{BP}_*}M$ for some $\textup{BP}_*$-module $M$. Recall from \cite[Definition~A.1.2.10]{Ravenel} that, for the $\textup{BP}_*\textup{BP}$-comodule $\textup{BP}_*$, its \emph{relative injective resolution}
\begin{equation}\label{reli} 
\textup{BP}_* \longrightarrow C_0^0 \longrightarrow C_0^1 \longrightarrow \cdots 
\end{equation}
is a long exact sequence in the abelian category of $\textup{BP}_*\textup{BP}$-comodules, that satisfies the following two conditions: 
\begin{enumerate}
\item The long exact sequence (\ref{reli}) is split exact as $\textup{BP}_*$-modules.
\item Each comodule $C_0^s$ is relative injective.
\end{enumerate}

For now on, we fix such a relative injective resolution $C_0^*$ of $\textup{BP}_*$ that is concentrated in even internal degrees. Such a relative injective resolution exists (for example the cobar complex).

For $a \geq 1$, let $C_a^*$ be the sub cochain complex of $C_0^*$ defined by 
$$C_a^s =  I^{a-s}C_0^s.$$
Since $I$ is an invariant ideal of $\textup{BP}_*$, each $C_a^s$ is a sub $\textup{BP}_*\textup{BP}$-comodule of $C_0^s$. It is understood that $I^r = \textup{BP}_*$ for $r\leq 0$. Therefore, for $s\geq a$, we have $C_a^s = C_0^s$.

For $a \geq 0$, let $Q_a^*$ be the quotient cochain complex of the inclusion map
\begin{displaymath}
    \xymatrix{
  C_{a+1}^* \ar[r]^{i_a} & C_a^*.
    }
\end{displaymath}
Therefore, we have a tower of cochain complexes, which induces 
the following tower in the derived category of $\textup{BP}_*\textup{BP}$-comodules:
\begin{displaymath}
    \xymatrix{
\textup{BP}_* \ar[r]^-{\simeq} & C^*_0 \ar[d]^{q_0} & C^*_1 \ar[d]^{q_1} \ar[l]^-{i_0} & C^*_2 \ar[d]^{q_2} \ar[l]^{i_1} & \cdots \ar[l]^{i_2} \\
  & Q^*_0 & Q^*_1 & Q^*_2 &
    }
\end{displaymath}
For $s \geq a+1$,
$$ Q_a^s = I^{a-s}C_0^s/I^{a-s+1}C_0^s  = 0.$$
So in particular, the cochain complex $Q_a^*$ is bounded. This implies that each cochain complex $C_a^*$ has bounded cohomology. Therefore, although the cochain complexes $C_a^*$ are unbounded, they live in the category $\mathcal{D}^b(\textup{BP}_*\textup{BP}\text{-}\textbf{Comod}^\textup{ev})$. 

Applying the functor 
$$\mathbf{R}^{*,*}\Hom_{\textup{BP}_*\textup{BP}}(\textup{BP}_*,-),$$
where $\mathbf{R}^{*,*}\Hom_{\textup{BP}_*\textup{BP}}(-,-)$ is the derived homomorphisms in the category\\ $\mathcal{D}^b (\textup{BP}_*\textup{BP}\text{-}\textbf{Comod}^\textup{ev})$, we get a spectral sequence
with the $E_1$-page
$$\mathbf{R}^{*,*}\Hom_{\textup{BP}_*\textup{BP}}(\textup{BP}_*, Q_a^*),$$
converging to 
$$\mathbf{R}^{*,*}\Hom_{\textup{BP}_*\textup{BP}}(\textup{BP}_*,\textup{BP}_*) = \Ext^{*,*}_{\textup{BP}_*\textup{BP}}(\textup{BP}_*,\textup{BP}_*).$$
This is the regraded algebraic Novikov spectral sequence.

\subsection{Characterization of Adams towers}

Recall that we denote by $\textup{H}\mathbb{F}_p^\textup{mot}$ the motivic mod $p$ Eilenberg-Mac Lane spectrum. It is shown by Hu-Kriz-Ormsby \cite{HKOrem} and Hoyois \cite{Hoyois} that $\textup{H}\mathbb{F}_p^\textup{mot}$ is cellular. We denote 
$$\textup{H}\mathbb{F}_p^\textup{mot}/\tau := \widehat{S^{0,0}}/\tau \wedge_{\widehat{S^{0,0}}} \textup{H}\mathbb{F}_p^\textup{mot}.$$
\newtheorem{condition}[thm]{Condition}

\begin{defn} \label{defn:motAdamstower}
A tower in $\widehat{S^{0,0}}/\tau\text{-}\mathbf{Mod}^b_{\textup{harm}}$
\begin{displaymath} \label{Tower0}
    \xymatrix{
  \widehat{S^{0,0}}/\tau \ar[r]^-{\simeq} & X_0 \ar[d]^{f_0} & X_1 \ar[d]^{f_1} \ar[l]^-{g_0} & X_2 \ar[d]^{f_2} \ar[l]^{g_1} & \cdots \ar[l]^{g_2} \\
  & K_0 & K_1 & K_2 &
    }
\end{displaymath}
 is a \emph{motivic Adams tower} if
\begin{enumerate}\label{aasss}
\item each motivic spectrum $K_m$ is a retract of a wedge of suspensions of $\textup{H}\mathbb{F}_p^\textup{mot}/\tau$, \label{Adams1}
\item each map $f_m: X_m\longrightarrow K_m$ induces an epimorphism on the $\textup{H}\mathbb{F}_p^\textup{mot}$-cohomology. Or equivalently, each map $g_m: X_{m+1}\longrightarrow X_{m}$ induces the zero map on the $\textup{H}\mathbb{F}_p^\textup{mot}$-cohomology. \label{Adams2}
\end{enumerate}
\end{defn}
By the adjunction between modules over $\widehat{S^{0,0}}$ and $\widehat{S^{0,0}}/\tau$ and that $\widehat{S^{0,0}}/\tau$ is Spanier-Whitehead dual to itself up to a bidegree shift (see \cite[Proposition 4.3]{GheCt} for a proof for example), it is equivalent to check that each map $g_m$ induces the zero map on 
$$[-, \textup{H}\mathbb{F}_p^\textup{mot}/\tau]_{\widehat{S^{0,0}}/\tau}$$
in Condition $(2)$.

From the general discussions (see \cite{EM, MillerAdams, MillerNotes, Chris} for example), all such towers are equivalent to each other in the sense that there exist towers maps that induce canonical isomorphisms on the $E_2$-pages.

Dugger-Isaksen \cite{DuggerIsaksenMASS} uses the cobar construction to define the motivic Adams spectral sequence for $\widehat{S^{0,0}}/\tau$, which satisfies the two conditions in Definition \ref{defn:motAdamstower}. Therefore, the motivic Adams spectral sequence for $\widehat{S^{0,0}}/\tau$ by Dugger-Isaksen \cite{DuggerIsaksenMASS} is canonically isomorphic to the motivic Adams spectral sequence defined by any motivic Adams tower satisfying the two conditions in Definition \ref{defn:motAdamstower}.



Having the regraded algebraic Novikov tower in the category $\mathcal{D}^b({\textup{BP}_*\textup{BP}\text{-}\mathbf{Comod}}^{\textup{ev}})$, we use the equivalence of stable $\infty$-categories in Theorem \ref{thm:equivCtCellcompalg} and Proposition \ref{prop:bpbpmumucomod} in Part 1
\begin{displaymath} 
    \xymatrix{
\mathcal{D}^b({\textup{BP}_*\textup{BP}\text{-}\mathbf{Comod}}^{\textup{ev}}) \ar[r]^{\simeq} & \mathcal{D}^b(\textup{MU}_*\textup{MU}\text{-}\mathbf{Comod}^\textup{ev}) \ar[r]^-{\simeq} & \widehat{S^{0,0}}/\tau\text{-}\mathbf{Mod}_{\textup{harm}}^b.
}
\end{displaymath}
to get a tower in the category $\widehat{S^{0,0}}/\tau\text{-}\mathbf{Mod}^b_{\textup{harm}}$:
\begin{displaymath} 
    \xymatrix{
  \widehat{S^{0,0}}/\tau \ar[r]^-{\simeq} & Y_0 \ar[d] & Y_1 \ar[d] \ar[l] & Y_2 \ar[d] \ar[l] & \cdots \ar[l] \\
  & L_0 & L_1 & L_2 &
    }
\end{displaymath}

\begin{prop} \label{conditions for algNSS}
	The above tower is a motivic Adams tower in the sense of Definition~\ref{defn:motAdamstower}, if the following two conditions are satisfied for the re-graded algebraic Novikov tower in the category $\mathcal{D}^b({\textup{BP}_*\textup{BP}\text{-}\mathbf{Comod}}^{\textup{ev}})$:
	\begin{enumerate}
		\item each $Q_a^*$ is quasi-isomorphic to a retract of a direct sum of shifts of $\textup{BP}_*\textup{BP}/I$,
		\item each map $q_a: C^*_a\longrightarrow Q^*_a$ induces an epimorphism on $\mathbf{R}^{*,*}\Hom_{\textup{BP}_*}(-, \mathbb{F}_p)$. Or equivalently, each map $i_a: C^*_{a+1}\longrightarrow C^*_{a}$ induces the zero map on $\mathbf{R}^{*,*}\Hom_{\textup{BP}_*}(-, \mathbb{F}_p)$.
	\end{enumerate} 
\end{prop}

\begin{proof}
By the following Lemmas~\ref{lem:BPBPmodi} and \ref{lem:derivedhomo}, the two conditions in this proposition correspond to the two conditions in Definition~\ref{defn:motAdamstower}.
\end{proof}

For the first condition in Proposition~\ref{conditions for algNSS}, we identify the $\textup{BP}_*\textup{BP}$-comodule that corresponds to $\textup{H}\mathbb{F}_p^\textup{mot}$ under the equivalences of the hearts in Proposition \ref{prop:CtCellequivheart} and Proposition \ref{prop:bpbpmumucomod}. 
\begin{displaymath} 
    \xymatrix{ 
\widehat{S^{0,0}}/\tau\text{-}\mathbf{Mod}_{\textup{harm}}^\heartsuit \ar[rr]_-{\simeq}^-{\textup{MU}^\textup{mot}_{*,*}} & & \textup{MU}_*\textup{MU}\text{-}\textbf{Comod}^\textup{ev} \ar[rrr]_-{\simeq}^-{\textup{BP}_*\otimes_{\textup{MU}_*}-} & & & \textup{BP}_*\textup{BP}\text{-}\textbf{Comod}^\textup{ev}. 
}
\end{displaymath}

\begin{lemma} \label{lem:BPBPmodi}
Under the equivalences of the hearts, $\textup{H}\mathbb{F}_p^\textup{mot}/\tau$ corresponds to 
$\textup{BP}_*\textup{BP}/I$.
 %
\end{lemma}

\begin{proof}

Since $\textup{H}\mathbb{F}_p^\textup{mot}$ is an $\textup{MU}^\textup{mot}$-module, both $\textup{MU}^\textup{mot}$ and $\textup{H}\mathbb{F}_p^\textup{mot}$ are cellular, 
and $\textup{MU}^{\textup{mot}}_{*,*}\textup{MU}^{\textup{mot}}/\tau$ is free over $\textup{MU}^{\textup{mot}}_{*,*}/\tau$, 
by Dugger-Isaksen's universal coefficient spectral sequence \cite[Proposition 7.7]{DuggerIsaksen} in the category $\textup{MU}^{\textup{mot}}/\tau\text{-}\mathbf{Mod}_{\textup{cell}}$, we have 
\begin{align*}
\textup{MU}^{\textup{mot}}_{*,*}\textup{H}\mathbb{F}_p^\textup{mot}/\tau 
& \cong  \textup{MU}^{\textup{mot}}_{*,*}\textup{MU}^{\textup{mot}}/\tau \otimes_{\textup{MU}^{\textup{mot}}_{*,*}/\tau} \mathbb{F}_p,
\end{align*}
which is isomorphic to $\textup{MU}_{*}\textup{MU} \otimes_{\textup{MU}_{*}} \mathbb{F}_p$ forgetting the motivic weight.

Therefore, under the equivalences in Proposition \ref{prop:CtCellequivheart} and Proposition \ref{prop:bpbpmumucomod}, the $\widehat{S^{0,0}}/\tau$-module $\textup{H}\mathbb{F}_p^\textup{mot}/\tau$ corresponds to
\begin{align*}
\textup{BP}_*\otimes_{\textup{MU}_*}\textup{MU}_*\textup{MU}\otimes_{\textup{MU}_*}\mathbb{F}_p & \cong \textup{BP}_*\textup{MU}\otimes_{\textup{MU}_*}\mathbb{F}_p \\
& \cong \textup{BP}_*\textup{MU}\otimes_{\textup{MU}_*} \textup{BP}_* \otimes_{\textup{BP}_*}\mathbb{F}_p \\
& \cong \textup{BP}_*\textup{BP}\otimes_{\textup{BP}_*}\mathbb{F}_p \\
& \cong \textup{BP}_*\textup{BP}/I.
\end{align*}
The first and third isomorphisms follow from the Landweber exactness of $\textup{BP}_*$.
This completes the proof.
\end{proof}

For the second condition in Proposition~\ref{conditions for algNSS}, we have the following Lemma~\ref{lem:derivedhomo}.

\begin{lemma} \label{lem:derivedhomo}
Suppose that $X$ is in the category $\widehat{S^{0,0}}/\tau\text{-}\mathbf{Mod}^b_{\textup{harm}}$ and that $C^*(X)$ is the cochain complex of $\textup{BP}_*\textup{BP}$-comodules representing the image of $X$ under the equivalence in Theorem \ref{thm:equivCtCellcompalg} in Part 1. 

Then we have 
$$[\Sigma^{*,*}X, \textup{H}\mathbb{F}_p^\textup{mot}/\tau]_{\widehat{S^{0,0}}/\tau} \cong \mathbf{R}^{*,*}\Hom_{\textup{BP}_*}(C^*(X), \mathbb{F}_p),$$
where $\mathbf{R}^{*,*}\Hom_{\textup{BP}_*}(-,-)$ is the derived homomorphism in the derived category of $\textup{BP}_*$-modules.
\end{lemma}


\begin{proof}
We have
\begin{align*}
	[\Sigma^{*,*}X, \ \textup{H}\mathbb{F}_p^\textup{mot}/\tau]_{\widehat{S^{0,0}}/\tau} & \cong \mathbf{R}^{*,*}\Hom_{\textup{BP}_*\textup{BP}}(C^*(X), \ \textup{BP}_*\textup{BP} \otimes_{\textup{BP}_*} \mathbb{F}_p) \\
	& \cong \mathbf{R}^{*,*}\Hom_{\textup{BP}_*}(C^*(X), \ \mathbb{F}_p).
\end{align*}
The first isomorphism follows from Theorem~\ref{thm:equivCtCellcompalg} and Lemma~\ref{lem:BPBPmodi}, and the second isomorphism follows from the adjunction of the derived functor of $\textup{BP}_*\textup{BP} \otimes_{\textup{BP}_*} - $ and the forgetful functor between the derived categories of $\textup{BP}_*$-modules and $\textup{BP}_*\textup{BP}$-comodules.

\end{proof}


In the rest of this section, we check that the two conditions in Proposition~\ref{conditions for algNSS} are satisfied by the regraded algebraic Novikov tower in the category $\mathcal{D}^b (\textup{BP}_*\textup{BP}\text{-}\textbf{Comod}^\textup{ev})$.

\subsection{Proof of the first condition}

\begin{lemma} \label{HLemma}
Suppose that $N$ is a relative injective 
$\textup{BP}_*\textup{BP}$-comodule that is concentrated in even degrees. Then for any $a$, $I^aN/I^{a+1}N$ 
is isomorphic to a retract of
a direct sum of shifts of $\textup{BP}_*\textup{BP}/I$.
\end{lemma}

\begin{proof}

Without loss of generality we assume that $N$ has the form $\textup{BP}_*\textup{BP}\otimes_{\textup{BP}_*}M$.

Because $I$ is an invariant ideal,
$${I^a\textup{BP}_*\textup{BP}\otimes_{\textup{BP}_*}M \over I^{a+1}\textup{BP}_*\textup{BP}\otimes_{\textup{BP}_*}M} \cong \textup{BP}_*\textup{BP}\otimes_{\textup{BP}_*} I^aM/I^{a+1}M$$
So it suffices to show that for any $\textup{BP}_*/I$-module $M'$, $\textup{BP}_*\textup{BP}\otimes_{\textup{BP}_*}M'$ corresponds to a direct sum of shifts of $\textup{BP}_*\textup{BP}/I$. This is straight forward.


\end{proof}

Now we prove that the regraded algebraic Novikov tower satisfies the first condition in Proposition~\ref{conditions for algNSS}.

\begin{prop} \label{prop:cttower1} \leavevmode
\begin{enumerate}
\item
All differentials in the cochain complex $Q_a^*$ are zero, and therefore $Q_a^*$ splits as a direct sum of cochain complexes that are concentrated in one cohomological degree. 
\item
Each $Q_a^s$ is
a retract of a direct sum of shifts of $\textup{BP}_*\textup{BP}/I$.
\end{enumerate}
Therefore, Condition (1) of Proposition \ref{conditions for algNSS} is satisfied.
\end{prop}

\begin{proof}
In $Q_a^*$, all differentials $Q_a^s \longrightarrow Q_a^{s+1}$ have the form
$$I^{a-s}C_0^s/I^{a-s+1}C_0^s \longrightarrow I^{a-s-1}C_0^{s+1}/I^{a-s}C_0^{s+1}.$$
They are all zero since they are $\textup{BP}_*$-linear. The second claim follows from Lemma~\ref{HLemma} and the definition of $Q_a^s$.
%
\end{proof}

\subsection{Proof of the second condition}

We will use the following Lemma~\ref{lem:tech} in the proof of Proposition~\ref{prop:cttower2}. The proof of Lemma~\ref{lem:tech} is technical. We will postpone it to the last subsection of this section.

\begin{lemma} \label{lem:tech}
The following homomorphisms 
$$\Ext^{*,*}_{\textup{BP}_*}(I^{a}, \mathbb{F}_p) \longrightarrow \Ext^{*,*}_{\textup{BP}_*}(I^{a+1}, \mathbb{F}_p),$$
that are induced by the inclusions $I^{a+1} \longrightarrow I^a$ are zero for all $a \geq 0$.
\end{lemma}

Now we prove that the regrade algebraic Novikov tower satisfies the second condition of Proposition~\ref{conditions for algNSS}.
\begin{prop} \label{prop:cttower2}
Each map $i_a: C^*_{a+1}\longrightarrow C^*_{a}$ induces the zero map on $\mathbf{R}^{*,*}\Hom_{\textup{BP}_*}(-, \mathbb{F}_p)$.
\end{prop}

\begin{proof}
Because we are computing derived Hom in the derived category of $\textup{BP}_*$-modules, we can first apply a forget functor from $\textup{BP}_*\textup{BP}$-comodules to $\textup{BP}_*$-modules on our complexes.

From the definition of a relative injective resolution, $C_0^*$ splits in the category of $\textup{BP}_*$-modules as a direct sum of cochain complexes 
$$C_0^* = \displaystyle\bigoplus_{j \geq 0} D_{0,j}^*,$$
where $D_{0,0}^*$ is isomorphic to the cochain complex
\begin{displaymath}
    \xymatrix{
  \textup{BP}_* \ar[r] & 0 \ar[r] & 0 \ar[r] & \cdots }
  \end{displaymath}  
and for $j \geq 1$, $D_{0,j}^*$ is a cochain complex of the form
  \begin{displaymath}
  \xymatrix{
  \cdots \ar[r] & 0 \ar[r] & N_j \ar[r]^{\textup{id}} & N_j \ar[r] & 0 \ar[r] & \cdots     }
\end{displaymath}
that is concentrated in cohomological degrees $j$ and $j-1$, where $N_j$ is a $\textup{BP}_*$-module.


The algebraic Novikov filtration only depends on the underlying $\textup{BP}_*$-module structure. So we have
$$C_a^* = \displaystyle\bigoplus_{j \geq 0} D_{a,j}^*$$ where for $j\geq1$, $D_{a,j}^*$ is the subcomplex
$$\cdots\longrightarrow 0\longrightarrow I^{a-j+1}N_j \longrightarrow I^{a-j}N_j\longrightarrow 0 \longrightarrow\cdots.$$
It follows that $D^*_{a,j}$ is quasi-isomorphic to the complex
$$\cdots\longrightarrow0\longrightarrow I^{a-j}N_j/I^{a-j+1}N_j\longrightarrow 0\longrightarrow \cdots.$$

Now we consider the maps
\begin{equation}\label{wer}
\mathbf{R}^{*,*}\Hom_{\textup{BP}_*}(D_{a,j}^*, \mathbb{F}_p) \longrightarrow \mathbf{R}^{*,*}\Hom_{\textup{BP}_*}(D_{a+1,j}^*, \mathbb{F}_p)
\end{equation}
that are induced by the inclusions 
$$D_{a+1,j}^* \longrightarrow D_{a,j}^*.$$
For $j\geq 1$, these maps can be identified as (shifts of)
$$\mathbf{R}^{*,*}\Hom_{\textup{BP}_*}(I^{a-j}N_j/I^{a-j+1}N_j, \mathbb{F}_p) \longrightarrow \mathbf{R}^{*,*}\Hom_{\textup{BP}_*}(I^{a-j+1}N_j/I^{a-j+2}N_j, \mathbb{F}_p).$$
It is clear that the maps 
 $$I^{a-j+1}N_j/I^{a-j+2}N_j \longrightarrow I^{a-j}N_j/I^{a-j+1}N_j$$ 
are all zero. Therefore the maps in (\ref{wer}) are all zero for $j\geq 1$.

For $j=0$, 
we have $D_{a,0}^*$ is the complex
$$I^a \longrightarrow 0 \longrightarrow\cdots$$
and the corresponding maps in (\ref{wer}) can be rewritten as
$$\Ext^{*,*}_{\textup{BP}_*}(I^m, \mathbb{F}_p) \longrightarrow \Ext^{*,*}_{\textup{BP}_*}(I^{m+1}, \mathbb{F}_p).$$
By Lemma \ref{lem:tech}, they are all zero. Therefore, the maps
$$\mathbf{R}^{*,*}\Hom_{\textup{BP}_*}(C_{m+1}^*, \mathbb{F}_p) \longrightarrow \mathbf{R}^{*,*}\Hom_{\textup{BP}_*}(C_{m}^*, \mathbb{F}_p)$$
are all zero, since they are zero on each direct summand.

\end{proof}

Combining Proposition \ref{prop:cttower1} and \ref{prop:cttower2}, we have shown that the regraded algebraic Novikov tower satisfies the two conditions of Proposition~\ref{conditions for algNSS}, and therefore corresponds to a motivic Adams tower for $\widehat{S^{0,0}}/\tau$. This proves that there exists an isomorphism between the regraded algebraic Novikov spectral sequence and the motivic Adams spectral sequence for $\widehat{S^{0,0}}/\tau$.

\subsection{Proof of Lemma \ref{lem:tech}}
We prove Lemma~\ref{lem:tech} in this subsection.

For any $\textup{BP}_*$-module $M$, we have 
$$\Ext_{\textup{BP}_*}^{*,*}(M,\mathbb{F}_p) \cong \Hom_{\mathbb{F}_p}(\textup{Tor}^{\textup{BP}_*}_{*,*}(M,\mathbb{F}_p),\mathbb{F}_p).$$
Lemma~\ref{lem:tech} is implied by its dual statement:

\begin{lemma} \label{lem tech 2}
The maps
$$\textup{Tor}_{*,*}^{\textup{BP}_*}(I^{n+1},\mathbb{F}_p) \rightarrow \textup{Tor}_{*,*}^{\textup{BP}_*}(I^n,\mathbb{F}_p)$$
are zero for $n \geq 0$.
\end{lemma}

\begin{proof}
The powers of $I$ filter $\textup{BP}_*$ as a $\textup{BP}_*$-module, and the $\textup{BP}_*$-action on the associated graded pieces factors through an $\mathbb{F}_p$-action. 

Therefore, we have an associated spectral sequence
$$E_1^{s,t,i} = \textup{Tor}^{\textup{BP}_*}_{s,t}(I^i/I^{i+1},\mathbb{F}_p) \Longrightarrow \textup{Tor}^{\textup{BP}_*}_{s,t}(\textup{BP}_*, \mathbb{F}_p).$$
The $E_1$-page can be identified as 
$$E(\tau_0,\tau_1,\cdots) \otimes \mathbb{F}_p[q_0, q_1, \cdots],$$
since $\textup{Tor}_{*,*}^{\textup{BP}_*}(\mathbb{F}_p, \mathbb{F}_p) = E(\tau_0,\tau_1,\cdots)$, where $E(\tau_0, \tau_1, \cdots)$ is the exterior algebra over $\mathbb{F}_p$ generated by the $\tau_j$'s, with internal degree the same as $v_j$ and homological degree $1$, and $\textup{gr}^*\textup{BP}_* = \mathbb{F}_p[q_0, q_1, \cdots]$, where $q_j$ corresponds to $v_j$.

Using the Koszul complex, it is straight forward to see that this spectral sequence is multiplicative and that $d_1\tau_n = q_n$. Therefore its $E_2$-page is concentrated in degrees $i=0$, 
and the following sequence is exact
$$0 \longrightarrow \textup{Tor}_{*,*}^{\textup{BP}_*}(\textup{BP}_*, \mathbb{F}_p) \longrightarrow E_1^{*,*,0} \longrightarrow E_1^{*,*,1} \longrightarrow \cdots. $$
Now the exact couple looks like this:

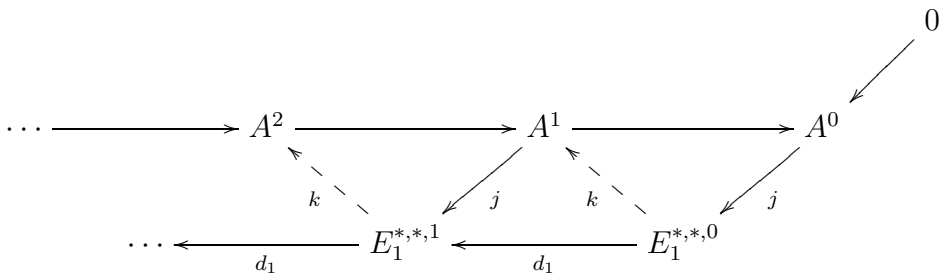
\begin{displaymath}
\xymatrix{
& & & & & & & 0 \ar[ld] &\\
\cdots \ar[rr] & & A^2 \ar[rr] & & A^1 \ar[dl]^-j \ar[rr] & & A^0 \ar[dl]^-j & &\\
& \cdots & & E_1^{*,*,1} \ar[ll]^-{d_1} \ar@{-->}[ul]^-k & & E_1^{*,*,0} \ar[ll]^-{d_1} \ar@{-->}[ul]^-k & & &
}	
\end{displaymath}
where $A^i = \textup{Tor}_{*,*}^{\textup{BP}_*}(I^i, \mathbb{F}_p)$, and the bottom line is exact. All we need to show is that all the $j$ maps are injective. This follows from an induction using a diagram chasing argument.
\end{proof}

\begin{remark}
For the polynomial ring with finitely many generators $\mathbb{Z}_p [x_1, x_2, \cdots, x_t]$, the analogue of  Lemmas~\ref{lem:tech} and \ref{lem tech 2} are well known (see \cite{Sega} for example).
	
\end{remark}

\section{Appendix - computation of some classical Adams differentials} \label{appendix:calculations}

In this appendix, we illustrate the power of the isomorphism of spectral sequences in Theorem \ref{thm:isoofss}, by re-computing certain low filtration and historically more difficult differentials in the range up to the 45-stem at the prime 2. We follow notations in Isaksen's Stable Stems \cite{StableStems} and Isaksen, the second and third author's More Stable Stems \cite{IWX}.

When computing nontrivial differentials in the classical Adams spectral sequence, it is usually harder to give proofs for the ones whose sources are in low Adams filtrations. There are at least two reasons for this. Firstly, there are more potential targets that it could hit, so it means more possibilities to check and rule out. Secondly, on the other hand, elements in high Adams filtrations can usually be detected by certain known spectrum in small chromatic height - for instance, elements above the $1/3$-line can be detected by the $K(1)$-local sphere and many elements around the $1/5$-line can be detected by the spectrum of topological modular forms. This gives ways to compare with Adams spectral sequences of other spectra.

Up to the 45-stem, we list the following 10 nontrivial differentials, whose sources are in low Adams filtrations. Five of them are $d_2$-differentials, four of them are $d_3$-differentials, and one of them is a $d_4$-differential.

\begin{table}[h]

\centering
\begin{tabular}{ l  | c | c }
Adams differential & Stem of the source & Filtration of the source \\ [0.5ex] 
\hline 
$d_2(h_4) = h_0 h_3^2$ & 15 & 1 \\
$d_3(h_0h_4)= h_0d_0$& 15 & 2\\
$d_2(e_0) = h_1^2d_0$ & 17 & 4\\
$d_2(f_0) = h_0^2e_0$ & 18 & 4 \\
$d_2(h_5) = h_0 h_4^2$ & 31 & 1 \\
$d_3(h_0^3 h_5) = h_0 \Delta h_2^2$ & 31 & 4 \\
$d_3(h_2 h_5) = h_0p = h_1d_1$ & 34 & 2 \\
$d_4(h_3 h_5) = h_0x$ & 38 & 2 \\
$d_3(e_1) = h_1t$ & 38 & 4 \\
$d_2(c_2) = h_0f_1$ & 41 & 3 \\
\end{tabular}

\end{table}

Historically, the first four of them were proved by May in his thesis, by comparing with Toda's unstable computation. The next two are obtained by the Hopf invariant one problem and by comparing with the J-spectrum. The elements $\Delta h_2^2$ and $h_0\Delta h_2^2$ were historically called $r$ and $s$, and there is a nontrivial extension in the May spectral sequence that gives us a relation $s = h_0 r$. The last four, except the one on $d_3(e_1)$, were proved by Barratt-Mahowald-Tangora \cite{BMT} using ad hoc methods. In fact, the differentials $d_3(h_2 h_5) = h_0p$ and $d_2(c_2) = h_0f_1$ are both closely related to the nontrivial $\nu$-extension from $h_4^2$ to the element $p$, and the differential $d_4(h_3 h_5) = h_0x$ is closely related to the nontrivial $\sigma$-extension from $h_4^2$ to the element $x$. For the element $e_1$, Barratt-Mahowald-Tangora \cite{BMT} erroneously thought it was a permanent cycle. It was later proved by Bruner \cite{Br1} using power operations that it supports a nontrivial differential $d_3(e_1) = h_1t$.

Now using Theorem \ref{thm:isoofss}, we compare them with the computations of the motivic Adams spectral sequence of $\widehat{S^{0,0}}/\tau$. All five $d_2$-differentials are present in the motivic Adams spectral sequence of $\widehat{S^{0,0}}/\tau$. This gives immediate proofs for all five $d_2$-differentials.

Moreover, the three out of the four $d_3$-differentials except $d_3(h_2h_5)$ are present in the motivic Adams spectral sequence of $\widehat{S^{0,0}}/\tau$. To be careful, one also need to rule out the possibility of shorter differentials - $d_2$'s in these cases. This can be done by multiplying $h_0$ to the proposed $d_2$-differentials and get contradictions. 

For the $d_3$-differential $d_3(h_2 h_5) = h_1d_1$, one can show the following three statements are equivalent, by considering the long exact sequence of motivic homotopy groups associated to the cofiber map of $\tau$.

\begin{enumerate}
	\item There is a differential $d_3(h_2 h_5) = \tau h_1d_1$ in the motivic Adams spectral sequence of $\widehat{S^{0,0}}$.
	\item In homotopy groups, $\{h_2h_5\}$ maps to $\{h_1d_1\}$ under the quotient map from $\widehat{S^{0,0}}/\tau$ to its top cell $S^{1, -1}$.
	\item There is an $\eta$-extension from $h_2h_5$ to $\overline{h_1^2d_1}$ in $\pi_{*,*}(\widehat{S^{0,0}}/\tau)$, where $\overline{h_1^2d_1}$ is the element in the motivic Adams $E_2$-page of $\widehat{S^{0,0}}/\tau$ that corresponds to $h_1^2d_1$ in that of the top cell $\widehat{S^{1, -1}}$.
\end{enumerate}

The statement $(3)$ can be checked in the $E_\infty$-page of the motivic Adams spectral sequence for $\widehat{S^{0,0}}/\tau$, which is isomorphic to the classical Adams-Novikov $E_2$-page. This gives a proof for the $d_3$-differential in the motivic Adams spectral sequence for $\widehat{S^{0,0}}$, and hence the classical $d_3$-differential. The statement $(2)$ is proved by Isaksen in Table 42 of Stable Stems \cite{StableStems}.

At last, the $d_4$-differential $d_4(h_3h_5)$ is also present in the motivic Adams spectral sequence for $\widehat{S^{0,0}}/\tau$. To pull it back and get the $d_4$-differential in the motivic sphere, one need to rule out the possibilities of nonzero $d_2$'s and $d_3$'s. For degree reasons, there is no possible $d_2$'s. To rule out the only $d_3$ possibility that $d_3(h_3h_5) = x$, since $h_3x = h_0^2 g_2$, this would give another $d_3$-differential by multiplying by $h_3$: $d_3(h_3^2h_5) = h_0^2g_2$. However, there is no such $d_3$ in the motivic Adams spectral sequence for $\widehat{S^{0,0}}/\tau$, which gives a contradiction.

In sum, we reprove all 10 nontrivial low filtration differentials up to the 45-stem without much effort. In fact, among all nontrivial differentials up to the 45-stem, there is only one that cannot be proved by our motivic $\widehat{S^{0,0}}/\tau$-method: $d_3(\Delta h_2^2) = h_1d_0^2$. This can be proved by other methods, such as the ad hoc method by Barratt-Mahowald-Tangora \cite{BMT}, the power operation method by Bruner \cite{Br2}, the method of detection by $tm\hspace{-0.1em}f$, and the $RP^\infty$-technique in \cite{Xu2}.

We include the following Isaksen's charts for the reader's reference of the differentials that are discussed in this appendix. 

There are eight charts in total. The first two charts are for the classical Adams spectral sequences. The horizontal degree is $t-s$, i.e., the topological stem, and the vertical degree is $s$, i.e., the Adams filtration. Each dot is a copy of $\mathbb{F}_2$. Vertical lines and lines of slope 1 and $1/3$ correspond to multiplication by $h_0$, $h_1$ and $h_2$ respectively. Lines of negative slopes correspond to differentials. We mark all algebraic generators for completeness.

For the rest of the six charts, we only mark certain low Adams filtration elements and the elements that are relevant to our discussion of differentials in this section.

The second two charts are the $E_2$-pages with $d_2$-differentials of the motivic Adams spectral sequences for $\widehat{S^{0,0}}/\tau$. Each arrow with slope 1 indicating an infinite $h_1$-tower, and each arrow with slope $-2$ is a family of $h_1$-periodic $d_2$-differentials. 

The third two charts are the $E_3$-pages with $d_3$ and $d_4$-differentials of the motivic Adams spectral sequences for $\widehat{S^{0,0}}/\tau$. 

The last two charts are the $E_\infty$-pages of the motivic Adams spectral sequences for $\widehat{S^{0,0}}/\tau$. The blue lines are nontrivial $2, \ \eta$ and $\nu$-extensions.


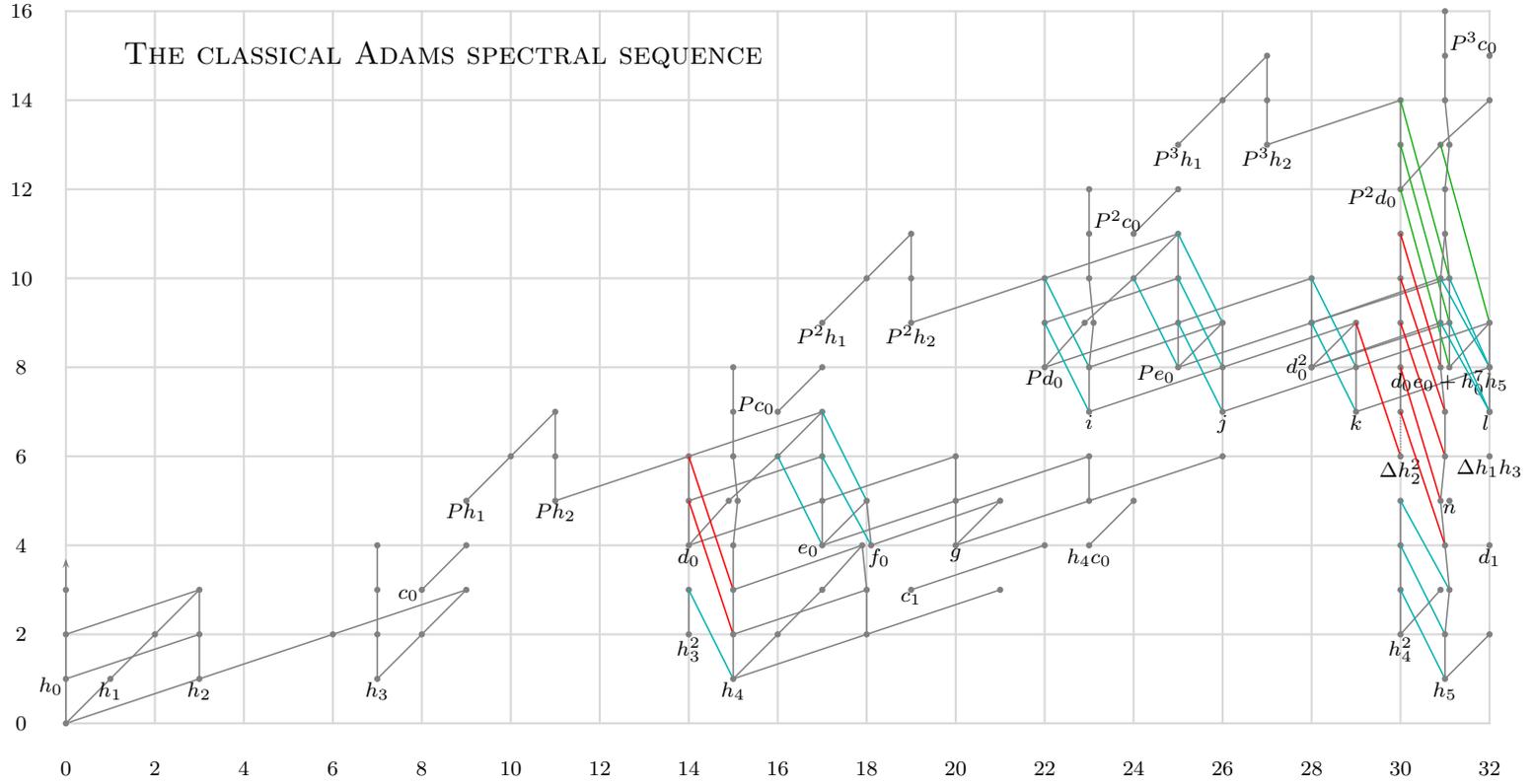
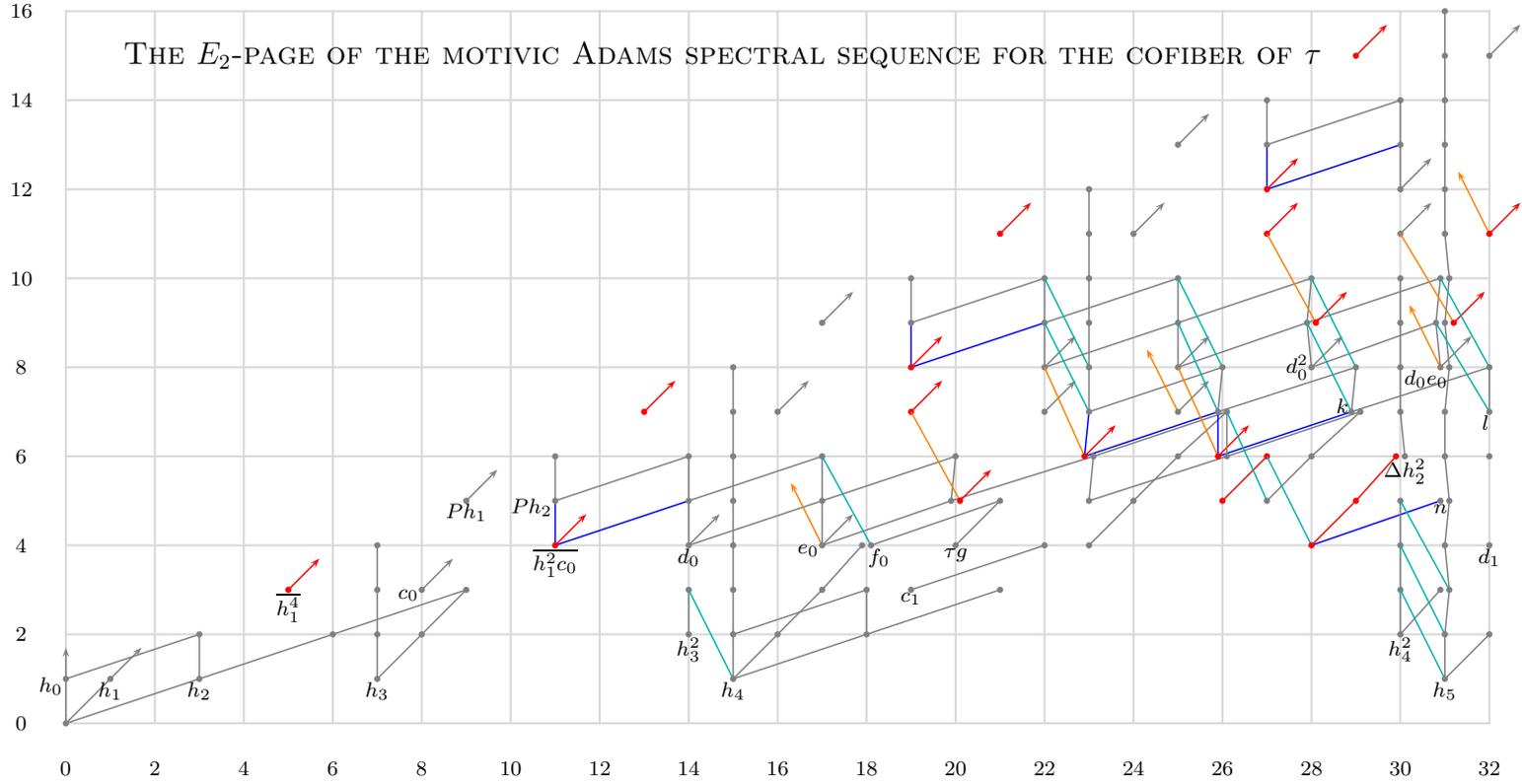
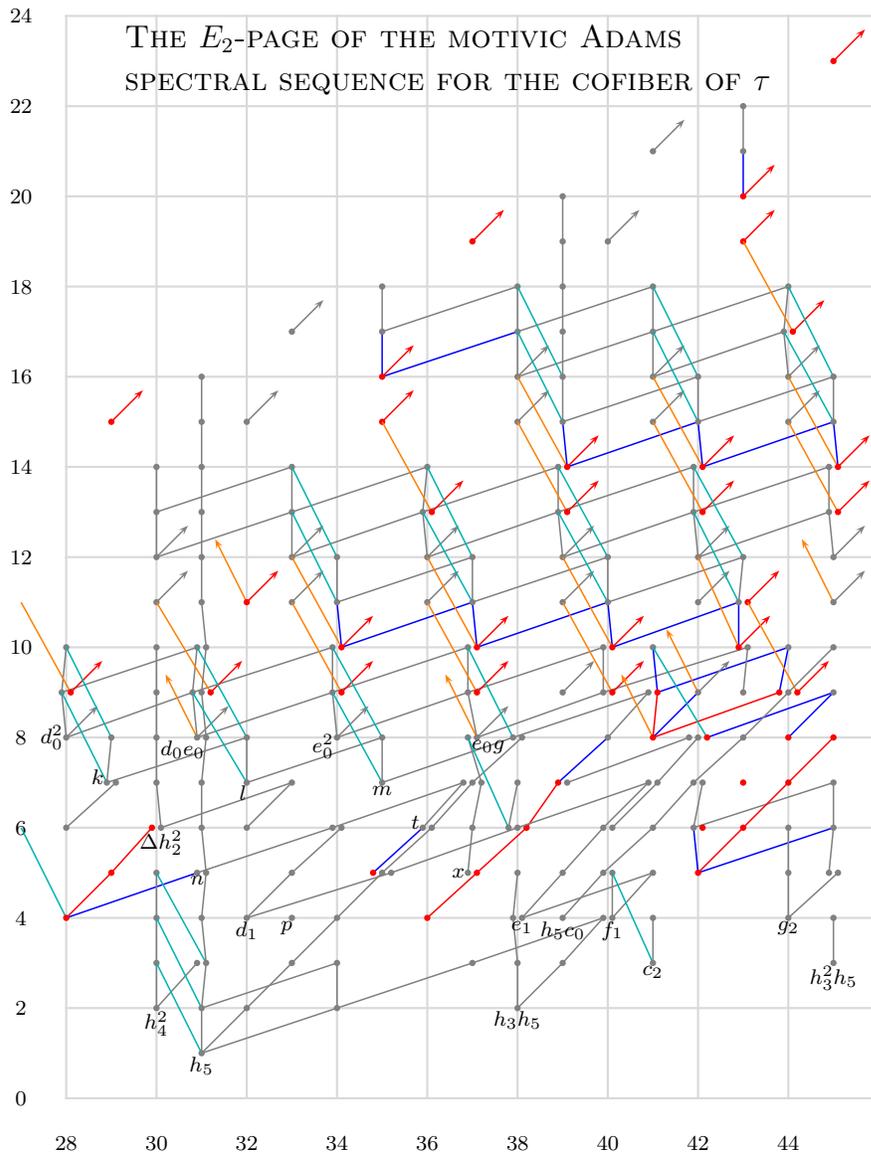
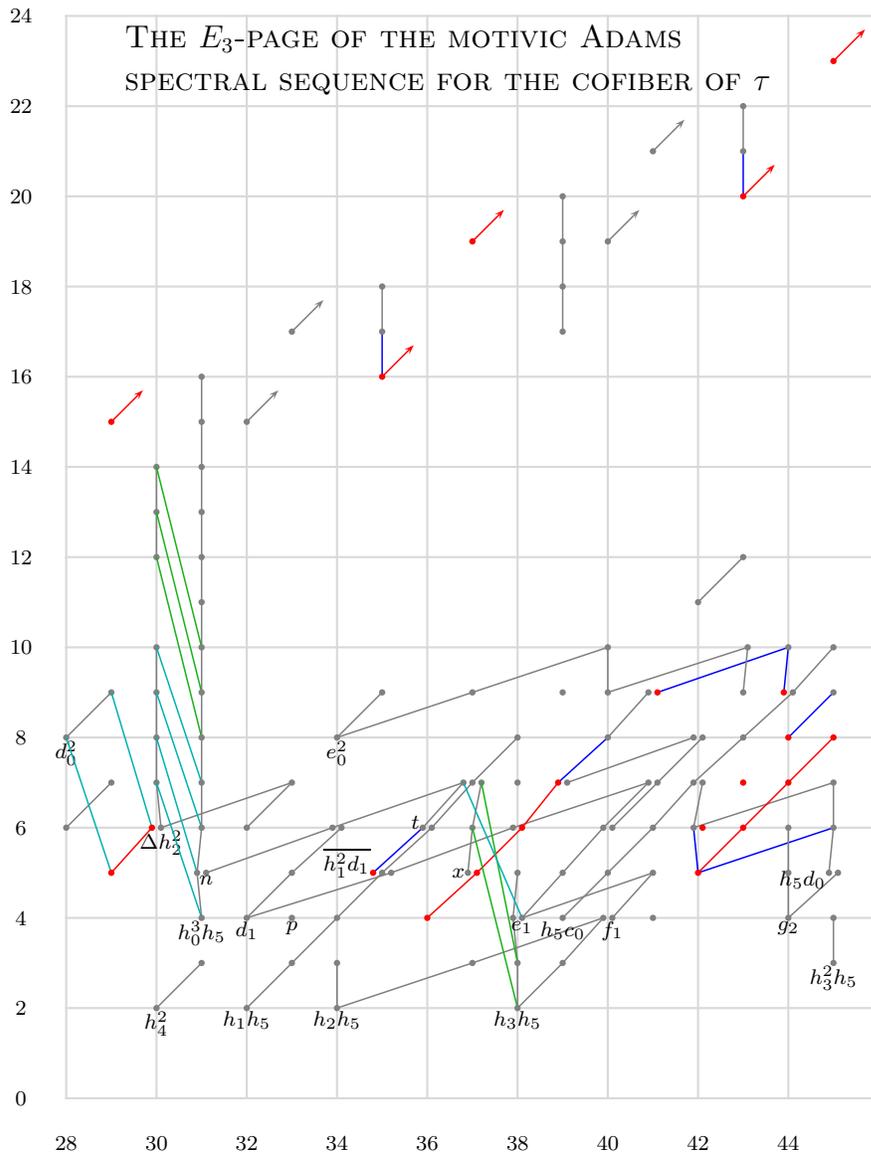
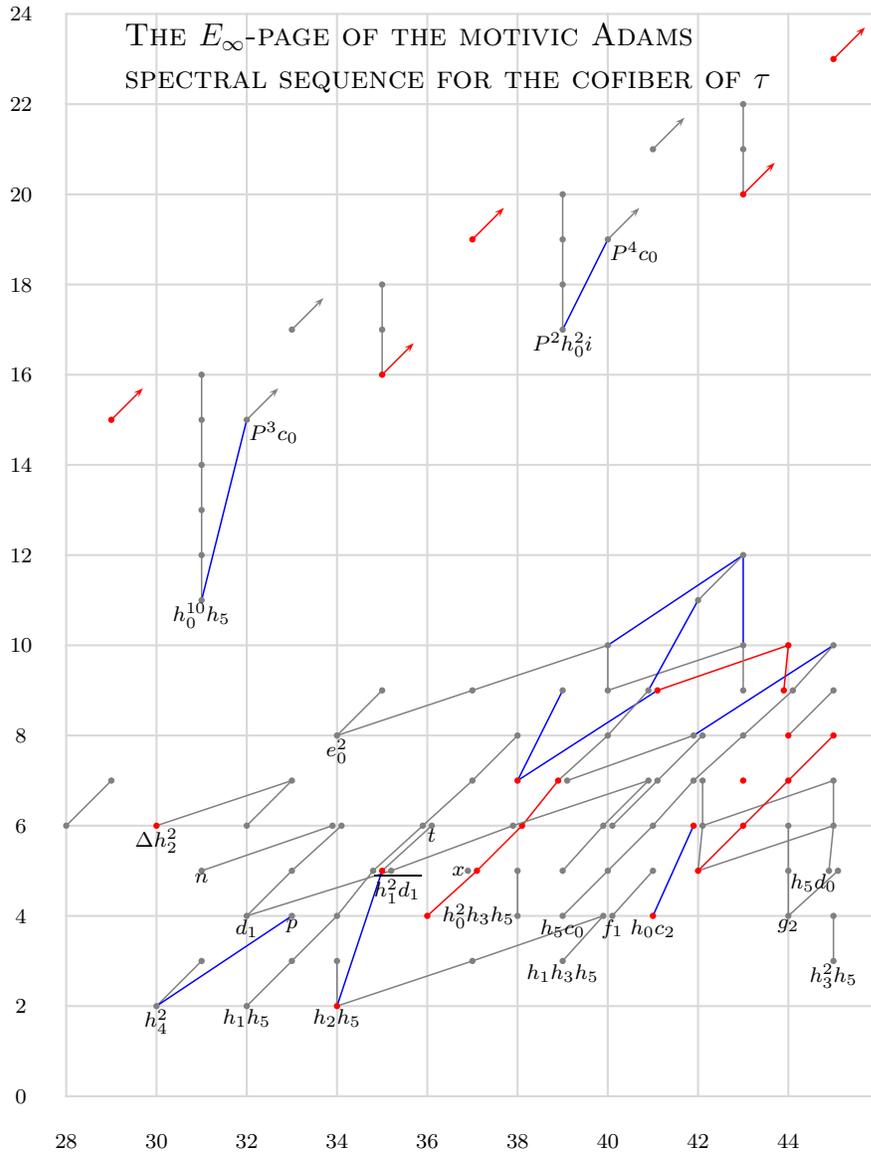
\begin{landscape}

\psset{unit=0.6cm}
\psset{linewidth=0.2mm}




\begin{thebibliography}{99}

\bibitem{AM16}
Michael Andrews and Haynes Miller.
Inverting the Hopf map.
J. Topol. 10 (2017), no. 4, 1145-1168.

\bibitem{ABG}
B. Antieau, T. Barthel, and D. Gepner, On localization sequences in the algebraic K-theory of ring spectra, Journal of the European Mathematical Society 20 (2018), no. 2, 459-487.

\bibitem{AGH}
B. Antieau, D. Gepner, and J. Heller. 
K-theoretic obstructions to bounded t-structures. 
Inventiones Math. 216 (2019), no. 1, 241-300.

\bibitem{BMT}
M. G. Barratt, M. E. Mahowald, and M. C. Tangora. Some differentials in the Adams spectral sequence. II. Topology 9 (1970), 309-316.

\bibitem{Barwick12}
Clark Barwick.
On exact $\infty$-categories and the Theorem of the Heart.
Compositio Mathematica 151 (2015), no. 11, 2160-2186.

\bibitem{BGT}
Andrew J. Blumberg, David Gepner, and Goncalo Tabuada, A universal characterization of higher algebraic K-theory, Geom. Topol. 17 (2013), no. 2, 733-838.

\bibitem{Boardman}
J. Michael Boardman, Conditionally convergent spectral sequences, Homotopy invariant algebraic structures (Baltimore, MD, 1998), 1999, pp. 49-84. 


\bibitem{BHM}
M. B\"{o}kstedt and W. C. Hsiang and I. Madsen.
The cyclotomic trace and algebraic K-theory of spaces, Invent. Math. 111 (1993), no. 3, 465-539.

\bibitem{Br1}
Robert Bruner, A new differential in the Adams spectral sequence, Topology 23 (1984), no. 3, 271-276. 

\bibitem{Br2}
R. R. Bruner, J. P. May, J. E. McClure, and M. Steinberger.
{$H_\infty $} ring spectra and their applications.
Lecture Notes in Mathematics, vol. 1176, Springer-Verlag, Berlin, 1986.

\bibitem{Chris}
J. Daniel Christensen, Ideals in triangulated categories: phantoms, ghosts and skeleta, Adv. Math. 136 (1998), no. 2, 284-339. 

\bibitem{DG}
Pierre Deligne and Alexander B. Goncharov, Groupes fondamentaux motiviques de Tate mixte, Ann. Sci. Ecole Norm. Sup. (4) 38 (2005), no. 1, 1-56.

\bibitem{DRO} B. I. Dundas, O. R\"ondings, P. A. \O stv\ae r. 
Motivic functors. Doc. Math. 8 (2003), 489-525.

\bibitem{DuggerIsaksenMASS}
Daniel Dugger and Daniel C. Isaksen.
The motivic Adams spectral sequence, Geom. Topol. 14 (2010), no. 2, 967-1014.

\bibitem{DuggerIsaksen}
Daniel Dugger and Daniel C. Isaksen.
Motivic cell structures, Algebr. Geom. Topol. 5 (2005), 615-652.

\bibitem{DI04}
Daniel Dugger and Daniel C. Isaksen.
Topological hypercovers and $\mathbb{A}^1$-realizations. Math. Z. 246 (2004), 667-689.

\bibitem{EM}
S. Eilenberg and J. C. Moore, Foundations of Relative Homological Algebra,
Mem. Amer. Math. Soc. 55, 1965.

\bibitem{GH}
Phillip Griffiths and Joseph Harris.
Principles of algebraic geometry. Reprint of the 1978 original. Wiley Classics Library. John Wiley \& Sons, Inc., New York, 1994. xiv+813 pp. ISBN: 0-471-05059-814-01

\bibitem{GM}
Sergei I. Gelfand and Yuri I. Manin, Methods of homological algebra, Second edition. Springer Monographs
in Mathematics, Springer-Verlag, Berlin, 2003.

\bibitem{GheCt}
Bogdan Gheorghe, The motivic cofiber of $\tau$. 
Doc. Math. 23, 1077-1127 (2018).

\bibitem{IsaP}
Bogdan Gheorghe and Daniel C. Isaksen, Private communication. 2016.

\bibitem{Go}
Paul Goerss, Quasi-coherent sheaves on the moduli stack of formal groups. 
arXiv:0208.0996.

\bibitem{Gre}
J. P. C. Greenlees.
Rational {$S^1$}-equivariant stable homotopy theory.
Mem. Amer. Math. Soc. 138
(1999), no. 661, xii+289.


\bibitem{Hop}
Michael J. Hopkins, Complex oriented cohomology theories and the language of stacks. available at
http://www.math.rochester.edu/u/faculty/doug/papers.html.

\bibitem{Hovey}
Mark Hovey, Homotopy theory of comodules over a Hopf algebroid, Homotopy theory: relations with
algebraic geometry, group cohomology, and algebraic K-theory, 2004, pp. 261-304. 

\bibitem{Hoyois}
Marc Hoyois, From algebraic cobordism to motivic cohomology, J. Reine Angew. Math. 702 (2015),
173-226.

\bibitem{Hu03}
Po Hu.
{$S$}-modules in the category of schemes.
Mem. Amer. Math. Soc. 161 (2003), no. 767, viii+125.

\bibitem{HKOrem}
Po Hu, Igor Kriz, and Kyle Ormsby, Remarks on motivic homotopy theory over algebraically closed fields, J. K-Theory 7 (2011), no. 1, 55-89.

\bibitem{HKOcon}
Po Hu, Igor Kriz, and Kyle Ormsby, Convergence of the Motivic Adams Spectral Sequence, J. K-Theory 7 (2011), 573-596.

\bibitem{IsaksenCharts}
Daniel C. Isaksen.
Classical and motivic Adams charts. arXiv:1401.4983v4.

\bibitem{IsaksenCharts2}
Daniel C. Isaksen.
Classical and motivic Adams-Novikov charts. arXiv:1408.0248v2.

\bibitem{StableStems}
Daniel C. Isaksen.
Stable Stems. To appear in Memoirs of American Mathematical Society, arXiv:1407.8418v2.


\bibitem{IWX}
Daniel C. Isaksen, Guozhen Wang, and Zhouli Xu, More Stable Stems. Preprint.

\bibitem{IX}
Daniel C. Isaksen and Zhouli Xu.
Motivic stable homotopy and the stable 51 and 52 stems, Topology
Appl. 190 (2015), 31-34. 

\bibitem{Kra}
Achim Krause, Periodicity in motivic homotopy theory and over $\textup{BP}_*\textup{BP}$. Bonn University. Thesis.

\bibitem{Landweber1}
Peter S. Landweber, Annihilator ideals and primitive elements in complex bordism, Illinois J. Math.
17 (1973), 273-284.

\bibitem{Landweber2}
Peter S. Landweber.
Associated prime ideals and Hopf algebras, J. Pure Appl. Algebra 3 (1973), 43-58.

\bibitem{Law}
Tyler Lawson, Secondary power operations and the Brown-Peterson spectrum at the prime 2, Ann.
Math. (2) 188 (2018), no. 2, 513-576 (English).

\bibitem{ML}
M. Levine and F. Morel, Algebraic cobordism, Springer Monographs in Mathematics, Springer,
Berlin, 2007.

\bibitem{HA}
Jacob Lurie, Higher algebra. available at http://www.math.harvard.edu/~lurie/papers/HA.pdf.


\bibitem{HTT}
Jacob Lurie, Higher topos theory, Annals of Mathematics Studies, vol. 170, Princeton University Press,
Princeton, NJ, 2009.

\bibitem{Mah}
Mark Mahowald, The metastable homotopy of $S^n$, Memoirs of the American Mathematical Society,
No. 72, American Mathematical Society, Providence, R.I., 1967.

\bibitem{MahowaldTangora}
Mark Mahowald and Martin Tangora, Some differentials in the Adams spectral sequence, Topology
6 (1967), 349-369.

\bibitem{May2}
J. Peter May, The cohomology of restricted Lie algebras and of Hopf algebras: Application to the Steenrod algebra, ProQuest LLC, Ann Arbor, MI,
1964. Thesis (Ph.D.)-Princeton University.

\bibitem{May}
J. Peter May.
Matric Massey products, J. Algebra 12 (1969), 533-568.

\bibitem{MillerAdams}
Haynes R. Miller, On relations between Adams spectral sequences, with an application to the stable
homotopy of a Moore space, J. Pure Appl. Algebra 20 (1981), no. 3, 287-312.

\bibitem{MillerThesis}
Haynes Robert Miller, Some algebraic aspects of the Adams-Novikov spectral sequence, ProQuest LLC, Ann Arbor, MI, 1975. Thesis (Ph.D.)-Princeton University.

\bibitem{MillerNotes}
Haynes Miller. The Adams spectral sequence: Course notes. available at: http://math.mit.edu/~hrm/18.917/notes.pdf

\bibitem{MillerRavenel}
Haynes R. Miller and Douglas C. Ravenel, Morava stabilizer algebras and the localization of
Novikov's $E_2$-term, Duke Math. J. 44 (1977), no. 2, 433-447.

\bibitem{MRW}
H. R. Miller, D. C. Ravenel and W. S. Wilson. 
Periodic phenomena in the Adams-Novikov spectral
sequence. Ann. of Math. 106 (1977) 469-516.


\bibitem{Mor}
Jack Morava, Noetherian localisations of categories of cobordism comodules. Ann. of Math. (2), 121 (1985), no. 1, 1-39.

\bibitem{MorelAdams}
Fabien Morel.
Suite spectrale d'{A}dams et invariants cohomologiques des
 formes quadratiques. C. R.
Acad. Sci. Paris Sr. I Math. 328 (1999), no. 11, 963-968.
 
\bibitem{MorelA1}
Fabien Morel.
The stable {${\Bbb A}^1$}-connectivity theorems.
K-Theory 35 (2005), no. 1-2, 1-68.

\bibitem{Morelpi0}
Fabien Morel.
On the motivic stable $\pi_0$ of the sphere spectrum. Axiomatic, Enriched and Motivic Homotopy Theory, pp. 219-260, 2004 Kluwer Academic Publishers.

\bibitem{MorelVoevodsky}
Fabien Morel and Vladimir Voevodsky.
{${\bf A}^1$}-homotopy theory of schemes.
Inst. Hautes \'Etudes Sci. Publ. Math. 90 (1999), 45-143 (2001).

\bibitem{Novikov}
S. P. Novikov.
Methods of algebraic topology from the point of view of cobordism theory, Izv. Akad.
Nauk SSSR Ser. Mat. 31 (1967), 855-951.

\bibitem{VoeMC2}
D. Orlov, A. Vishik, and V. Voevodsky, An exact sequence for {$K^M_\ast/2$} with applications to quadratic forms. Ann. of Math. (2) 165 (2007), no. 1, 1-13.

\bibitem{PPR2}
Ivan Panin, Konstantin Pimenov, and Oliver R\"ondigs, A universality theorem for Voevodsky's algebraic cobordism spectrum, Homology Homotopy Appl. 10 (2008), no. 2, 211-226.

\bibitem{Pat}
Irakli Patchkoria, The derived category of complex periodic K-theory localized at an odd prime, Adv. Math. 309 (2017), 392-435. 

\bibitem{Pst}
Piotr Pstr\text{\c a}gowski.
Synthetic spectra and the cellular motivic category, arXiv:1803.01804.


\bibitem{Qui}
Daniel Quillen, On the formal group laws of unoriented and complex cobordism theory, Bull. Amer. Math. Soc. 75 (1969), 1293-1298.

\bibitem{Ravenel}
Douglas C. Ravenel, Complex cobordism and stable homotopy groups of spheres, Pure and Applied Mathematics, vol. 121, Academic Press, Inc., Orlando, FL, 1986. 

\bibitem{Robalo}
Marco Robalo.
K-theory and the bridge from motives to noncommutative motives. Adv. Math. 269 (2015), 399-550. 

\bibitem{Schwede}
Stefan Schwede, Algebraic versus topological triangulated categories, Triangulated categories, 2010, pp. 389-407. 

\bibitem{SS}
Stefan Schwede and Brooke Shipley, Stable model categories are categories of modules, Topology 42 (2003), no. 1, 103-153.

\bibitem{Sega}
Liana M. $\text{\c S}$ega.
Homological properties of powers of the maximal ideal of a local ring, J. Algebra 241 (2001), no. 2, 827-858.

\bibitem{VoeBK}
Andrei Suslin and Vladimir Voevodsky, Bloch-Kato conjecture and motivic cohomology with finite coefficients, The arithmetic and geometry of algebraic cycles (Banff, AB, 1998), 2000, pp. 117-189.



\bibitem{Todacomp}
Hirosi Toda, Composition methods in homotopy groups of spheres, Annals of Mathematics Studies, No. 49, Princeton University Press, Princeton, N.J., 1962. 


\bibitem{VoeEM}
Vladimir Voevodsky.
Motivic Eilenberg-Maclane spaces, Publ. Math. Inst. Hautes \'Etudes Sci. 112 (2010), 1-99.

\bibitem{Voered}
Vladimir Voevodsky.
Reduced power operations in motivic cohomology, Publ. Math. Inst. Hautes \'Etudes Sci. 98 (2003), 1-57.


\bibitem{VoeZ2}
Vladimir Voevodsky.
Motivic cohomology with {${\bf Z}/2$}-coefficients.
Publ. Math. Inst. Hautes \'Etudes Sci.  98 (2003), 59-104.

\bibitem{VoeMC1}
Vladimir Voevodsky.
The Milnor conjecture, Preprint.


\bibitem{Voe}
Vladimir Voevodsky.
{$\bold A^1$}-homotopy theory.
Proceedings of the International Congress of Mathematicians, Vol. I (Berlin, 1998), 1998, pp. 579-604 (electronic).

\bibitem{WX}
Guozhen Wang and Zhouli Xu.
The triviality of the 61-stem in the stable homotopy groups of spheres., Ann. Math. (2) 186 (2017), no. 2, 501-580 (English).

\bibitem{WX2}
Guozhen Wang and Zhouli Xu.
Some extensions in the Adams spectral sequence and the 51-stem, Algebraic and Geometric Topology. 18 (2018), 3887-3906.

\bibitem{Wen}
Matthias Wendt. 
More examples of motivic cell structures. 
arXiv:1012.0454.

\bibitem{Xu1}
Zhouli Xu.
The strong Kervaire invariant problem in dimension 62, Geom. Topol. 20 (2016), no. 3, 1611-1624.

\bibitem{Xu2}
Zhouli Xu.
Mahowald square and Adams differentials, Contemporary Mathematics. Homotopy Theory: Tools and Applications, vol. 729, Amer. Math. Soc., Providence, RI, 2019, pp. 255-268.




	
\end{thebibliography}

\end{document}